\documentclass[10pt, oneside]{article}

\usepackage{fancyhdr}
\fancyheadoffset{ 0.5 cm}
\fancyfootoffset{ 0.5 cm}

\usepackage{titlesec}
\titlelabel{\thetitle.\quad}

\usepackage{times}
\usepackage{geometry}
\usepackage{graphicx}
\usepackage{amssymb, amscd}
\usepackage{amsmath}
\makeatletter
\@addtoreset{equation}{section}
\makeatother
\numberwithin{equation}{section}
\usepackage{amsthm} 
\usepackage{mathrsfs}
\usepackage{epstopdf}
\usepackage{enumerate}
\usepackage{url}
\usepackage{graphicx}
\usepackage{mathtools}
\usepackage{tikz}							
\usepackage{setspace}
\usepackage{amsfonts, amsmath, wasysym}

\usepackage{appendix}

\usepackage{relsize}

\usepackage{shuffle}

\usepackage{float}
\restylefloat{table}

\usepackage{hyperref}
\usepackage{cite}

\usepackage{bbm}

\usepackage[draft]{optional}

\usepackage{color}

\usepackage{xcolor}
\newtagform{red}{\color{red}(}{)}

\usepackage{enumitem}

\usepackage{mathtools}

\usepackage{amsmath}
\makeatletter
\@addtoreset{equation}{section}
\makeatother
\numberwithin{equation}{section}

\theoremstyle{plain}
\newtheorem{theorem}{Theorem}[section]
\newtheorem{lemma}[theorem]{Lemma}

\theoremstyle{definition}
\newtheorem{definition}[theorem]{Definition}

\theoremstyle{remark}

\newtheorem{remark}[theorem]{\bf{Remark}}

\def\DataSig{DataS{\i}g }

\newcommand{\cO}{\ensuremath{{\cal O}}}

\newcommand{\m}{\ensuremath{{\cal M}}}
\newcommand{\n}{\ensuremath{{\cal N}}}

\newcommand{\ca}{\ensuremath{{\cal A}}}

\newcommand{\cf}{\ensuremath{{\cal F}}}

\newcommand{\cl}{\ensuremath{{\cal L}}}

\newcommand{\cv}{\ensuremath{{\cal V}}}

\newcommand{\ti}{\tilde}

\newcommand{\al}{\alpha}
\newcommand{\be}{\beta}

\newcommand{\de}{\delta}

\newcommand{\Tau}{{\cal T}}  
\renewcommand{\th}{\theta}

\newcommand{\vph}{\varphi}
\newcommand{\ep}{\varepsilon}

\newcommand{\R}{\ensuremath{{\mathbb R}}}

\newcommand{\B}{\ensuremath{{\mathbb B}}}

\newcommand{\Z}{\ensuremath{{\mathbb Z}}}

\newcommand{\ovB}{\ensuremath{{ \overline{ \mathbb B }}}}

\newcommand{\bA}{{\bf A }}
\newcommand{\bB}{{\bf B}}

\newcommand{\bE}{{\bf E}}

\newcommand{\bI}{{\bf I}}
\newcommand{\bII}{{\bf II}}

\DeclareMathOperator{\inj}{inj}

\DeclareMathOperator{\dist}{dist}

\DeclareMathOperator{\Span}{Span}

\DeclareMathOperator{\proj}{proj}
\DeclareMathOperator{\ovSpan}{\overline{Span}}
\DeclareMathOperator{\Reach}{Reach}

\DeclareMathOperator{\argmax}{argmax}
\DeclareMathOperator{\argmin}{argmin}

\DeclareMathOperator{\ord}{ord}

\DeclareMathOperator{\pack}{pack}
\DeclareMathOperator{\cov}{cov}

\newcommand{\support}{{\rm support}}

\newcommand{\beq}{\begin{equation}}
\newcommand{\eeq}{\end{equation}}
\newcommand{\beqa}{\begin{equation}\begin{aligned}}
\newcommand{\eeqa}{\end{aligned}\end{equation}}
\newcommand{\brmk}{\begin{rmk}}
\newcommand{\ermk}{\end{rmk}}
\newcommand{\partref}[1]{\hbox{(\csname @roman\endcsname{\ref{#1}})}}

\newcommand{\Lip}{{\mathrm{Lip}}}

\usepackage{soul}

\newcommand{\twopartdef}[4]
{
	\left\{
		\begin{array}{ll}
			#1 & \mbox{if } #2 \\ 
			#3 & \mbox{if } #4
		\end{array}
	\right.
}

\makeatletter
\def\dual#1{\expandafter\dual@aux#1\@nil}
\def\dual@aux#1/#2\@nil{\begin{tabular}{@{}c@{}}#1\\#2\end{tabular}}
\makeatother

\makeatletter
\def\three#1{\expandafter\three@aux#1\@nil}
\def\three@aux#1/#2/#3\@nil{\begin{tabular}{@{}c@{}c@{}}#1\\#2\\#3\end{tabular}}
\makeatother

\makeatletter
\def\four#1{\expandafter\four@aux#1\@nil}
\def\four@aux#1/#2/#3/#4\@nil{\begin{tabular}{@{}c@{}c@{}c@{}}#1\\#2\\#3\\#4\end{tabular}}
\makeatother

\title{Higher Order Lipschitz Greedy Recombination Interpolation 
Method (HOLGRIM)}
\author{Terry Lyons and Andrew D. McLeod}
\date{\today}

\begin{document}
\usetagform{red}
\maketitle

\begin{abstract}
In this paper we introduce the 
\emph{Higher Order Lipschitz Greedy Recombination Interpolation 
Method} (HOLGRIM) for finding sparse approximations of 
$\Lip(\gamma)$ functions, in the sense of Stein, given as a linear
combination of a (large) number of simpler $\Lip(\gamma)$ functions.
HOLGRIM is developed as a refinement of the 
\emph{Greedy Recombination Interpolation Method} (GRIM) in 
the setting of $\Lip(\gamma)$ functions.
HOLGRIM combines dynamic growth-based interpolation techniques 
with thinning-based reduction techniques in a data-driven fashion.
The dynamic growth is driven by a greedy selection algorithm in 
which multiple new points may be selected at each step.
The thinning reduction is carried out by \emph{recombination}, 
the linear algebra technique utilised by GRIM.
We establish that the number of non-zero weights for the 
approximation returned by HOLGRIM is controlled by a particular
packing number of the data. The level of data concentration 
required to guarantee that HOLGRIM returns a good sparse 
approximation is decreasing with respect to the regularity parameter
$\gamma > 0$.
Further, we establish complexity cost estimates verifying that
implementing HOLGRIM is feasible.
\end{abstract}

{\small \tableofcontents}

\section{Introduction}
\label{intro}
Approximating the behaviour of a complex system via a 
linear combination of simpler functions is a central challenge
within machine learning. Models following this guiding ethos
have, for example, been used 
for facial recognition \cite{BGLT97}, 
COVID-19 deaths prediction \cite{BS22}, 
image recognition \cite{HRSZ15,HKR15,KTT18},
language translation \cite{CDLT18},
anomaly detection \cite{CCFLS20,Arr24}, 
landmark-based action recognition \cite{JLNSY17,JNYZ24}, 
emotion recognition \cite{LLNNSW19}, 
early sepsis detection \cite{HKLMNS19,HKLMNS20}, 
Bipolar and Borderline Personality Disorder classification 
\cite{AGGLS18,LLNSTWW20,LLSVWW21}, 
longitudinal language modelling \cite{BCKLLT24}, 
and learning solutions to differential equations 
\cite{FKLM20,FKLMS21,FKLLO21,FKLL21,CLLMQW24}.

A priori, considering a larger number of simple functions 
enables one to capture the affects of more complicated systems.
However, a large number of simple functions can lead to a 
linear combination with a high computational complexity.
One approach to reducing this computational complexity is to 
find a sparse approximation of the linear combination of 
functions.
Sparse representations have been applied to tasks including, 
for example, image processing \cite{EMS08,BMPSZ08,BMPSZ09},
data assimilation \cite{MM13}, 
sensor placement in nuclear reactors \cite{ABGMM16,ABCGMM18},
DNA denoising \cite{KK20}, and inference within machine 
learning \cite{ABDHP21,NPS22}.

Numerous techniques have been developed for finding sparse 
approximations, including 
\emph{Least Absolute Shrinkage and Selection Operator} (LASSO) 
regression 
\cite{CM73,SS86,Tib96,LY07,DGOY08,GO09,XZ16,TW19}, the
Empirical Interpolation Method \cite{BMNP04, GMNP07, MNPP09}, 
its subsequent generalisation the Generalised Empirical 
Interpolation Method (GEIM) \cite{MM13, MMT14}, 
Pruning \cite{Ree93,AK13,CHXZ20,GLSWZ21}, 
Kernel Herding \cite{Wel09a,Wel09b,CSW10,BLL15,BCGMO18,TT21,PTT22},
Convex Kernel Quadrature \cite{HLO21}, and 
Kernel Thinning \cite{DM21a,DM21b,DMS21}.
In this paper we focus on the 
\emph{Greedy Recombination Interpolation Method} (GRIM) 
introduced in \cite{LM22}.

GRIM is a technique for finding sparse approximations of 
linear combinations of functions that 
combines dynamic growth-based interpolation techniques
with thinning-based reduction techniques to inductively 
refine a sequence of approximations. 
Theoretical guarantees for the performance of GRIM are 
established in Section 6 of \cite{LM22}, whilst empirical 
demonstrations of its performance are provided by the numerical 
examples covered in Section 7 of \cite{LM22}.
It is established in \cite{LM22} that GRIM can be applied in 
several familiar settings including kernel quadrature 
and approximating sums of continuous functions; see Section 2 
in \cite{LM22}.

For the readers convenience, we briefly recall the 
general sparse approximation problem that GRIM is designed 
to tackle \cite{LM22}.
Let $X$ be a real Banach space and 
$\n \in \Z_{\geq 1}$ be a (large) positive integer. Assume that 
$\cf = \{ f_1 , \ldots , f_{\n} \} \subset X$ is a collection
of non-zero elements, and that $a_1 , \ldots , a_{\n} \in \R 
\setminus \{0\}$. Consider the element $\vph \in X$ defined 
by $\vph := \sum_{i=1}^{\n} a_i f_i$. 
Let $X^{\ast}$ denote the dual of $X$ and suppose that
$\Sigma \subset X^{\ast}$ is a finite subset with
cardinality $\Lambda \in \Z_{\geq 1}$. 
We adopt the same terminology as used in \cite{LM22} that the 
set $\cf$ consists of the
\emph{features} whilst the set $\Sigma$ consists of \emph{data}.
Then GRIM is designed to tackle the following sparse 
approximation problem.
Given $\ep > 0$, find an element 
$u = \sum_{i=1}^{\n} b_i f_i\in \Span(\cf) \subset X$ 
such that the cardinality of the set 
$\{ i \in \{1, \ldots , \n \} : b_i \neq 0 \}$ is 
\emph{less} than $\n$ and that $u$ is close to $\vph$ 
throughout $\Sigma$ in the sense that, for every 
$\sigma \in \Sigma$, we have $|\sigma(\vph-u)| \leq \ep$.

We briefly recall the strategy of the \textbf{Banach GRIM} 
algorithm developed in \cite{LM22} to tackle this problem.
The \textbf{Banach GRIM} algorithm is a hybrid combination of 
the dynamic growth of a greedy selection algorithm, in a 
similar spirit to
GEIM \cite{MM13, MMT14, MMPY15}, with the thinning reduction 
of recombination that underpins the successful
convex kernel quadrature approach of \cite{HLO21}. 
A collection of linear functionals $L \subset \Sigma$ is
greedily grown during the \textbf{Banach GRIM} algorithm.
After each extension of $L$, recombination is applied to 
find an approximation of $\vph$ that coincides with $\vph$ 
throughout $L$ (cf. the \emph{Recombination Thinning Lemma 3.1} 
in \cite{LM22}).
The subset $L$ is extended by first fixing an integer 
$m \in \Z_{\geq 1}$, and then adding to $L$ the $m$ linear 
functionals from $\Sigma$ achieving the largest absolute value 
when applied to the difference between 
$\vph$ and the current approximation of $\vph$ (cf. Section 4 of 
\cite{LM22}).
These steps are iteratively applied a set number of times 
during the \textbf{Banach GRIM} algorithm detailed in Section 4 
of \cite{LM22}.

In this article we consider the use of the 
\textbf{Banach GRIM} algorithm for finding 
sparse approximations of sums of continuous functions.
That is, we assume that the Banach space $X$ in the formulation 
above can be embedded into a space of continuous functions.
We assume that $X$ embeds into a space of continuous
functions rather than assuming $X$ is itself a space of continuous 
functions since continuity alone is insufficient regularity for 
interpolation methods.
Interpolation methods, broadly speaking, work on the premise that 
knowing the value of a function at a point informs one about the 
functions values at nearby points.

However, this is \emph{not} true for continuous functions.
Indeed, suppose that $\Omega \subset \R^d$ and $f \in C^0(\Omega)$ 
is a real-valued continuous function $\Omega \to \R$.
Assume that $p \in \Omega$ and $f(p)=0$.
Then, regardless of how close a point $q \in \Omega \setminus \{p\}$
is to $p$, we cannot obtain any better bound for the value 
$f(q)$ than the immediate naive estimate that 
$|f(q)| \leq ||f||_{C^0(\Omega)}$. Knowing that $f$ vanishes at 
$p$ does not enable us to conclude that $f$ must remain small on 
some definite neighbourhood of the point $p$.

A consequence of this simple observation is that the 
theoretical guarantees for the performance of the 
\textbf{Banach GRIM} algorithm 
established in Theorem 6.2 in \cite{LM22} are essentially 
useless for the case that $X = C^0(\Omega)$ and $\Sigma$
is taken to be the finite collection of point masses in 
the dual-space $C^0(\Omega)^{\ast}$ associated to a finite 
subset of $\Omega$.
Given $p,q \in \Omega$ with $p \neq q$, the existence of a 
continuous function $f \in C^0(\Omega)$ with 
$||f||_{C^0(\Omega)} = 1$ and $f(p)=1$ and $f(q)=-1$ means that
the point masses $\de_p , \de_q \in C^0(\Omega)^{\ast}$ satisfy
\beq
    \label{eq:intro_point_mass_sep}
        \left|\left| \de_p - \de_q 
        \right|\right|_{C^0(\Omega)^{\ast}} 
        = \sup \left\{ \left| \left( \de_p - \de_q \right)(f) 
        \right| ~:~ f \in C^0(\Omega) 
        \text{ with } ||f||_{C^0(\Omega)} = 1 \right\} = 2.
\eeq
Consequently, if $\ep/2C < 2$ then \eqref{eq:intro_point_mass_sep}
means that the upper bound $N$ on the maximum number of steps 
the \textbf{Banach GRIM} algorithm 
can complete without terminating established 
in Theorem 6.2 in \cite{LM22} is simply the cardinality 
of $\Sigma$, i.e. $N = \#(\Sigma)$.
Thus Theorem 6.2 in \cite{LM22} provides only the 
obvious (and essentially useless) guarantee that the 
\textbf{Banach GRIM} 
algorithm will return an approximation that is within 
$\ep$ of $\vph$ at every $\sigma \in \Sigma$ once we have grown 
the collection of linear functionals $L \subset \Sigma$ at which we 
require the approximation to match $\vph$ 
(see Section 4 in \cite{LM22}) to be the entirety of $\Sigma$.

In this article we overcome the limitations outlined above by 
achieving the following main goals.
\vskip 4pt
\noindent
\textbf{MAIN GOALS}
\begin{enumerate}
    \item[(1)]\label{paper_goal_1} 
    Fix a choice of a Banach space $X$ of functions
    exhibiting an appropriate level of regularity from the 
    perspective of interpolation.
    \item[(2)]\label{paper_goal_2}
    Adapt the \textbf{Banach GRIM} algorithm from \cite{LM22}
    to the setting of the particular choice of the Banach space
    $X$ determined in \textbf{MAIN GOALS} (\ref{paper_goal_1}).
\end{enumerate}
\vskip 4pt
\noindent
From the perspective of interpolation, we want the notion of 
regularity determined by our choice of Banach space $X$ to 
ensure, in a quantified sense, that for
$f \in X$ the knowledge of the value of $f$ 
at a point determines its values at nearby points up to 
an arbitrarily small error.
Moreover, given our machine learning motivations, we want 
the notion of regularity determined by our choice of Banach 
space $X$ to make sense on finite subsets.

The class of $\Lip(\gamma)$ functions, for some 
$\gamma > 0$, in the sense of Stein \cite{Ste70}
exhibit an appropriate level of regularity.
This notion of regularity is central to the study of 
\emph{rough paths} initiated by the first author in 
\cite{Lyo98}; an introduction to the theory of rough paths 
may be found in \cite{CLL04}, for example. 
Rough path theory has subsequently underpinned numerous 
machine learning techniques developed through the use of 
the \emph{signature} of a path as a feature map; an introductory 
coverage of such signature methods may be found in the survey 
articles \cite{CK16} and \cite{LM22B}, for example.
Moreover, $\Lip(\gamma)$ regularity is essential in the 
recent efforts to extend the theory of rough paths to the setting 
of manifolds \cite{CLL12,BL22}, 
has been combined with classical ODE techniques to obtain 
the terminal solutions of \emph{Rough Differential Equations} (RDEs) 
\cite{Bou15,Bou22}, 
and is central to the introduction of 
\emph{Log Neural Controlled Differential Equations} (Log-NCDEs) 
in \cite{CLLMQW24}.

We consider the notion of $\Lip(\gamma)$ regularity in 
similar generality to that considered in \cite{LM24}.
In particular, we let $V$ and $W$ be finite-dimensional 
real Banach spaces, 
$\m \subset V$ be an arbitrary closed subset, 
$\gamma > 0$, and consider
$X := \Lip(\gamma,\m,W)$ to be the space of $\Lip(\gamma)$ 
functions from $\m$ to $W$ defined, for example, 
in Definition 2.2 of \cite{LM24} 
(cf. Definition \ref{lip_k_def} in this paper).
The class $\Lip(\gamma,\m,W)$ is a generalisation of the 
more familiar notion of $\gamma$-H\"{o}lder continuous 
functions $\m \to W$ in the following sense. 
For $\gamma \in (0,1]$ the space $\Lip(\gamma,\m,W)$
coincides with the space of bounded $\gamma$-H\"{o}lder continuous 
functions $\m \to W$.
That is, $\psi \in \Lip(\gamma,\m,W)$ means
that $\psi : \m \to W$ and there exists a 
constant $C > 0$ such that whenever $x,y \in \m$ 
we have both $||\psi(x)||_W \leq C$ and
$||\psi(y) - \psi(x)||_W \leq C ||y-x||_V^{\gamma}$.
For $\gamma > 1$ the space $\Lip(\gamma,\m,W)$
gives a sensible higher-order H\"{o}lder continuity condition. 
If $k \in \Z_{\geq 0}$ such that $\gamma \in (k,k+1]$ and 
$O \subset V$ is open, then the space 
$\Lip(\gamma,O,W)$ coincides with the space 
$C^{k,\gamma - k}(O;W)$ of $k$-times Fr\'{e}chet differentiable 
functions $O \to W$ with $(\gamma - k)$-H\"{o}lder continuous 
$k^{\text{th}}$ derivative
(cf. Remark \ref{rmk:lip_gamma_on_open_sets}).
Stein's notion of $\Lip(\gamma)$ regularity extends 
this notion to arbitrary closed subsets.
For any $\gamma > 0$ and a general closed subset 
$\m \subset V$ there is a natural embedding map 
$\Lip(\gamma,\m,W) \hookrightarrow C^0(\m;W)$
(cf. Remark \ref{rmk:Lip_gamma_subset_C0}).

The choice $X := \Lip(\gamma,\m,W)$ satisfies both our 
desired properties. 
Firstly, the class $\Lip(\gamma,\m,W)$ is well-defined for 
\emph{any} closed subset $\m \subset V$, and hence the class
$\Lip(\gamma,\Sigma,W)$ is well-defined for a finite subset 
$\Sigma \subset V$.
Secondly, the \emph{Lipschitz Sandwich Theorems} 
established in \cite{LM24} yield that the value of 
$\psi \in \Lip(\gamma,\m,W)$ at a point $p \in \m$ 
determines the behaviour of $\psi$, up to arbitrarily small 
errors, at points that are sufficiently close to $p$.
The \emph{Lipschitz Sandwich Theorem 3.1} and the 
\emph{Pointwise Lipschitz Sandwich Theorem 3.11} in \cite{LM24} 
exhibit quantified statements of this phenomena in which both 
the sense in which the behaviour of $\psi$ is determined and 
the notion of sufficiently close to the point $p$ are made precise
(cf. Theorems \ref{thm:lip_sand_thm} and 
\ref{thm:pointwise_lip_sand_thm} in this paper).

There are additional properties that make the class 
$\Lip(\gamma)$ functions an attractive choice from a learning 
perspective.
Given a closed subset $\m \subset V$, a function in 
$\Lip(\gamma,\m,W)$ is defined only through its values 
at the points in $\m$. No knowledge of its behaviour outside
$\m$ is required; we do not 
need to even define their value at \emph{any} point
in the complement $V \setminus \m \subset V$ of $\m$.

Further, any function in $\Lip(\gamma,\m,W)$ admits an extension 
to the entirety of $V$ possessing the same regularity. 
That is, if $\psi \in \Lip(\gamma,\m,W)$ then there exists 
$\Psi \in \Lip(\gamma,V,W)$ with $\Psi \equiv \psi$ throughout $\m$.
When $V = \R^d$ and $W = \R$ this is a consequence of the 
Stein-Whitney extension theorem presented as Theorem 4 
in Chapter VI of \cite{Ste70}. 
Often referred to as the Stein extension theorem, we include 
Whitney to reflect the reliance of Stein's proof on the 
machinery introduced by Whitney in his own extension theorems in 
\cite{Whi34,Whi44}.
Since $V$ is finite dimensional, a verbatim repetition of the 
argument presented in \cite{Ste70} establishes the extension 
result we claim above. 
Modulo defining a $\Lip(\gamma)$-norm (cf. Definition 
\ref{lip_k_def}),
it follows that the resulting extension operator 
$\Lip(\gamma,\m,W) \to \Lip(\gamma,V,W)$ is a bounded linear 
operator.

This extension property ensures that $\Lip(\gamma)$ functions 
are well-suited for inference on unseen data. 
That is, suppose we know that $\psi \in \Lip(\gamma,\Sigma,W)$
well-approximates a system on a finite subset $\Sigma \subset V$.
Then, given an unseen point $p \in V \setminus \Sigma$, 
we can extend $\psi$ to $\Psi \in \Lip(\gamma,V,W)$ and use the 
evaluation of the extension $\Psi$ at $p$ for the purpose of 
inferring the systems response to the unseen input $p$.
Moreover, if we know that the underlying system satisfies some 
global $\Lip(\gamma)$ regularity, then we can quantify how well 
$\Psi$ approximates the system at $p$ in terms of how well 
$\psi$ approximates the system on $\Sigma$ and the $V$-distance 
from the point $p$ to the subset $\Sigma$
(cf. Remark \ref{rmk:HOLGRIM_conv_thm_3}).

With our choice of $\Lip(\gamma)$ regularity fixed, 
we turn our attention to considering the 
\textbf{Banach GRIM} algorithm from \cite{LM22} in the 
case that $X = \Lip(\gamma,\Sigma,W)$ for a finite subset 
$\Sigma \subset V$.
For this purpose we develop the 
\emph{Higher Order Lipschitz Greedy Recombination Interpolation 
Method} (HOLGRIM) as a modification of the GRIM tailored to the 
setting that $X := \Lip(\gamma,\Sigma,W)$. 
In particular, the \textbf{HOLGRIM} algorithm, detailed in Section 
\ref{sec:HOLGRIM_alg}, is designed to seek a sparse approximation 
of a given linear combination $\vph$ of $\Lip(\gamma,\Sigma,W)$ 
functions. 
For a fixed choice of $q \in \{0 , \ldots , k\}$,
The approximation is required to be close to $\vph$
throughout $\Sigma$ in an ``order $q$" pointwise sense; this is 
made precise in Section \ref{sec:HOLGRIM_alg}, and is analogous 
to requiring the derivatives up to order $q$ or the approximation 
to be close to the derivatives up to order $q$ of $\vph$ at every 
point in $\Sigma$.

The \textbf{HOLGRIM} algorithm dynamically grows a subset 
$P \subset \Sigma$ of the data at which we require approximations
of the target $\vph$ to coincide with $\vph$.
As in the case of the \textbf{Banach GRIM} algorithm in 
\cite{LM22}, after each extension of the subset $P$, the thinning
or recombination is used to find an approximation $u \in \Span(\cf)$
coinciding with $\vph$ throughout $P$.
This is achieved via an application of the 
\emph{Recombination Thinning Lemma 3.1} from \cite{LM22}
(cf. the \textbf{HOLGRIM Recombination Step} in Section 
\ref{sec:HOLGRIM_alg}).

The applicability of the 
\emph{Recombination Thinning Lemma 3.1} from \cite{LM22} to find 
such an approximation of $\vph$ is a consequence of the 
following correspondence established in Section 
\ref{sec:pointwise_via_lin_funcs}.
Given a point $p \in \Sigma$ there is a subset
$\Tau_{p,k} \subset \Lip(\gamma,\Sigma,W)^{\ast}$ of bounded 
linear functionals $\Lip(\gamma,\Sigma,W) \to \R$ such 
that the value of any $\psi \in \Lip(\gamma,\Sigma,W)$ 
at the point $p$ \emph{is}
determined by the set of real numbers
$\{ \sigma(\psi) : \sigma \in \Tau_{p,k} \} \subset \R$. 
The set $\Tau_{p,k}$ is defined carefully in Section 
\ref{sec:pointwise_via_lin_funcs}, and a quantified estimate 
regarding the sense in which the value set 
$\{ \sigma(\psi) : \sigma \in \Tau_{p,k} \} \subset \R$
determines the value of $\psi$ at $p$ is provided by Lemmas
\ref{lemma:lin_funcs_det_pointwise_value} and 
\ref{lemma:number_of_coeffs_for_point_value}.
This observation allows us to associate the finite collection 
of bounded linear functionals 
$\Sigma^{\ast}_k := \cup_{p \in \Sigma} \Tau_{p,k}$ to 
the finite subset of data $\Sigma \subset V$.

Both the dynamic growth of the subset $P \subset \Sigma$, 
governed by the \textbf{HOLGRIM Extension Step} detailed in 
Section \ref{sec:HOLGRIM_alg}, and the use of recombination 
to obtain an approximation coinciding with $\vph$ 
throughout $P$, detailed in the 
\textbf{HOLGRIM Recombination Step} in Section 
\ref{sec:HOLGRIM_alg},
are reliant on the correspondence established in Section 
\ref{sec:pointwise_via_lin_funcs}.
Indeed, the dynamic growth detailed in the 
\textbf{HOLGRIM Extension Step} in Section \ref{sec:HOLGRIM_alg} 
is done in a similar spirit to the \textbf{Modified Extension Step}
appearing in Section 7 of \cite{LM22}.

Heuristically, to grow the subset $P$ we find the linear functional 
$\sigma \in \Sigma^{\ast}_k$ returning the largest absolute value 
when applied to the difference between $\vph$ and the current 
approximation, and then extend $P$ by addition of the point 
$p \in \Sigma$ for which $\sigma \in \Tau_{p,k}$.
In the
\textbf{HOLGRIM Extension Step} in Section \ref{sec:HOLGRIM_alg}
we allow for the extension of $P$ by more than a single new 
point and we only consider a particular subset of the 
linear functionals $\Sigma_k^{\ast}$ determined by the 
strength of approximation we seek (cf. the
\textbf{HOLGRIM Extension Step} in Section \ref{sec:HOLGRIM_alg}).
The number of new points to be added at each step gives a 
parameter that may be optimised.

After each extension of the subset $P$ we apply the 
\emph{Recombination Thinning Lemma 3.1} from \cite{LM22} to the 
collection of linear functionals $L := \cup_{p \in P} \Tau_{p,k}$
to find a new approximation that coincides with $\vph$ 
throughout $P$; see the \textbf{HOLGRIM Recombination Step} in 
Section \ref{sec:HOLGRIM_alg} for full details.
As in \cite{LM22}, we optimise our use of recombination over 
multiple permutations of the ordering of the equations 
determining the linear system to which recombination is 
applied (cf. the \textbf{HOLGRIM Recombination Step} in 
Section \ref{sec:HOLGRIM_alg}).
The number of permutations to be considered at each step gives 
a parameter that may be optimised.

HOLGRIM inherits the same data-driven benefits enjoyed by GRIM.
The growth in HOLGRIM is data-driven 
rather than feature-driven. The extension of
the data to be interpolated with respect to in HOLGRIM does 
not involve making any choices of features from $\cf$. The
new information to be matched is determined by examining where in 
$\Sigma$ the current approximation is furthest from
the target $\vph$ (cf. Section \ref{sec:HOLGRIM_alg}).
Only a subset $P \subset \Sigma$ of data is dynamically grown; 
there is no corresponding subset $F \subset \cf$ of functions that 
is dynamically grown.
The functions that an approximation will be a linear combination of 
are \emph{not} predetermined; they are determined by recombination
(cf. the \textbf{HOLGRIM Recombination Step} in Section 
\ref{sec:HOLGRIM_alg}).
Besides an upper bound on the number of functions used to construct 
an approximation at a given step (cf. Section \ref{sec:HOLGRIM_alg}),
we have no control over the functions used.
Moreover, as is the case for GRIM, there is no requirement 
that any of the functions used at a specific step must be used in
any of the subsequent steps.

The \textbf{HOLGRIM} algorithm detailed in Section 
\ref{sec:HOLGRIM_alg} is designed to approximate a given linear
combination of $\Lip(\gamma,\Sigma,W)$ functions when 
$\Sigma \subset V$ is a finite subset. 
However, a consequence of the \emph{Lipschitz Sandwich Theorems}
established in \cite{LM24} is that a $\Lip(\gamma)$ function 
defined throughout a compact subset can be well-approximated 
throughout the compact subset provided one can well-approximate 
the $\Lip(\gamma)$ function on a particular finite subset.
(cf. Section 4 in \cite{LM24}). 
In Section \ref{sec:compact_domains} of this paper we outline how 
this observation enables 
the case of compact domains to be tackled via the case of finite
domains that the \textbf{HOLGRIM} algorithm is designed to solve.
Moreover, we additionally illustrate in Section 
\ref{sec:compact_domains} how the 
\emph{Lipschitz Sandwich Theorem 3.1} in \cite{LM24} can be used 
to strengthen the pointwise closeness of an approximation $u$ to 
$\vph$ to an estimate on the $\Lip(\eta)$ norm of the difference
$\vph - u$ for a fixed $\eta \in (0,\gamma)$.

The complexity cost of the \textbf{HOLGRIM} algorithm is analysed 
in Section \ref{sec:HOLGRIM_complexity_cost}.
The \textbf{HOLGRIM} algorithm is designed to be a one-time tool 
that is applied a single time to find a sparse approximation of 
the target $\vph$.
Repeated use of the returned approximation, which will be more 
cost-effectively computed than the original target $\vph$, 
for inference on new 
inputs will recover the up-front cost of the 
\textbf{HOLGRIM} algorithms implementation.
Models, such as recognition models or classification models, 
that will be repeatedly computed on new inputs for the purpose
of inference or prediction are ideal candidates for HOLGRIM 
to approximate.

Consequently, the primary aim of our complexity cost considerations 
in Section \ref{sec:HOLGRIM_complexity_cost} is to verify that 
implementing the \textbf{HOLGRIM} algorithm is feasible. 
This is verified by proving that, at worst, the complexity cost
of running the \textbf{HOLGRIM} algorithm is 
\beq
    \label{eq:worst_HOLGRIM_cost}
        \cO \left( s c D(d,k) M \n \Lambda + 
        s c^2 D(d,k)^2 M \n \Lambda^2 + 
        s c^3 D(d,k)^3 M \log \left( \frac{\n}{cD(d,k)} \right) 
        \Lambda^3 \right)
\eeq 
where $\n$ is the number of functions in $\cf$, 
$\Lambda$ is the number of points in $\Sigma$, 
$M$ is the maximum number of steps for which the 
\textbf{HOLGRIM} algorithm will be run,
$s$ is the maximum number of shuffles considered during 
each application of recombination, 
$c$ is the dimension of $W$, $d$ is the dimension of $V$, 
$k \in \Z_{\geq 0}$ is the integer for which 
$\gamma \in (k,k+1]$, and $D = D(d,k)$ is an integer that depends 
only on $d$ and $k$ and whose precise definition may be found 
in \eqref{eq:D_ab_Q_ijab_def_not_sec} in 
Section \ref{sec:pointwise_via_lin_funcs}. 
The upper bound on the complexity cost stated in 
\eqref{eq:worst_HOLGRIM_cost} is a consequence of the 
complexity cost bounds established in Lemma 
\ref{lemma:complex_cost_HOLGRIM_alg} in 
Section \ref{sec:HOLGRIM_complexity_cost}.

Performance guarantees for the \textbf{HOLGRIM} algorithm are 
considered in Sections 
\ref{sec:HOLGRIM_conv_anal_via_GRIM_conv_thm}, 
\ref{sec:HOLGRIM_sup_lemmata}, 
and \ref{sec:HOLGRIM_conv_anal}.
The approach adopted in \cite{LM22} leads to guarantees in 
terms of specific geometric properties of the linear functionals
$\Sigma_k^{\ast} \subset \Lip(\gamma,\Sigma,W)$ 
(cf. Section \ref{sec:HOLGRIM_conv_anal_via_GRIM_conv_thm}).
In order to obtain guarantees involving the geometry of the 
data $\Sigma$ itself, we utilise the 
\emph{Lipschitz Sandwich Theorems} established in \cite{LM24}.

Theoretical guarantees in terms of a particular packing number
of $\Sigma$ in $V$ are established for the \textbf{HOLGRIM} 
algorithm in which a single new point is chosen at each step 
in the \emph{HOLGRIM Convergence Theorem} 
\ref{thm:HOLGRIM_conv_thm}.
Packing numbers, and the closely related notion of covering 
numbers, were first considered by Kolmogorov \cite{Kol56}, 
and have subsequently arisen in contexts including 
eigenvalue estimation \cite{Car81,CS90,ET96},
Gaussian Processes \cite{LL99,LP04}, and machine learning 
\cite{EPP00,SSW01,Zho02,Ste03,SS07,Kuh11,MRT12,FS21}.
The main technical result underpinning the 
\emph{HOLGRIM Convergence Theorem} \ref{thm:HOLGRIM_conv_thm}
is the \emph{HOLGRIM Point Separation Lemma} 
\ref{lemma:HOLGRIM_dist_between_interp_points}
which establishes that points chosen during the 
\textbf{HOLGRIM} algorithm are a definite $V$ distance apart.

The paper is structured as follows.
In Section \ref{sec:lip_funcs} we rigorously define the notion 
of a $\Lip(\gamma,\m,W)$ function for an arbitrary closed subset
$\m \subset V$, and additionally record variants of the 
particular \emph{Lipschitz Sandwich Theorems} from \cite{LM24} 
that will be used within this article.

In Section \ref{sec:prob_formulation} we rigorously formulate
the sparse approximation problem that the \textbf{HOLGRIM}
algorithm is designed to tackle.

In Section \ref{sec:compact_domains} we outline how being 
able to solve the sparse approximation problem detailed in 
Section \ref{sec:prob_formulation} for a finite subset
$\Sigma \subset V$ enables one to solve the same sparse 
approximation problem for a compact subset $\Omega \subset V$.
Both a strategy utilising the 
\emph{Pointwise Lipschitz Sandwich Theorem 3.11} of \cite{LM24}
(cf. Theorem \ref{thm:pointwise_lip_sand_thm} in this paper) 
to obtain a pointwise approximation throughout $\Omega$ 
and a strategy utilising the 
\emph{Lipschitz Sandwich Theorem 3.1} of \cite{LM24} 
(cf. Theorem \ref{thm:lip_sand_thm} in this paper) 
to obtain, for a fixed $\eta \in (0, \gamma)$, a 
$\Lip(\eta,\Omega,W)$-norm approximation throughout $\Omega$ 
are covered.

In Section \ref{sec:pointwise_via_lin_funcs},
given an arbitrary closed subset $\m \subset V$ and a point
$z \in \m$, we define a finite subset 
$\Tau_{z,k} \subset \Lip(\gamma,\m,W)^{\ast}$ of the dual-space
$\Lip(\gamma,\m,W)^{\ast}$ of bounded linear functionals 
$\Lip(\gamma,\m,W) \to \R$ such that for any 
$\psi \in \Lip(\gamma,\m,W)$ the value of $\psi$ at $z$ 
is determined by the values 
$\sigma(\psi)$ for the linear functionals $\sigma \in \Tau_{z,k}$
(cf. Lemma \ref{lemma:number_of_coeffs_for_point_value}).

In Section \ref{sec:HOLGRIM_alg} we present and discuss the
\textbf{HOLGRIM} algorithm.
We formulate both the \textbf{HOLGRIM Extension Step} 
governing how we extend an existing collection of points 
$P \subset \Sigma$ and the 
\textbf{HOLGRIM Recombination Step} governing the 
\emph{Recombination Thinning Lemma 3.1} in \cite{LM22}
is used in the \textbf{HOLGRIM} algorithm.
Additionally, we establish an upper bound on the number of 
functions used to construct the approximation at each step
of the \textbf{HOLGRIM} algorithm.

In Section \ref{sec:HOLGRIM_complexity_cost} we analyse the 
complexity cost of the \textbf{HOLGRIM} algorithm. 
We prove Lemma \ref{lemma:complex_cost_HOLGRIM_alg} establishing 
the complexity cost of any implementation of the \textbf{HOLGRIM} 
algorithm. We further record an upper bound for the complexity 
cost of the most expensive implementation.

In Section \ref{sec:HOLGRIM_conv_anal_via_GRIM_conv_thm} we 
discuss both the performance guarantees inherited by the 
\textbf{HOLGRIM} algorithm from the 
\emph{Banach GRIM Convergence Theorem 6.2} in \cite{LM22} 
and the limitations/issues of the results obtained via this 
approach.

In Section \ref{sec:HOLGRIM_sup_lemmata} we prove the 
\emph{HOLGRIM Point Separation Lemma} 
\ref{lemma:HOLGRIM_dist_between_interp_points} establishing that 
the points selected during the \textbf{HOLGRIM} algorithm 
in which a single new point is chosen at each step are 
guaranteed to be a definite $V$-distance apart from one and other.

In Section \ref{sec:HOLGRIM_conv_anal} we prove the 
\emph{HOLGRIM Convergence Theorem} \ref{thm:HOLGRIM_conv_thm}.
This result establishes that a particular packing number of the
data $\Sigma$ in $V$ provides an upper bound for the maximum 
number of steps that the \textbf{HOLGRIM} algorithm can 
complete before an approximation of $\vph$ possessing the 
desired level of accuracy is returned. 
We subsequently discuss the consequences of the 
\emph{HOLGRIM Convergence Theorem} \ref{thm:HOLGRIM_conv_thm}.
We convert the assertions of the 
\emph{HOLGRIM Convergence Theorem} \ref{thm:HOLGRIM_conv_thm} into
a reversed implication establishing how well the 
\textbf{HOLGRIM} algorithm can approximate $\vph$ under a 
restriction on the maximum number of functions that can be used 
to construct the approximation 
(cf. Remark \ref{rmk:HOLGRIM_conv_thm_reverse_imp}).
We verify that the data concentration required to guarantee 
that the \textbf{HOLGRIM} algorithm returns a good 
approximation of $\vph$ using strictly less than $\n$ of the 
functions in $\cf$ decreases as the regularity parameter 
$\gamma > 0$ increases 
(cf. Remark \ref{rmk:HOLGRIM_conv_thm_gamma_depend}).
We establish a robustness property satisfied by the returned 
approximation (cf. Remark \ref{rmk:HOLGRIM_conv_thm_3}).
We show that the guarantees of Theorem \ref{thm:HOLGRIM_conv_thm} 
can be both extended to the setting of a compact subset 
$\Omega \subset V$ and have the conclusion strengthened to, 
for a given $\eta \in (0,\gamma)$, a $\Lip(\eta)$-norm 
estimates on the difference between the target $\vph$ 
and the returned approximation $u$ 
following the strategies outlined in Section 
\ref{sec:compact_domains}
(cf. Remarks \ref{rmk:HOLGRIM_conv_thm_Omega_ptwise}, 
\ref{rmk:HOLGRIM_conv_thm_Omega_n0_pointwise}, 
\ref{rmk:HOLGRIM_conv_thm_Omega_lip_eta}, and 
\ref{rmk:HOLGRIM_conv_thm_Omega_n0_lip_eta}).
Finally we illustrate a number of these consequences of the 
\emph{HOLGRIM Convergence Theorem} \ref{thm:HOLGRIM_conv_thm} via
an explicit example in the Euclidean setting 
$\Omega := [0,1]^d \subset \R^d$ 
(cf. Remark \ref{rmk:HOLGRIM_conv_thm_example}).
\vskip 4pt
\noindent 
\emph{Acknowledgements}: This work was supported by 
the \DataSig Program under the EPSRC grant 
ES/S026347/1, the Alan Turing Institute under the
EPSRC grant EP/N510129/1, the Data Centric Engineering 
Programme (under Lloyd's Register Foundation grant G0095),
the Defence and Security Programme (funded by the UK 
Government) and the Hong Kong Innovation and Technology 
Commission (InnoHK Project CIMDA). 
This work was funded by the Defence and Security
Programme (funded by the UK Government).

\section{Lipschitz Functions and Sandwich Theorems}
\label{sec:lip_funcs}
In this section we provide a rigorous definition of a 
$\Lip(\gamma)$ function in the sense of Stein \cite{Ste70} 
in the same generality considered in \cite{LM24}.
Additionally, we record the particular 
\emph{Lipschitz Sandwich Theorems} from \cite{LM24} 
that will be useful within this article. 

Let $V$ and $W$ be real Banach spaces and assume that 
$\m \subset V$ is a closed subset. 
Throughout this article we adopt the convention that balls 
denoted by $\B$ are taken to be open, whilst those denoted
by $\ovB$ are taken to be closed.
To define $\Lip(\gamma,\m,W)$ functions, we must
first make a choice of norm on the tensor powers of $V$. 
We restrict to considering norms that are 
\emph{admissible} in the following sense.

\begin{definition}[Admissible Norms on Tensor Powers]
\label{admissible_tensor_norm}
Let $V$ be a Banach space. We say that its tensor powers
are endowed with admissible norms if for each 
$n \in \Z_{\geq 2}$ we have equipped the tensor power 
$V^{\otimes n}$ of $V$ with a norm 
$|| \cdot ||_{V^{\otimes n}}$ such that the following 
conditions hold.
\begin{itemize}
    \item For each $n \in \Z_{\geq 2}$ the symmetric 
    group $S_n$ acts isometrically on $V^{\otimes n}$,
    i.e. for any $\rho \in S_n$ and any
    $v \in V^{\otimes n}$ we have
    \beq
        \label{sym_group_iso}
            || \rho (v) ||_{V^{\otimes n}}
            = || v ||_{V^{\otimes n}}.
    \eeq
    The action of $S_n$ on $V^{\otimes n}$
    is given by permuting the order of the letters,
    i.e. if $a_1 \otimes \ldots \otimes a_n 
    \in V^{\otimes n}$ and $\rho \in S_n$ then 
    $\rho ( a_1 \otimes \ldots \otimes a_n )
    := a_{\rho(1)} \otimes \ldots \otimes a_{\rho(n)}$,
    and the action is extended to the entirety of
    $V^{\otimes n}$ by linearity.
    \item For any $n,m \in \Z_{\geq 1}$ and any
    $v \in V^{\otimes n}$ and $w \in V^{\otimes m}$
    we have
    \beq
        \label{ten_prod_unit_norm}
            || v \otimes w ||_{V^{\otimes (n+m)}}
            \leq 
            || v ||_{V^{\otimes n}} 
            || w ||_{V^{\otimes m}}.
    \eeq
    \item For any $n,m \in \Z_{\geq 1}$ and any
    $\phi \in \left(V^{\otimes n}\right)^{\ast}$ 
    and $\sigma \in \left(V^{\otimes m}\right)^{\ast}$
    we have
    \beq
        \label{ten_prod_unit_dual_norm}
            || \phi \otimes \sigma
            ||_{\left(V^{\otimes (n+m)}\right)^{\ast}}
            \leq 
            || \phi 
            ||_{\left(V^{\otimes n}\right)^{\ast}} 
            || \sigma 
            ||_{\left(V^{\otimes m}\right)^{\ast}}.
    \eeq
    Here, given any $k \in \Z_{\geq 1}$,
    the norm 
    $|| \cdot ||_{\left(V^{\otimes k}\right)^{\ast}}$
    denotes the dual-norm induced by
    $|| \cdot ||_{V^{\otimes k}}$.
\end{itemize}
\end{definition}
\vskip 4pt
\noindent
A consequence of having \emph{both} the inequalities 
\eqref{ten_prod_unit_norm} and 
\eqref{ten_prod_unit_dual_norm} being satisfied is 
that we may conclude that there is actually equality 
in both estimates (see \cite{Rya02}).
Thus if the tensor powers of $V$ are all equipped with 
admissible tensor norms in the sense of Definition 
\ref{admissible_tensor_norm}, we have equality in both 
\eqref{ten_prod_unit_norm} and
\eqref{ten_prod_unit_dual_norm}.

We next use this notion of admissible tensor norm to 
rigorously define a $\Lip(\gamma,\m,W)$ function 
within this setting (cf. Definition 2.2 in \cite{LM24}).

\begin{definition}[$\Lip(\gamma,\m,W)$ functions]
\label{lip_k_def}
Let $V$ and $W$ be Banach spaces, $\m \subset V$ 
a closed subset, and assume that the tensor powers of 
$V$ are all equipped with admissible norms. 
Let $\gamma > 0$ with $k \in \Z_{\geq 0}$ such that
$\gamma \in (k,k+1]$. Then let 
$\psi = (\psi^{(0)} , \psi^{(1)} , \ldots , \psi^{(k)})$ 
where $\psi^{(0)} : \m \to W$ is a function taking 
its values in $W$ and, for each 
$l \in \{1, \ldots , k\}$,  
$\psi^{(l)} : \m \to \cl( V^{\otimes l} ; W)$
is a functions taking its values in the space of 
symmetric $l$-linear forms from $V$ to $W$.
Then $\psi$ is a $\Lip(\gamma,\m,W)$ function 
if there exists a constant $M \geq 0$ for which the 
following is true.
\begin{itemize} 
    \item For each $l \in \{0, \ldots , k\}$ and 
    every $z \in \m$ we have that
    \beq    
        \label{lip_k_bdd}
            || \psi^{(l)}(z) ||_{\cl(V^{\otimes l} ; W)}
            \leq M
    \eeq
    \item For each $j \in \{0, \ldots , k\}$ 
    define $R^{\psi}_j : \m \times \m \to 
    \cl (V^{\otimes j} ; W)$ for  
    $z,p \in \m$ and $v \in V^{\otimes j}$ by 
    \beq
        \label{lip_k_tay_expansion}
            R^{\psi}_j(z,p)[v] :=
            \psi^{(j)}(p)[v] -
            \sum_{s=0}^{k-j} \frac{1}{s!}
            \psi^{(j+s)}(z)
            \left[v \otimes (p-z)^{\otimes s}\right].
    \eeq
    Then whenever $l \in \{0, \ldots , k\}$ and 
    $x,y \in \m$ we have 
    \beq
        \label{lip_k_remain_holder}
            \left|\left| 
            R^{\psi}_l(x,y) 
            \right|\right|_{\cl(V^{\otimes l} ; W)}
            \leq
            M ||y - x||_V^{\gamma - l}.
    \eeq
\end{itemize}
We sometimes say that $\psi \in \Lip(\gamma,\m,W)$
without explicitly mentioning the functions 
$\psi^{(0)} , \ldots , \psi^{(k)}$. 
Furthermore, given $l \in \{0, \ldots , k\}$, 
we introduce the notation that 
$\psi_{[l]} := (\psi^{(0)} , \ldots , \psi^{(l)})$.
The $\Lip(\gamma,\m,W)$ norm of $\psi$, denoted by
$||\psi||_{\Lip(\gamma,\m,W)}$, is the smallest 
$M \geq 0$ satisfying the requirements
\eqref{lip_k_bdd} and \eqref{lip_k_remain_holder}.
\end{definition}

\begin{remark}
\label{rmk:op_norm_convention}
Assuming the same notation as in Definition \ref{lip_k_def},
we adopt the same convention as used in \cite{LM24} 
regarding the choice of norm on the space 
$\cl(V^{\otimes l};W)$ for each $l \in \{0, \ldots , W\}$.
That is, recall that $V^{\otimes 0} := \R$ so that 
$\cl(V^{\otimes 0};W) = W$. Hence we take 
$||\cdot||_{\cl(V^{\otimes 0};W)}$ to be the norm 
$||\cdot||_W$ given on $W$.
Further, given $l \in \{1, \ldots , k\}$, we equip 
$\cl(V^{\otimes l};W)$ with its operator norm determined by
\beq
    \label{eq:op_norm_l_def}
        || \bA ||_{\cl(V^{\otimes l};W)} 
        := 
        \sup \left\{ \left|\left| \bA[v] \right|\right|_W
        ~:~ v \in V^{\otimes l} \text{ and }
        ||v||_{V^{\otimes l}} = 1
        \right\}.
\eeq
This results in the same notion of a $\Lip(\gamma,\m,W)$
function as considered in \cite{LM24}. Consequently, it is
the case that an upper bound on the resulting 
$\Lip(\gamma)$-norm $||\cdot||_{\Lip(\gamma,\m,W)}$ is a 
stronger restriction than the same upper bound on the 
corresponding notion of $\Lip(\gamma)$-norm considered in 
the works \cite{Bou15,Bou22,BL22}.
\end{remark}

\begin{remark}
\label{rmk:Lip_gamma_subset_C0}
Assuming the same notation as in Definition 
\ref{lip_k_def}, suppose 
$\psi = \left( \psi^{(0)} , \ldots , \psi^{(k)} \right) 
\in \Lip(\gamma,\m,W)$.
Let $C^0(\m;W)$ denote the space of continuous 
maps $\m \to W$ and for $f \in C^0(\m;W)$ take
$||f||_{C^0(\m;W)} := \sup \left\{ ||f(x)||_W :
x \in \m \right\}$.
Then $\psi^{(0)} \in C^0(\m;W)$
and $|| \psi^{(0)} ||_{C^0(\m;W)} \leq 
||\psi||_{\Lip(\gamma,\m,W)}$.
Consequently, the mapping $\psi \mapsto \psi^{(0)}$
gives an embedding of $\Lip(\gamma,\m,W)$ into 
the space $C^0(\m;W)$.
\end{remark}

\begin{remark}
\label{rmk:Lip_gamma_nesting}
Assuming the same notation as in Definition 
\ref{lip_k_def}, suppose 
$\psi = \left( \psi^{(0)} , \ldots , \psi^{(k)} \right) 
\in \Lip(\gamma,\m,W)$, and that $\eta \in (0, \gamma)$ 
with $q \in \Z_{\geq 0}$ such that $\eta \in (q,q+1]$.
Then it follows that 
$\psi_{[q]} = \left( \psi^{(0)} , \ldots , \psi^{(q)} 
\right) \in \Lip(\eta,\m,W)$ and there exists a 
constant $C = C(\gamma,\eta) \geq 1$ for which 
$||\psi_{[q]}||_{\Lip(\eta,\m,W)} \leq 
C || \psi ||_{\Lip(\gamma,\m,W)}$; 
see, for example, Lemma 6.1 in \cite{LM24}.
Consequently, the mapping $\psi \mapsto \psi_{[q]}$
gives an embedding of $\Lip(\gamma,\m,W)$ into the 
space $\Lip(\eta,\m,W)$.
\end{remark}

\begin{remark}
\label{rmk:lip_gamma_on_open_sets}
Suppose that $O \subset V$ is open and that
$\psi = \left( \psi^{(0)} , \ldots , \psi^{(k)} \right) 
\in \Lip(\gamma,O,W)$.
Then the expansions \eqref{lip_k_tay_expansion} and the estimates
\eqref{lip_k_bdd} and \eqref{lip_k_remain_holder} yield that 
$\psi^{(0)} : O \to W$ is $k$-times Fr\'{e}chet differentiable 
with, for $l \in \{0, \ldots , k\}$, $l^{\text{th}}$ Fr\'{e}chet 
derivative given by $\psi^{(l)}$.
Hence the $k^{\text{th}}$ Fr\'{e}chet derivative of 
$\psi^{(0)}$ is $(\gamma - k)$-H\"{o}lder continuous so that
$\psi^{(0)} \in C^{k,\gamma-k}(O;W)$.
Consequently, the mapping $\psi \mapsto \psi^{(0)}$ gives 
an embedding of $\Lip(\gamma,O,W)$ into the space
$C^{k,\gamma-k}(O;W)$.
Moreover, Taylor's theorem establishes that this mapping 
is in fact a one-to-one correspondence.
\end{remark}

\begin{remark}
\label{rmk:lip_gamma_locally_polynomial}
A helpful way to think about  
$\Lip(\gamma,\m,W)$ functions is that they 
``locally looks like a polynomial function".
To elaborate, assume the same notation as in Definition 
\ref{lip_k_def} and let
$\psi = \left( \psi^{(0)} , \ldots , \psi^{(k)} \right) 
\in \Lip(\gamma,\m,W)$.
For each point $x \in \m$ we can define a 
polynomial $\Psi_x : V \to W$  for $v \in V$ by
\beq
    \label{lip_k_psi_proposed_value}
        \Psi_x(v)  
        := \sum_{s=0}^{k} 
        \frac{1}{s!}
        \psi^{(s)}(x)\left[  
        (v-x)^{\otimes s} \right].
\eeq
For each $l \in \{0, \ldots , k\}$ and any point 
$x \in \m$ the element 
$\psi^{(l)}(x) \in \cl(V^{\otimes l};W)$ is the 
$l^{\text{th}}$ Fr\'{e}chet derivative of the 
polynomial $\Psi_{x}$ evaluated at the point $x \in \m$.
Moreover, the polynomial $\Psi_x$ provides a proposal, 
based at the point $x \in \m$, for how 
$\psi^{(0)}$ could be extended to the entirety of $V$.
The expansions \eqref{lip_k_tay_expansion} and
the estimates \eqref{lip_k_remain_holder} ensure 
that for points $y \in \m$ that are close, in the 
$||\cdot||_V$ norm sense, to $x$ that the values 
$\psi^{(0)}(y)$ and $\Psi_x(y)$ must be close
in a quantifiable $||\cdot||_W$ sense.
In fact, \eqref{lip_k_tay_expansion} and 
\eqref{lip_k_remain_holder} ensure that if 
$x_1 , x_2 \in \m$ are close in the $||\cdot||_V$
sense, then the polynomial proposals $\Psi_{x_1}$
and $\Psi_{x_2}$ must be close as functions $V \to W$ 
in a pointwise sense.
A precise estimate giving quantifiable meaning to the 
notion of closeness can be found in Section 2 of \cite{LM24}.
\end{remark}

\begin{remark}
\label{rmk:notational_easing}
Assuming the same notation as in Definition 
\ref{lip_k_def}, and suppose 
$\psi = \left( \psi^{(0)} , \ldots , \psi^{(k)} \right) 
\in \Lip(\gamma,\m,W)$.
To ease notation, given  
$l \in \{0, \ldots , k\}$ we define 
$\Lambda_{\psi}^l : \m \to \R_{\geq 0}$ 
by setting, for $x \in \m$,
\beq
    \label{lip_eta_lip_norm_func}
        \Lambda^l_{\psi}(x) := 
        \max_{ j \in \{0 , \ldots , l \} }
        \left|\left| \psi^{(j)}(x)
        \right|\right|_{\cl(V^{\otimes j} ; W)}.
\eeq
A consequence of \eqref{lip_eta_lip_norm_func}
is that for any $l \in \{0, \ldots , k\}$ and any 
$z \in \m$ we have that $\Lambda^l_{\psi}(z) \leq 
|| \psi ||_{\Lip(\gamma,\ca,W)}$ 
for any subset $\ca \subset \m$ with $z \in \ca$.
\end{remark}
\vskip 4pt
\noindent
We end this subsection by stating the various 
\emph{Lipschitz Sandwich Theorems} from \cite{LM24} that 
we will use in this article.
We first record a variant of the 
\emph{Lipschitz Sandwich Theorem} 3.1 in \cite{LM24}
in which both the constants $K_1$ and $K_2$ of that result 
are equal to the same constant $K_0 > 0$.

\begin{theorem}[\textbf{Lipschitz Sandwich Theorem 3.1} 
in \cite{LM24}]
\label{thm:lip_sand_thm}
Let $V$ and $W$ be Banach spaces, and assume 
that the tensor powers of 
$V$ are all equipped with admissible norms (cf. 
Definition \ref{admissible_tensor_norm}).
Assume that $\m \subset V$ is non-empty and closed.
Let $\ep, K_0 > 0$, and $\gamma > \eta > 0$ with 
$k,q \in \Z_{\geq 0}$ such that $\gamma \in (k,k+1]$
and $\eta \in (q,q+1]$. Then there exist constants 
$\de_0 = \de_0 (\ep,K_0,\gamma ,\eta) > 0$ and 
$\ep_0 = \ep_0(\ep, K_0, \gamma, \eta) > 0$ for
which the following is true.

Suppose $B \subset \m$ is a closed subset that is
a $\de_0$-cover of $\m$ in the sense that
\beq
    \label{eq:lip_sand_thm_B_cover_sigma}
        \m \subset \bigcup_{x \in B}
        \ovB_V (x, \de_0)
        = B_{\de_0}
        := \left\{ v \in V ~:~
        \text{There exists } z \in B \text{ such that }
        ||v - z ||_V \leq \de_0 \right\}.
\eeq
Suppose 
$\psi = \left(\psi^{(0)} ,\ldots ,\psi^{(k)}\right) , 
\phi = \left( \phi^{(0)} , \ldots , \phi^{(k)} \right)
\in \Lip(\gamma,\m,W)$ with 
$||\psi||_{\Lip(\gamma,\m,W)} , 
||\phi||_{\Lip(\gamma,\m,W)} \leq K_0$.
Then, if 
$\psi_{[q]}:= \left(\psi^{(0)},\ldots ,\psi^{(q)}\right)$ and 
$\psi_{[q]}:= \left(\phi^{(0)},\ldots ,\phi^{(q)}\right)$,
we have that 
\beq
    \label{eq:lip_sand_thm_imp}
        \sup_{z \in B} \Lambda^k_{\psi - \phi}(z) \leq \ep_0
        \qquad \implies \qquad 
        \left|\left| \psi_{[q]} - \phi_{[q]} \right|\right|_{
        \Lip(\eta,\m,W)}
        \leq \ep.
\eeq
\end{theorem}
\vskip 4pt
\noindent 
The following result records a variant of the 
\emph{Pointwise Lipschitz Sandwich Theorem} 3.11 in \cite{LM24}
in which both the constants $K_1$ and $K_2$ of that result 
are equal to the same constant $K_0 > 0$.

\begin{theorem}[\textbf{Pointwise Lipschitz Sandwich Theorem 
3.11} in \cite{LM24}]
\label{thm:pointwise_lip_sand_thm}
Let $V$ and $W$ be Banach spaces, and assume  
that the tensor powers of 
$V$ are all equipped with admissible norms (cf. 
Definition \ref{admissible_tensor_norm}).
Assume that $\m \subset V$ is closed.
Let $K_0 , \gamma , \ep > 0$ with $k \in \Z_{\geq 0}$
such that $\gamma \in (k,k+1]$ and 
$0 \leq \ep_0 < \ep \leq K_0$.
Then given any $l \in \{0, \ldots , k\}$, there exists
a constant $\de_0 = \de_0 (\ep,\ep_0,K_0,\gamma,l) > 0$,
defined by
\beq
    \label{eq:pointwise_lip_sand_thm_de0}
        \de_0 := \sup \left\{
        \th > 0 ~:~
        2K_0 \th^{\gamma - l} + \ep_0 e^{\th} 
        \leq 
        \ep \right\} \in (0,1]
\eeq
for which the following is true. 

Suppose $B \subset \m$ is a $\de_0$-cover of $\m$ in the sense that 
\beq
    \label{eq:pointwise_lip_sand_thm_B_cover_Sigma}
        \m \subset \bigcup_{x \in B}
        \ovB_V (x, \de_0)
        =
        B_{\de_0} :=
        \left\{ v \in V ~:~
        \text{There exists } z \in B \text{ such that }
        ||v - z||_V \leq \de_0 \right\}. 
\eeq
Suppose 
$\psi = \left(\psi^{(0)} ,\ldots ,\psi^{(k)}\right) , 
\phi = \left(\psi^{(0)} ,\ldots ,\phi^{(k)}\right)
\in \Lip(\gamma,\m,W)$ with 
$||\psi||_{\Lip(\gamma,\m,W)} , 
||\phi||_{\Lip(\gamma,\m,W)} \leq K_0$.
Then we have that 
\beq
    \label{eq:pointwise_lip_sand_thm_imp}
        \sup_{z \in B} \Lambda^k_{\psi - \phi}(z) \leq \ep_0 
        \qquad \implies \qquad 
        \sup_{z \in \m} \Lambda^l_{\psi - \phi}(z) \leq \ep.
\eeq 
\end{theorem}
\vskip 4pt 
\noindent
We end this subsection by recording the following 
explicit estimates pointwise estimates for 
$\psi \in \Lip(\gamma,\m,W)$ under the assumption 
that, at a point $p \in \m$, we have 
$\Lambda^k_{\psi}(p) \leq \ep_0$.
These estimates are contained in the 
\emph{Pointwise Estimates Lemma} 7.1 in \cite{LM24} 
under the choices of the $\Gamma$, $F$, $A$, $r_0$, $\rho$, 
and $q$ in that result as $\m$, $\psi$, $K_0$, $\ep_0$, 
$\gamma$, and $k$ here.
The precise result is the following lemma.

\begin{lemma}[Variant of Pointwise Estimates 
Lemma 7.1 in \cite{LM24}]
\label{lemma:explicit_pointwise_est}
Let $V$ and $W$ be Banach spaces, and assume that the 
tensor powers of $V$ are all equipped with admissible 
norms (cf. Definition \ref{admissible_tensor_norm}).
Assume $\m \subset V$ is closed with $p \in \m$.
Let $K_0, \gamma > 0$, $\ep_0 \geq 0$, and  
$k \in \Z_{\geq 0}$ such that $\gamma \in (k,k+1]$.
Let $\psi = \left( \psi^{(0)} , \ldots , \psi^{(k)}
\right) \in \Lip(\gamma, \m,W)$ with 
$|| \psi ||_{\Lip(\gamma, \m,W)} \leq K_0$.
Then if $\Lambda^k_{\psi}(p) \leq \ep_0$ we may conclude,
for any $z \in \m$ and any $l \in \{0, \ldots , k\}$,
that
\beq
    \label{eq:explicit_pointwise_est_lemma_conc}
        \left|\left| \psi^{(l)}(z) 
        \right|\right|_{\cl(V^{\otimes l};W)} 
        \leq 
        \min \left\{ K_0 ~,~   
        K_0 ||z-p||_V^{\gamma - l}  + 
        \ep_0 \sum_{j=0}^{k-l}
        \frac{1}{j!} ||z-p||_V^j \right\}.
\eeq
\end{lemma}

\section{Sparse Approximation Problem Formulation}
\label{sec:prob_formulation}
In this section we rigorously formulate the sparse 
approximation problem that the 
\textbf{HOLGRIM} algorithm will be designed to tackle.
For this purpose we again let $V$ and $W$ be real Banach spaces, but
we now additionally impose that both $V$ and $W$ are finite 
dimensional. In particular, 
we assume that $V$ has dimension $d \in \Z_{\geq 1}$, 
and that $W$ has dimension $c \in \Z_{\geq 1}$.
Whilst we do not do so in this article, one could work
modulo isometric isomorphism and assume that 
$V = \R^d$ and $W = \R^c$.
Instead, we work directly with $V$ and $W$ and assume 
that the tensor powers of $V$ are 
all equipped with admissible tensor norms in the sense 
of Definition \ref{admissible_tensor_norm}.

We now formulate the sparse approximation problem that 
we consider in this article.
Let $\gamma > 0$ with $k \in \Z_{\geq 0}$ such 
that $\gamma \in (k,k+1]$. 
Let $\n , \Lambda \in \Z_{\geq 1}$ be (large) positive integers.
Let $\Sigma \subset V$ 
be a finite subset of cardinality $\Lambda$.
For every $i \in \{ 1 , \ldots , \n \}$ we assume that  
$f_i = \left( f_i^{(0)} , \ldots , f^{(k)}_i 
\right) \in \Lip(\gamma,\Sigma,W)$ is non-zero, and set  
$\cf := \left\{ f_i :~ i \in \{1 , \ldots , \n\}
\right\} \subset \Lip(\gamma,\Sigma,W)$.
Fix a choice of scalars 
$A_1 , \ldots , A_{\n} \in \R_{>0}$ such that, for every 
$i \in \{1, \ldots , \n\}$, we have 
$||f_i||_{\Lip(\gamma,\Sigma,W)} \leq A_i$.
Fix a choice of non-zero coefficients
$a_1, \ldots , a_{\n} \in \R \setminus \{0\}$ and
consider $\vph = \left( \vph^{(0)} , \ldots , 
\vph^{(k)} \right) \in \Lip(\gamma,\Sigma,W)$ defined by
\beq
    \label{lip_eta_varphi}
        \varphi := \sum_{i=1}^{\n} a_i f_i
        \quad \text{so that for every }
        l \in \{0, \ldots , k\} \text{ we have}
        \quad
        \vph^{(l)} = \sum_{i=1}^{\n} a_i 
        f_i^{(l)}.
\eeq
We consider the following sparse approximation problem.
Given $\ep > 0$ and a fixed $l \in \{0, \ldots , k\}$,
find an element $u = \sum_{i=1}^{\n} b_i f_i \in \Span(\cf)$
satisfying the following properties.
\vskip 4pt
\noindent
\textbf{Finite Domain Approximation Conditions}
\begin{enumerate}[label=(\arabic*)]
    \item\label{approx_prop_1} 
    There exists an integer $M \in \{1, \ldots , \n - 1\}$
    for which the set 
    \beq    
        \label{eq:u_support_goal}
            \support(u) := \left\{ i \in \{1, \ldots , \n\} 
            : b_i \neq 0 \right\} 
            \qquad \text{satisfies that} \qquad 
            \# \left( \support(u) \right) \leq M.
    \eeq 
    \item\label{approx_prop_2}
    The coefficients $b_1 , \ldots , b_{\n} \in \R$ satisfy 
    that 
    \beq
        \label{eq:coeffs_sum_goal}
            \sum_{i=1}^{\n} |b_i| A_i
            \leq C := 
            \sum_{i=1}^{\n} |a_i| A_i.
    \eeq 
    \item\label{approx_prop_3}
    The function $u \in \Span(\cf)$ well-approximates $\vph$ 
    throughout $\Sigma$ in the pointwise sense that 
    \beq
        \label{eq:approx_poitwise_est_goal}
            \max_{p \in \Sigma} \left\{
            \Lambda^l_{\vph - u}(p) \right\}
            \stackrel{
            \eqref{lip_eta_lip_norm_func}
            }{=} 
            \max_{p \in \Sigma} \left\{
            \max_{j \in \{0 , \ldots , l\}} \left\{
            \left|\left| \vph^{(j)}(p) - u^{(j)}(p)
            \right|\right|_{\cl(V^{\otimes j};W)}
            \right\} \right\}
            \leq \ep.
    \eeq
\end{enumerate}
\vskip 8pt
\noindent
The choice of $l \in \{0, \ldots , k\}$ provides
a gauge for the strength of approximation we seek.
Loosely, the condition $\Lambda^l_{\vph -u}(p) \leq \ep$
generalises the notion of approximating the derivatives of 
$\vph$ up to order $l$ at the point $p$.
Indeed, if we have the situation that 
$f_1 , \ldots , f_{\n} \in \Lip(\gamma,O,W)$ for some open subset 
$O \subset V$ satisfying $\Sigma \subset O$, then each $f_i^{(0)}$
is $k$ times continuously Frech\'{e}t differentiable 
at $p$ with, for each $j \in \{0, \ldots , k\}$,
$f_i^{(j)}(p) = D^j f_i^{(0)}(p)$.
Consequently, the condition that 
$\Lambda^l_{\vph - u}(p) \leq \ep$ requires 
$\vph^{(0)}$ and $u^{(0)}$ to be close in the $C^l$-sense 
that their $j^{\text{th}}$ derivatives, for every 
$j \in \{0, \ldots , l\}$, at $p$ are within 
$\ep$ of one and other as elements in $\cl(V^{\otimes j};W)$.
In the general case that the functions $f_1 , \ldots , f_{\n}$
are defined only on the finite subset $\Sigma \subset V$, 
the condition $\Lambda^l_{\vph-u}(p) \leq \ep$
is a generalisation in which we require, for every 
$j \in \{0, \ldots , l\}$, that at $p \in \Sigma$ the functions 
$\vph^{(j)}(p) , u^{(j)}(p) \in \cl(V^{\otimes j};W)$ 
are within $\ep$ of one and other in the 
$||\cdot||_{\cl(V^{\otimes j};W)}$ norm sense.

Observe that the choice 
$u := \vph$, so that for every $i \in \{1, \ldots , \n\}$ we 
have $b_i := a_i$, satisfies both 
\textbf{Finite Domain Approximation Conditions}
\ref{approx_prop_2} and \ref{approx_prop_3}.
However, from the perspective of reducing computational 
complexity, this trivial approximation of $\vph$ is useless.
\textbf{Finite Domain Approximation Condition} 
\ref{approx_prop_1} is imposed avoid 
this trivial solution.
The upper bound on the cardinality of $\support(u)$ 
required in \eqref{eq:u_support_goal} in 
\textbf{Finite Domain Approximation Condition}
\ref{approx_prop_1} ensures that an approximation 
is acceptable only if it is a linear combination of strictly 
less than $\n$ of the functions in $\cf$.

Perhaps the strangest requirements are those imposed 
on the coefficients
$b_1, \ldots , b_{\n} \in \R$ by
\eqref{eq:coeffs_sum_goal} in 
\textbf{Finite Domain Approximation Condition} 
\ref{approx_prop_2}.
There are two main reasons for this restriction.
Firstly, the requirement \eqref{eq:coeffs_sum_goal} 
in \textbf{Finite Domain Approximation Condition}
\ref{approx_prop_2} ensures that the $\Lip(\gamma,\Omega,W)$ 
norm of the approximation $u$ is bounded above by the 
fixed constant $C$. Secondly, 
\textbf{Finite Domain Approximation Condition} 
\ref{approx_prop_2} will enable us to 
tackle this sparse approximation problem when the functions 
$f_1 , \ldots , f_{\n}$ have a compact domain $\Omega \subset V$ 
rather than a finite domain $\Sigma \subset V$.
We provide the details of this extension in Section 
\ref{sec:compact_domains}.

It is primarily for the purpose of tackling the compact domain
sparse approximation problem specified in Section 
\ref{sec:compact_domains} that we introduce the scalars
$A_1 , \ldots , A_{\n} \in \R_{>0}$ satisfying, for every 
$i \in \{1, \ldots , \n\}$, that 
$||f_i||_{\Lip(\gamma,\Sigma,W)} \leq A_i$, rather than 
simply taking 
$A_1 := ||f_1||_{\Lip(\gamma,\Sigma,W)}, \ldots , 
A_{\n} := ||f_{\n}||_{\Lip(\gamma,\Sigma,W)}$.
If we are only interested in approximating $\vph$ 
on the finite subset $\Sigma$ then this choice is arguably
the most natural one to make. 
Under this choice, the $\Lip(\gamma,\Sigma,W)$ norm 
of the approximation is constrained to be no worse than 
the best a priori upper bound we have for the 
$\Lip(\gamma,\Sigma,W)$ norm of $\vph$.
When we happen to know that the 
functions $f_1 , \ldots , f_{\n}$ are additionally 
be defined at points \emph{not} in $\Sigma$ then we consider
scalars $A_1 , \ldots , A_{\n}$ that are not necessarily 
equal to $||f_1||_{\Lip(\gamma,\Sigma,W)} , \ldots , 
||f_{\n}||_{\Lip(\gamma,\Sigma,W)}$.
This situation is illustrated in the extension to compact
domains presented in Section \ref{sec:compact_domains}.

As stated, there is of course no guarantee that a solution 
to our sparse approximation problem satisfying conditions 
\ref{approx_prop_1}, \ref{approx_prop_2}, and 
\ref{approx_prop_3} exists. Thus the strategy we adopt 
in this paper is as follows. 
Firstly, we define the \textbf{HOLGRIM} algorithm in 
Section \ref{sec:HOLGRIM_alg} that is designed to find an 
approximation $u \in \Span(\cf)$ of $\vph$ satisfying 
conditions \ref{approx_prop_2} and \ref{approx_prop_3}.
We then subsequently establish conditions in Section 
\ref{sec:HOLGRIM_conv_anal} under which 
the approximation returned by the \textbf{HOLGRIM} algorithm 
is guaranteed to additionally satisfy condition 
\ref{approx_prop_1} 
(cf. the \emph{HOLGRIM Convergence Theorem} 
\ref{thm:HOLGRIM_conv_thm} in Section 
\ref{sec:HOLGRIM_conv_anal}).

\section{Extension to Compact Domains}
\label{sec:compact_domains}
In this section we illustrate how being able to solve the 
sparse approximation problem detailed in Section 
\ref{sec:prob_formulation} for a finite domain 
$\Sigma \subset V$ enables one to solve the same sparse 
approximation problem in the setting that the finite domain 
$\Sigma \subset V$ is replaced by a compact domain 
$\Omega \subset V$.
This is achieved by utilising the 
\emph{Lipschitz Sandwich Theorems} obtained in \cite{LM24}
and stated in Section \ref{sec:lip_funcs} of this paper.

For this purpose we let $c,d \in \Z_{\geq 1}$ and assume that
$V$ is a $d$-dimensional real Banach space and that 
$W$ is a $c$-dimensional real Banach space. 
Let $\gamma > 0$ with $k \in \Z_{\geq 0}$ such that 
$\gamma \in (k,k+1]$.
Let $\n \in \Z_{\geq 1}$
be a (large) positive integer, $\Omega \subset V$ a non-empty 
compact subset, and 
$\cf := \{ f_1 , \ldots , f_{\n} \} \subset \Lip(\gamma,\Omega,W)$ 
be a finite collection of non-zero $\Lip(\gamma,\Omega,W)$ 
functions.
Fix a choice of non-zero coefficients 
$a_1 , \ldots , a_{\n} \in \R \setminus \{0\}$ and 
consider $\vph \in \Lip(\gamma,\Omega,W)$ defined by 
$\vph := \sum_{i=1}^{\n} a_i f_i$. 
Define 
$C := \sum_{i=1}^{\n} |a_i| ||f_i||_{\Lip(\gamma,\Omega,W)} > 0$ 
so that $||\vph||_{\Lip(\gamma,\Omega,W)} \leq C$.

Fix $\ep > 0$ and $l \in \{0, \ldots , k\}$.
We first consider a sparse approximation problem analogous 
to the sparse approximation problem specified in Section 
\ref{sec:prob_formulation}.
That is, find an element 
$u = \sum_{i=1}^{\n} b_i f_i \in \Span(\cf)$
satisfying the following properties.
\vskip 4pt
\noindent
\textbf{Compact Domain Approximation Conditions}
\begin{enumerate}[label=(\arabic*)]
    \item\label{compact_approx_prop_1} 
    There exists an integer $M \in \{1, \ldots , \n - 1\}$
    for which the set 
    \beq    
        \label{eq:compact_u_support_goal}
            \support(u) := \left\{ i \in \{1, \ldots , \n\} 
            : b_i \neq 0 \right\} 
            \qquad \text{satisfies that} \qquad 
            \# \left( \support(u) \right) \leq M.
    \eeq 
    \item\label{compact_approx_prop_2}
    The coefficients $b_1 , \ldots , b_{\n} \in \R$ satisfy 
    that 
    \beq
        \label{eq:compact_coeffs_sum_goal}
            \sum_{i=1}^{\n} |b_i| 
            || f_i ||_{\Lip(\gamma,\Omega,W)}
            \leq C := 
            \sum_{i=1}^{\n} |a_i| 
            || f_i ||_{\Lip(\gamma,\Omega,W)}.
    \eeq 
    \item\label{compact_approx_prop_3}
    The function $u \in \Span(\cf)$ well-approximates $\vph$ 
    throughout $\Omega$ in the pointwise sense that 
    \beq
        \label{eq:compact_approx_poitwise_est_goal}
            \sup_{z \in \Sigma} \left\{ 
            \Lambda^l_{\vph - u}(z) \right\}
            \stackrel{
            \eqref{lip_eta_lip_norm_func}
            }{=} 
            \sup_{z \in \Sigma} \left\{
            \max_{j \in \{0 , \ldots , l\}} \left\{
            \left|\left| \vph^{(j)}(z) - u^{(j)}(z)
            \right|\right|_{\cl(V^{\otimes j};W)}
            \right\} \right\}
            \leq \ep.
    \eeq
\end{enumerate}
\vskip 8pt
\noindent
Observe that if $C \leq \ep$ then the zero function in 
$\Span(\cf) \subset \Lip(\gamma,\Omega,W)$ is an approximation 
of $\vph$ satisfying 
\textbf{Compact Domain Approximation Conditions}
\ref{compact_approx_prop_1}, \ref{compact_approx_prop_2}, 
and \ref{compact_approx_prop_3}.
Thus we restrict to the meaningful setting in which 
$\ep \in (0,C)$.

In this case, fix a choice of $\ep_0 \in (0,\ep)$ and 
define a constant
$r = r(C,\gamma,\ep,\ep_0,l) > 0$ by 
(cf. \eqref{eq:pointwise_lip_sand_thm_de0})
\beq
    \label{eq:reduct_r_pointwise}
        r := \sup \left\{ \th > 0 ~:~
        2C \th^{\gamma - l} + \ep_0 e^{\th} 
        \leq \ep \right\} \in (0,1].
\eeq
Subsequently define an integer 
$\Lambda = \Lambda (\Omega,C,\gamma,\ep,\ep_0,l) \in \Z_{\geq 1}$
by 
\beq
    \label{eq:reduct_Lambda_pointwise}
        \Lambda := \min \left\{ a \in \Z_{\geq 1} ~:~
        \exists~ z_1 , \ldots , z_a \in \Omega 
        \text{ for which } \Omega \subset \bigcup_{s=1}^a 
        \ovB_V(z_s,r) \right\}.
\eeq 
A consequence of $\Omega \subset V$ being compact is that 
the integer $\Lambda \in \Z_{\geq 1}$ defined in 
\eqref{eq:reduct_Lambda_pointwise} is finite.
Choose a subset $\Sigma = \{ p_1 , \ldots , p_{\Lambda} \} 
\subset \Omega \subset V$ such that 
$\Omega \subset \cup_{s=1}^{\Lambda} \ovB_V(p_s,r)$.

It is evident that, by restriction, we have that 
$f_1 , \ldots , f_{\n} \in \Lip(\gamma,\Sigma,W)$ 
with, for every $i \in \{1, \ldots , \n\}$, the norm estimate
$||f_i||_{\Lip(\gamma,\Sigma,W)} \leq 
||f_i||_{\Lip(\gamma,\Omega,W)}$.
Thus we may consider the sparse approximation problem detailed
in Section \ref{sec:prob_formulation} for the 
$\ep$ there as $\ep_0$ here, 
finite subset $\Sigma \subset V$ there as the finite subset 
$\Sigma \subset \Omega \subset V$ here, the functions 
$f_1 , \ldots , f_{\n}$ there as the restriction of the functions
$f_1 , \ldots , f_{\n}$ here to the finite subset $\Sigma$ here, 
the scalars $A_1 , \ldots , A_{\n}$ there as 
$||f_1||_{\Lip(\gamma,\Omega,W)} , \ldots , 
||f_{\n} ||_{\Lip(\gamma,\Omega,W)}$ here, 
and the integer $l \in \{0, \ldots , k\}$ there as 
$k$ here.
We assume that we are able to find a solution $u \in \Span(\cf)$
to this sparse approximation problem satisfying 
\textbf{Finite Domain Approximation Conditions}
\ref{approx_prop_1}, \ref{approx_prop_2}, and 
\ref{approx_prop_3} from Section \ref{sec:prob_formulation}.

That is, we assume that we can find coefficients 
$b_1 , \ldots , b_{\n} \in \R$ such that 
(cf. \eqref{eq:u_support_goal})
\beq
    \label{eq:coeff_cardinality_bd}
        \# \{ b_s \neq 0 : s \in \{1, \ldots , \n\} \} 
        \leq M < \n,
\eeq 
with (cf. \eqref{eq:coeffs_sum_goal})
\beq
    \label{eq:coeff_sum_bd}
        \sum_{i=1}^{\n} |b_i| 
        || f_i ||_{\Lip(\gamma,\Omega,W)} 
        \leq 
        \sum_{i=1}^{\n} |a_i| 
        || f_i ||_{\Lip(\gamma,\Omega,W)}
        = C,
\eeq    
and such that $u := \sum_{i=1}^{\n} b_i f_i$ satisfies that
(cf. \eqref{eq:approx_poitwise_est_goal})
\beq
    \label{eq:approx_pointwise_est_bd}
        \max_{z \in \Sigma} \left\{ 
        \Lambda^k_{\vph-u}(z) \right\}
        \leq \ep_0.
\eeq
A particular consequence of 
\eqref{eq:coeff_sum_bd} is that $u \in \Lip(\gamma,\Omega,W)$
satisfies $||u||_{\Lip(\gamma,\Omega,W)} \leq C$.
Therefore, together with \eqref{eq:approx_pointwise_est_bd}, 
this yields the required hypotheses to appeal to the 
\emph{Pointwise Lipschitz Sandwich Theorem} 
\ref{thm:pointwise_lip_sand_thm}.
Indeed, observe that the constant $r$ defined in 
\eqref{eq:reduct_r_pointwise} coincides with the 
constant $\de_0$ defined in \eqref{eq:pointwise_lip_sand_thm_de0} 
for the constants $\ep$, $\ep_0$, $K_0$, $\gamma$, and $l$ 
there as $\ep$, $\ep_0$, $C$, $\gamma$, and $l$ here respectively.
Moreover, the definition of $\Lambda$ in 
\eqref{eq:reduct_Lambda_pointwise} ensures that 
$\Sigma$ is a $r$-cover of $\Omega$ in the sense 
required in \eqref{eq:pointwise_lip_sand_thm_B_cover_Sigma} 
for the subsets $B$ and $\m$ there as $\Sigma$ and $\Omega$
here respectively.
Consequently we are able to apply the 
\emph{Pointwise Lipschitz Sandwich Theorem} 
\ref{thm:pointwise_lip_sand_thm} with 
the subsets $B$ and $\m$ of that result as $\Sigma$ and 
$\Omega$ here respectively, the constants 
$\ep$, $\ep_0$, $K_0$, $\gamma$, and $l$ of that result as 
$\ep$, $\ep_0$, $C$, $\gamma$, and $l$ here respectively,
and the functions $\psi$ and $\phi$ of that result as
$\vph$ and $u$ here respectively.
By doing so we may conclude via the implication in 
\eqref{eq:pointwise_lip_sand_thm_imp} that 
$\sup_{z \in \Omega} \left\{ \Lambda^l_{\vph-u}(z) 
\right\} \leq \ep$.
Consequently, the function 
$u = \sum_{i=1}^{\n} b_i f_i \in \Span(\cf)$
is an approximation of $\vph$ on $\Omega$ that satisfies 
\textbf{Compact Domain Approximation Conditions} 
\ref{compact_approx_prop_1}, 
\ref{compact_approx_prop_2}, and 
\ref{compact_approx_prop_3}.

Moreover, at the cost of the involved constants becoming less
explicit, the \emph{Lipschitz Sandwich Theorem} 
\ref{thm:lip_sand_thm} from \cite{LM24} can be utilised to 
strengthen the sense in which the returned approximation 
$u$ is required to approximate $\vph$ throughout $\Omega$ in 
\textbf{Compact Domain Approximation Condition} 
\ref{compact_approx_prop_3}.
To be more precise, we now consider the following problem.
For a fixed $\eta \in (0,\gamma)$, find an element 
$u = \sum_{i=1}^{\n} b_i f_i \in \Span(\cf)$
satisfying the following properties.
\vskip 4pt
\noindent
\textbf{Compact Domain Lipschitz Approximation Conditions}
\begin{enumerate}[label=(\arabic*)]
    \item\label{compact_lip_eta_approx_prop_1} 
    There exists an integer $M \in \{1, \ldots , \n - 1\}$
    for which the set 
    \beq    
        \label{eq:compact_lip_eta_u_support_goal}
            \support(u) := \left\{ i \in \{1, \ldots , \n\} 
            : b_i \neq 0 \right\} 
            \qquad \text{satisfies that} \qquad 
            \# \left( \support(u) \right) \leq M.
    \eeq 
    \item\label{compact_lip_eta_approx_prop_2}
    The coefficients $b_1 , \ldots , b_{\n} \in \R$ satisfy 
    that 
    \beq
        \label{eq:compact_lip_eta_coeffs_sum_goal}
            \sum_{i=1}^{\n} |b_i| 
            || f_i ||_{\Lip(\gamma,\Omega,W)}
            \leq 
            \sum_{i=1}^{\n} |a_i| 
            || f_i ||_{\Lip(\gamma,\Omega,W)} 
            =: C.
    \eeq 
    \item\label{compact_lip_eta_approx_prop_3}
    The function $u \in \Span(\cf)$ well-approximates $\vph$ 
    throughout $\Omega$ in the $\Lip(\eta)$ sense that 
    \beq
        \label{eq:compact_approx_lip_eta_est_goal}
            \left|\left| \vph_{[q]} - u_{[q]} 
            \right|\right|_{\Lip(\eta,\Omega,W)}
            \leq \ep
    \eeq
    where $q \in \{0,\ldots,k\}$ is such that $\eta \in (q,q+1]$.
    Recall that if 
    $\phi = \left( \phi^{(0)} , \ldots , \phi^{(k)} 
    \right) \in \Lip(\gamma,\Omega,W)$ then 
    $\phi_{[q]} := \left( \phi^{(0)} , \ldots , \phi^{[q]} \right)$.
\end{enumerate}
\vskip 8pt
\noindent
Observe that if $C \leq \ep$ for the constant $C$ defined in 
\eqref{eq:compact_lip_eta_coeffs_sum_goal} 
then the zero function in 
$\Span(\cf) \subset \Lip(\gamma,\Omega,W)$ is an approximation 
of $\vph$ satisfying 
\textbf{Compact Domain Lipschitz Approximation Conditions}
\ref{compact_lip_eta_approx_prop_1}, 
\ref{compact_lip_eta_approx_prop_2}, 
and \ref{compact_lip_eta_approx_prop_3}.
Therefore we again restrict to the meaningful setting in which 
$\ep \in (0,C)$.

Retrieve the constants 
$r = r(C,\gamma,\eta,\ep) > 0$ and 
$\ep_0 = \ep_0(C,\gamma,\eta,\ep) > 0$ 
arising as $\de_0$ and $\ep_0$ respectively in the
\emph{Lipschitz Sandwich Theorem} \ref{thm:lip_sand_thm} 
for the constants $K_0$, $\gamma$, $\eta$, and $\ep$ there
as $C$, $\gamma$, $\eta$, and $\ep$ here respectively. 
Note we are not actually applying Theorem \ref{thm:lip_sand_thm}, 
but simply retrieving constants in preparation for its future 
application.
Subsequently define an integer 
$\Lambda = \Lambda (\Omega,C,\gamma,\eta,\ep) \in \Z_{\geq 1}$ 
by 
\beq
    \label{eq:reduct_Lambda_lip_eta}
        \Lambda := \min \left\{ a \in \Z_{\geq 1} 
        ~:~ \exists ~ z_1 , \ldots , z_a \in \Omega 
        \text{ for which } 
        \Omega \subset \bigcup_{s=1}^a \ovB_V(z_s,r) 
        \right\}.
\eeq 
The subset $\Omega \subset V$ being compact ensures that 
$\Lambda$ defined in \eqref{eq:reduct_Lambda_lip_eta} 
is finite. Choose a subset 
$\Sigma = \left\{ p_1 , \ldots , p_{\Lambda} \right\} 
\subset \Omega \subset V$ such that 
$\Omega \subset \cup_{s=1}^{\Lambda} \ovB_V(p_s,r)$.

As before, we assume that we are able to find a solution 
$u \in \Span(\cf)$ to this sparse approximation problem 
for the finite subset $\Sigma$ satisfying 
\textbf{Finite Domain Approximation Conditions}
\ref{approx_prop_1}, \ref{approx_prop_2}, and 
\ref{approx_prop_3} from Section \ref{sec:prob_formulation}.
That is, we assume that we can find coefficients 
$b_1 , \ldots , b_{\n} \in \R$ satisfying 
\eqref{eq:coeff_cardinality_bd} and \eqref{eq:coeff_sum_bd}, 
and such that $u := \sum_{i=1}^{\n} b_i f_i$ satisfies 
\eqref{eq:approx_pointwise_est_bd} for the constant 
$\ep_0$ retrieved from the 
\emph{Lipschitz Sandwich Theorem} \ref{thm:lip_sand_thm} above.
It is again a consequence of 
\eqref{eq:coeff_sum_bd} that 
$u \in \Lip(\gamma,\Omega,W)$ satisfies that 
$||u||_{\Lip(\gamma,\Omega,W)} \leq C$, and so the pointwise 
estimates \eqref{eq:approx_pointwise_est_bd} now provide the 
required hypothesis to appeal to the 
\emph{Lipschitz Sandwich Theorem} \ref{thm:lip_sand_thm}.

Indeed we have specified that the constants $r$ and $\ep_0$ 
here coincide with the constants $\de_0$ and $\ep_0$ 
arising in the \emph{Lipschitz Sandwich Theorem} 
\ref{thm:lip_sand_thm}, 
for the choices of the constants $K_0$, $\gamma$, $\eta$ and 
$\ep$ there as $C$, $\gamma$, $\eta$ and $\ep$ here, respectively.
Further, the subset $\Sigma$ is chosen to be a $r$-cover of 
$\Omega$ in the sense required in 
\eqref{eq:lip_sand_thm_B_cover_sigma}.
Consequently we are able to apply the 
\emph{Lipschitz Sandwich Theorem} for the choices of the subsets
$B$ and $\m$ there as $\Sigma$ and $\Omega$ here respectively,
the choices of the constants $K_0$, $\gamma$, $\eta$, and $\ep$ 
there are $C$, $\gamma$, $\eta$, and $\ep$ here respectively, 
and for the choices of the functions $\psi$ and $\phi$ 
there as $\vph$ and $u$ here respectively.
By doing so we may conclude via the implication 
\eqref{eq:lip_sand_thm_imp} that 
$\left|\left| \vph_{[q]} - u_{[q]} \right|\right|_{ 
\Lip(\eta,\Omega,W)} \leq \ep$.
Consequently, the function 
$u = \sum_{i=1}^{\n} b_i f_i \in \Span(\cf)$
is an approximation of $\vph$ on $\Omega$ that satisfies 
\textbf{Compact Domain Lipschitz Approximation Conditions} 
\ref{compact_lip_eta_approx_prop_1}, 
\ref{compact_lip_eta_approx_prop_2}, and 
\ref{compact_lip_eta_approx_prop_3}.

\section{Pointwise Values Via Linear Functionals}
\label{sec:pointwise_via_lin_funcs}
In this section, given a closed 
subset $\m \subset V$ and a point $z \in \m$, 
we define a finite subset 
$\Tau_{z,k} \subset \Lip(\gamma,\m,W)^{\ast}$ of the dual-space
$\Lip(\gamma,\m,W)^{\ast}$ of bounded linear functionals 
$\Lip(\gamma,\m,W) \to \R$ such that for any 
$\psi \in \Lip(\gamma,\m,W)$ the value of $\psi$ at $z$ 
is determined, up to a quantifiable error, by the values 
$\sigma(\psi)$ for the linear functionals $\sigma \in \Tau_{z,k}$.
We begin by introducing the following particularly useful 
quantities.

For integers $a,b \in \Z_{\geq 0}$, define $\be(a,b)$ by
\beq
    \label{eq:beta_ab_def_not_sec}
        \be(a,b) := 
        \left( 
        \begin{array}{c}
            a+b-1 \\
            b
        \end{array}
        \right)
        = \frac{(a+b-1)!}{(a-1)!b!}
        =
        \twopartdef{1}
        {b=0}
        {\frac{1}{b!}\prod_{s=0}^{b-1}(a+s)}
        {b \geq 1.}
\eeq
Further, for integers $a,b,i,j \in \Z_{\geq 0}$ define 
$D(a,b)$ and $Q(i,j,a,b)$ by
\beq
    \label{eq:D_ab_Q_ijab_def_not_sec}
        D(a,b) := \sum_{l=0}^b \be(a,l)
        \qquad \text{and} \qquad
        Q(i,j,a,b) := 1 + ijD(a,b)
        = 
        1 + ij \sum_{l=0}^b \be(a,l)
\eeq
respectively.

Let $d \in \Z_{\geq 1}$ and $V$ be a real $d$-dimensional 
Banach space with $\m \subset V$ a closed subset.
Assume that the tensor powers of $V$ are equipped with 
admissible norms (cf. Definition \ref{admissible_tensor_norm}).
Fix a choice of a unit $V$-norm 
basis $v_1 , \ldots , v_d$ of $V$.
Then, recalling that the tensor powers of $V$ are 
assumed to be equipped with admissible tensor norms in 
the sense of definition \ref{admissible_tensor_norm},
for each $n \in \Z_{\geq 1}$ the set
\beq
    \label{eq:V_ten_prod_basis}
        \cv_n := \left\{
        v_{l_1} \otimes \ldots \otimes v_{l_n} :
        (l_1 , \ldots , l_n) \in
        \{ 1 ,\ldots , d\}^n 
        \right\}
\eeq 
is a unit $V^{\otimes n}$-norm basis for $V^{\otimes n}$. 
Consequently, for each 
$j \in \{1 , \ldots , k\}$ a symmetric $j$-linear form 
$\bB \in \cl( V^{\otimes j} ;W)$ is determined
by its action on the subset $\cv^{\ord}_j \subset \cv_j$
defined by
\beq
    \label{eq:symm_j_lin_form_determine_set}
        \cv^{\ord}_j := \left\{
        v_{l_1} \otimes \ldots \otimes v_{l_j} :
        (l_1 , \ldots , l_j) \in
        \{ 1 ,\ldots , d\}^j 
        \text{ such that }
        l_1 \leq \ldots \leq l_j 
        \right\} \subset \cv_j.
\eeq
That is, the collection of values 
$\bB \left[ 
v_{l_1} \otimes \ldots \otimes v_{l_j} \right]$
for every 
$(l_1 , \ldots , l_j ) \in \{ 1 , \ldots , d\}^j$
with $l_1 \leq \ldots \leq l_j$
determine the value $\bB[v]$
for every $v \in V^{\otimes j}$.
The cardinality 
of $\cv^{\ord}_j$ is given by $\be(d,j)$ defined
in \eqref{eq:beta_ab_def_not_sec}.
The following lemma records how the bases 
$\cv_n$ for the tensor powers $V^{\otimes n}$
interact with a choice of admissible tensor norms 
for the tensor powers $V^{\otimes n}$ (cf. Definition 
\ref{admissible_tensor_norm}).

\begin{lemma}
\label{lemma:admissible_ten_norm_coeff_bound}
Assume that $V$ is a Banach space of dimension $d \in \Z_{\geq 1}$
and that the tensor powers of $V$ are equipped with admissible
tensor norms (cf. Definition \ref{admissible_tensor_norm}).
Suppose that $v_1 , \ldots , v_d \in V$ is a basis of $V$ 
normalised so that $||v_1||_V = \ldots = || v_d ||_V = 1$.
Let $m \in \Z_{\geq 1}$ and consider the basis $\cv_m$ of 
$V^{\otimes m}$ given by 
\beq
    \label{eq:ten_norm_coeff_basis_m}
        \cv_m := \left\{
        v_{l_1} \otimes \ldots \otimes v_{l_m} :
        (l_1 , \ldots , l_m) \in
        \{ 1 ,\ldots , d\}^m 
        \right\}.
\eeq 
Then for any $v \in V^{\otimes m}$ there exist coefficients 
$ \left\{ C_{j_1 \dots j_m} : (j_1 , \ldots , j_m) \in 
\{ 1 , \ldots , d \}^m \right\}$ for which 
\beq 
    \label{basis_expansion}
        v = \sum_{j_1 = 1}^d \ldots \sum_{j_m = 1}^d 
        C_{j_1 \dots j_m} v_{j_1} \otimes \ldots \otimes v_{j_m}.
\eeq
Moreover, if $A \geq 0$ and $||v||_{V^{\otimes m}} \leq A$, then
\beq
    \label{coeff_sum_bound}
        \sum_{j_1 = 1}^{d} \dots \sum_{j_m = 1}^d
        \left| C_{j_1 \dots j_m} \right|
        \leq A d^m.
\eeq
\end{lemma}

\begin{proof}[Proof of Lemma 
\ref{lemma:admissible_ten_norm_coeff_bound}]
Assume that $V$ is a Banach space of dimension $d \in \Z_{\geq 1}$
and that the tensor powers of $V$ are equipped with admissible
tensor norms (cf. Definition \ref{admissible_tensor_norm}).
Suppose that $v_1 , \ldots , v_d \in V$ is a basis of $V$ 
normalised so that $||v_1||_V = \ldots = || v_d ||_V = 1$.
We first verify the claims made in 
\eqref{basis_expansion} and \eqref{coeff_sum_bound} 
in the case that $m=1$.

Let $v \in V$; the existence of coefficients 
$C_1 , \ldots , C_d \in \R$ such that 
$v = \sum_{j=1}^d C_j v_j$ is an immediate consequence of 
$v_1 , \ldots , v_d \in V$ being a basis of $V$.
Now assume $A \geq 0$ and that $||v||_V \leq A$.
It is convenient to consider the dual-basis 
$v_1^{\ast} , \ldots , v^{\ast}_d \in V^{\ast}$ 
corresponding to $v_1, \ldots , v_d \in V$. 
This means that whenever $i,j \in \{1, \ldots , d\}$
we have $v_j^{\ast}(v_i) = \de_{ij}$, and that 
for every $j \in \{1, \ldots , d\}$ we have 
$||v_j^{\ast}||_{V^{\ast}} = 1$.
Here, for $\sigma \in V^{\ast}$, we have 
$||\sigma||_{V^{\ast}} = \sup \left\{ |\sigma(u)| : u \in V 
\text{ and } ||u||_V \leq 1 \right\}$.

Being a finite dimensional Banach space means $V$ is reflexive.
A particular consequence of this is, for any $u \in V$, that
$||u||_V = || u ||_{V^{\ast\ast}}$ where we view 
$u$ as an element in $V^{\ast\ast}$ by defining its action 
on an element $\sigma \in V^{\ast}$ by 
$u(\sigma) := \sigma(u)$.
Further, since $||u||_{V^{\ast \ast}} = 
\sup \left\{ |u(\sigma)| : \sigma \in V^{\ast} \text{ and } 
||\sigma||_{V^{\ast}} \leq 1 \right\}$ , 
we have, for every $j \in \{1, \ldots , d\}$, that
$|C_j| = \left| v^{\ast}_j (v) \right| \leq ||v||_{V^{\ast \ast}} 
= ||v||_V \leq A$.
It is now immediate that $\sum_{j=1}^d |C_j| \leq Ad$, 
verifying that \eqref{coeff_sum_bound} is valid in the case 
that $m=1$.

It remains only to verify that \eqref{basis_expansion} and 
\eqref{coeff_sum_bound} are valid in the case that 
$m \in \Z_{\geq 2}$. 
For this purpose fix $m \in \Z_{\geq 2}$ and  
let $v \in V^{\otimes m}$.
The existence of coefficients 
$ \left\{ C_{j_1 \dots j_m} : (j_1 , \ldots , j_m) \in 
\{ 1 , \ldots , d \}^m \right\}$ for which 
\eqref{basis_expansion} is true is a consequence of the fact
that the set 
$\cv_m := \left\{ v_{j_1} \otimes \ldots \otimes v_{j_m} : 
(j_1 , \ldots , j_m) \in \{ 1 , \ldots , d \}^m \right\}$
defined in \eqref{eq:ten_norm_coeff_basis_m}
is a basis for $V^{\otimes m}$.
Moreover, since the tensor powers of $V$ are equipped with 
admissible tensor norms, we have, for every 
$(j_1 , \ldots , j_m) \in \{ 1 , \ldots , d \}^m$, that 
(cf. Definition \ref{admissible_tensor_norm})
\beq
	\label{unit_Vm_norm}
		\left|\left| v_{j_1} \otimes \ldots \otimes v_{j_m}
		\right|\right|_{V^{\otimes m}}
		= 
		|| v_{j_1} ||_V \dots || v_{j_m} ||_V = 1.
\eeq
Now assume that $A \geq 0$ and that $||v||_{V^{\otimes m}} \leq A$.
A consequence of the tensor powers of $V$ being equipped with 
admissible norms is that for any $u \in V^{\otimes m}$ we have
(cf. Proposition 2.1 in \cite{Rya02})
\beq
    \label{inj_norm_proj_norm_bound}
        || u ||_{\inj(V^{\otimes m})} \leq 
        || u ||_{V^{\otimes m}}  \leq 
        || u ||_{\proj(V^{\otimes m})}.
\eeq
Here $|| \cdot ||_{\inj(V^{\otimes m})}$ denotes the 
\emph{injective cross norm} on $V^{\otimes m}$ defined by 
\beq
	\label{inj_norm_def}
		|| u ||_{\inj(V^{\otimes m})}
		:=
		\sup \left\{ 
		\left| \phi_1 \otimes \ldots \otimes \phi_m (u) \right|
		:
		\phi_1 , \ldots , \phi_m \in V^{\ast} 
		\text{ with }
		||\phi_1||_{V^{\ast}} = \ldots = ||\phi_m||_{V^{\ast}} = 1
		\right\},
\eeq
and $|| \cdot ||_{\proj(V^{\otimes m})}$ denotes the 
\emph{projective cross norm} on $V^{\otimes m}$ defined by 
\beq
	\label{proj_norm_def}
		|| u ||_{\proj(V^{\otimes m})}
		:=
		\inf \left\{ 
		\sum_{i} ||a_{1,i}||_V \ldots || a_{m,i} ||_V 
		:
		u = \sum_{i} a_{1,i} \otimes \ldots \otimes a_{m,i}
		\right\}.
\eeq
If we let $v_1^{\ast} , \ldots , v^{\ast}_d \in V^{\ast}$ 
denote the dual-basis to $v_1 , \ldots , v_d \in V$ then, 
for each $(j_1 , \ldots , j_m) \in \{ 1 , \ldots , d \}^m$, 
it follows that 
\beq
	\label{single_coeff_ub}
		A 
		\geq 
		|| v ||_{V^{\otimes m}} 
		\stackrel{\eqref{inj_norm_proj_norm_bound}}{\geq} 
		|| v ||_{\inj(V^{\otimes m})}
		\stackrel{\eqref{inj_norm_def}}
		{\geq} 
		\left| v_{j_1}^{\ast} \otimes \ldots \otimes v_{j_m}^{\ast}
		(v) \right|
		\stackrel{\eqref{basis_expansion}}{=}
		\left| C_{j_1 \dots j_m} \right|.
\eeq
Summing over 
$(j_1 , \ldots , j_m) \in \{ 1 , \ldots , d \}^m$
in \eqref{single_coeff_ub} yields \eqref{coeff_sum_bound}, 
which completes the proof of Lemma 
\ref{lemma:admissible_ten_norm_coeff_bound}.
\end{proof}
\vskip 4pt
\noindent
Let $c \in \Z_{\geq 1}$ and $W$ be a real $c$-dimensional 
Banach space. Fix a choice of a unit $W$-norm 
basis $w_1 , \ldots , w_c$ of $W$, and consider 
its corresponding dual basis 
$w_1^{\ast} , \ldots , w_c^{\ast}$ of $W^{\ast}$.
Then whenever $a,b \in \{1, \ldots , c\}$ we have
\beq
    \label{eq:W_dual_basis_identity}
        w_a^{\ast} (w_b) := 
        \twopartdef{1}{a=b}{0}{a \neq b.}
\eeq 
A particularly useful consequence 
is that any $w \in W$ can be decomposed as 
\beq
    \label{eq:W_dual_basis_decomp}
        w = \sum_{s=1}^c w^{\ast}_s(w) w_s.
\eeq
We use the these bases for $V$, $W$, and the tensor powers of 
$V$ to define, for each point in $\m$, an associated 
collection of bounded linear functionals in the dual-space
$\Lip(\gamma,\m,W)^{\ast}$.

For $p \in \m$ and $s \in \{1, \ldots , c\}$ define 
$\de_{p,0,s} \in \Lip(\gamma,\m,W)^{\ast}$ by 
setting, for 
$\phi = \left( \phi^{(0)} , \ldots , \phi^{(k)} \right) 
\in \Lip(\gamma,\m,W)$, 
\beq
    \label{eq:de_p_0_s}
        \de_{p,0,s}(\phi) := 
        w_s^{\ast} \left( \phi^{(0)}(p) \right).
\eeq
For $p \in \m$, $j \in \{1, \ldots , k\}$, 
$v \in V^{\otimes j}$, and $s \in \{1, \ldots , c\}$ define
$\de_{p,j,v,s} \in \Lip(\gamma,\m,W)^{\ast}$ by setting, 
for
$\phi = \left( \phi^{(0)} , \ldots , \phi^{(k)} \right) 
\in \Lip(\gamma,\m,W)$,
\beq
    \label{eq:de_p_j_v_s}
        \de_{p,j,v,s}(\phi) := 
        w_s^{\ast}\left(\phi^{(j)}(p)[v]\right).
\eeq
Then for a given $j \in \{0 , \ldots , k\}$ define 
$\m_{p,j} \subset \Lip(\gamma,\m,W)^{\ast}$ by
\beq
    \label{eq:m_p_j_lin_funcs_not_sec}
        \m_{p,j} := \twopartdef{
        \left\{ \de_{p,0,s} : s \in \{1, \ldots , c\} \right\}
        }
        {j=0}
        {
        \left\{ \de_{p,j,v,s} : v \in \cv_j^{\ord} 
        \text{ and } s \in \{ 1 , \ldots , c \} \right\}
        }
        {j \geq 1.}
\eeq
It follows from \eqref{eq:beta_ab_def_not_sec} 
and \eqref{eq:m_p_j_lin_funcs_not_sec} that 
\beq
    \label{eq:not_sec_m_pj_card}
        \# \left( \m_{p,j} \right) 
        = c \be(d,j).
\eeq
For a given $l \in \{0, \ldots , k\}$, define
$\Tau_{p,l} \subset \Lip(\gamma,\m,W)^{\ast}$ by
\beq
    \label{eq:not_sec_Tau_pl}
        \Tau_{p,l} := \bigcup_{j=0}^l \m_{p,j}.
\eeq
It follows from \eqref{eq:not_sec_Tau_pl} that, since
the union in \eqref{eq:not_sec_Tau_pl} is disjoint we have
\beq
    \label{eq:Tau_pl_card_not_sec}
        \# \left( \Tau_{p,l} \right) 
        = \sum_{j=0}^l \# \left( \m_{p,j} \right)
        \stackrel{
        \eqref{eq:not_sec_m_pj_card}
        }{=} c \sum_{j=0}^l \be(d,j)
        \stackrel{
        \eqref{eq:D_ab_Q_ijab_def_not_sec}
        }{=}
        c D(d,l).
\eeq
Finally, for a given $l \in \{0, \ldots , k\}$, 
define $\m_l^{\ast} \subset \Lip(\gamma,\m,W)$ by
\beq    
    \label{eq:Sigma_l_star_not_sec}
        \m_l^{\ast} := \bigcup_{z \in \m}
        \Tau_{z, l}.
\eeq
A consequence of the union in \eqref{eq:Sigma_l_star_not_sec} 
being disjoint is that when $\m \subset V$ is finite with 
cardinality $\Lambda \in \Z_{\geq 1}$ we have
\beq
    \label{eq:Sigma_l_star_card_not_sec}
        \# \left( \m_l^{\ast} \right)
        = \sum_{z \in \m} \# \left( \Tau_{z,l} \right)
        \stackrel{
        \eqref{eq:Tau_pl_card_not_sec}
        }{=}
        \sum_{s=1}^{\Lambda} c D(d,l) 
        = c D(d,l) \Lambda.
\eeq
In the remainder of this section we establish several 
lemmata in order to
provide a precise quantified meaning 
to the statement that the value of a $\Lip(\gamma,\m,W)$ 
function $\phi \in \Lip(\gamma,\m,W)$
at a point $p \in \m$ is determined by the set 
$\left\{ \sigma(\phi) : \sigma \in \Tau_{p,k} \right\} 
\subset \R$.
The first result records estimates relating the 
values $|\de_{p,0,s}(\phi)|$ for $s \in \{1, \ldots ,c\}$ 
to the value 
$\left|\left| \phi^{(0)} \right|\right|_W$ and, for 
each $j \in \{1, \ldots , k\}$, the values 
$| \de_{p,j,v,s}(\phi) |$ for $s \in \{1, \ldots , c\}$ 
and $v \in V^{\otimes j}$ to the value 
$\left|\left| \phi^{(j)}(v) \right|\right|_{\cl(V^{\otimes j};W)}$.
The precise result is the following lemma.

\begin{lemma}[Linear Functional Inequalities]
\label{lemma:delta_funcs_estimates}
Let $c,d \in \Z_{\geq 1}$ and assume that $V$ and $W$ are 
real Banach spaces of dimensions $d$ and $c$ respectively.
Let $\m \subset V$ be closed with $p \in \m$ 
and assume that the tensor 
powers of $V$ are all equipped with admissible norms 
(cf. Definition \ref{admissible_tensor_norm}).
Let $v_1 , \ldots , v_d \in V$ be a unit $V$-norm basis of $V$. 
Further, let $w_1, \ldots , w_c \in W$ be a unit $W$-norm 
basis of $W$, and let 
$w_1^{\ast} , \ldots , w_c^{\ast} \in W^{\ast}$
be the corresponding dual basis of $W^{\ast}$.
Let $\gamma > 0$ with $k \in \Z_{\geq 0}$ such that 
$\gamma \in (k,k+1]$.
Given $s \in \{1 , \ldots , c\}$, $j \in \{1, \ldots , k\}$,
and $v \in V^{\otimes j}$, define linear functionals 
$\de_{p,0,s} , \de_{p,j,v,s} \in \Lip(\gamma,\m,W)^{\ast}$ 
by, for
$F = \left( F^{(0)} , \ldots , F^{(k)} \right) 
\in \Lip(\gamma,\m,W)$, setting
(cf. \eqref{eq:de_p_0_s} and \eqref{eq:de_p_j_v_s}).
\beq
    \label{eq:delta_funcs_def_lemma}
        \de_{p,0,s}(F) := w_s^{\ast} \left( F^{(0)}(p) \right)
        \qquad \text{and} \qquad
        \de_{p,j,v,s}(F) := w_s^{\ast} \left( F^{(j)}(p)[v] \right).
\eeq
Suppose that $\phi = \left( \phi^{(0)} , \ldots , \phi^{(k)} 
\right) \in \Lip(\gamma,\m,W)$. Then, for any 
$s \in \{1, \ldots , c\}$, we have that
\beq
    \label{eq:de_p_0_s_bounds}
        \left| \de_{p,0,s}(\phi) \right|
        \leq 
        \left|\left| \phi^{(0)} \right|\right|_W
        \leq 
        \sum_{a=1}^c
        \left| \de_{p,0,a}(\phi) \right|.
\eeq 
Now suppose that $k \geq 1$ and fix $j \in \{1, \ldots ,k\}$.
Then, for any $v \in V^{\otimes j}$
and any $s \in \{1, \ldots , c\}$, we have that
\beq
    \label{eq:de_p,j,v,s(phi)_bounds}
        \left|\de_{p,j,v,s}(\phi) \right|
        \leq
        \left|\left| \phi^{(j)}(p)[v] \right|\right|_W 
        \leq 
        \sum_{a=1}^c 
        \left| \de_{p,j,v,a}(\phi) \right|.
\eeq
\end{lemma}

\begin{proof}[Proof of Lemma 
\ref{lemma:delta_funcs_estimates}]
Let $c,d \in \Z_{\geq 1}$ and assume that $V$ and $W$ are 
real Banach spaces of dimensions $d$ and $c$ respectively.
Let $\m \subset V$ be closed with $p \in \m$ 
and assume that the tensor 
powers of $V$ are all equipped with admissible norms 
(cf. Definition \ref{admissible_tensor_norm}).
Let $v_1 , \ldots , v_d \in V$ be a unit $V$-norm basis of $V$. 
Further, let $w_1, \ldots , w_c \in W$ be a unit $W$-norm 
basis of $W$, and let 
$w_1^{\ast} , \ldots , w_c^{\ast} \in W^{\ast}$
be the corresponding dual basis of $W^{\ast}$.
Let $\gamma > 0$ with $k \in \Z_{\geq 0}$ such that 
$\gamma \in (k,k+1]$.
Given $s \in \{1 , \ldots , c\}$, $j \in \{1, \ldots , k\}$,
and $v \in V^{\otimes j}$, define linear functionals 
$\de_{p,0,s} , \de_{p,j,v,s} \in \Lip(\gamma,\m,W)^{\ast}$ 
by, for
$F = \left( F^{(0)} , \ldots , F^{(k)} \right) 
\in \Lip(\gamma,\m,W)$, setting
(cf. \eqref{eq:delta_funcs_def_lemma}).
\beq
    \label{eq:delta_funcs_def_lemma_pf}
        \de_{p,0,s}(F) := w_s^{\ast} \left( F^{(0)}(p) \right)
        \qquad \text{and} \qquad
        \de_{p,j,v,s}(F) := w_s^{\ast} \left( F^{(j)}(p)[v] \right).
\eeq
Suppose that $\phi = \left( \phi^{(0)} , \ldots , \phi^{(k)} 
\right) \in \Lip(\gamma,\m,W)$ and fix 
$s \in \{1, \ldots , c\}$.
Then, since $||w_1^{\ast}||_{W^{\ast}} = \ldots = 
||w_c^{\ast}||_{W^{\ast}} = 1$, we have 
\begin{multline}
    \label{eq:de_0_ests}
        \left| \de_{p,0,s}(\phi) \right|
        =
        \left| w_s^{\ast} \left( \phi^{(0)}(p) \right) \right|
        \leq
        \left|\left| \phi^{(0)}(p) \right|\right|_W
        = \\
        \left|\left| \sum_{a=1}^c
        w_a^{\ast} \left( \phi^{(0)}(p) \right) w_a 
        \right|\right|_W
        \leq 
        \sum_{a=1}^c \left| 
        w_a^{\ast} \left( \phi^{(0)}(p) \right) \right|
        = 
        \sum_{a=1}^c 
        \left| \de_{p,0,a}(\phi) \right|.
\end{multline}
The arbitrariness of $s \in \{1, \ldots , c\}$ means that
the estimates established in \eqref{eq:de_0_ests} are 
precisely those claimed in \eqref{eq:de_p_0_s_bounds}.

Now suppose that $k \geq 1$ and fix $j \in \{1, \ldots , k\}$.
Fix $v \in V^{\otimes j}$ and $s \in \{1, \ldots , c\}$.
Then, again since $||w_1^{\ast}||_{W^{\ast}} = \ldots = 
||w_c^{\ast}||_{W^{\ast}} = 1$, we have 
\begin{multline}
    \label{eq:de_j_ests}
        \left| \de_{p,j,v,s}(\phi) \right|
        =
        \left| w_s^{\ast} \left( \phi^{(j)}(p)[v] \right) \right|
        \leq 
        \left|\left| \phi^{(j)}(p)[v] \right|\right|_W = \\
        \left|\left| \sum_{a=1}^c 
        w_a^{\ast} \left( \phi^{(j)}(p)[v] \right) w_a 
        \right|\right|_W 
        \leq
        \sum_{a=1}^c 
        \left| w_a^{\ast} \left( \phi^{(j)}(p)[v] \right) \right|
        = 
        \sum_{a=1}^c \left| \de_{p,j,v,a}(\phi) \right|.
\end{multline}
The arbitrariness of $j \in \{1, \ldots , k\}$, 
$v \in V^{\otimes j}$, and $s \in \{1, \ldots , c\}$ mean 
that the estimates established in \eqref{eq:de_j_ests} are 
precisely those claimed in 
\eqref{eq:de_p,j,v,s(phi)_bounds}.
This completes the proof of Lemma 
\ref{lemma:delta_funcs_estimates}.
\end{proof}
\vskip 4pt 
\noindent
We can use the inequalities established in Lemma 
\ref{lemma:delta_funcs_estimates} to establish the 
following dual norm estimates for linear functionals 
in $\Tau_{p,l}$ (cf. \eqref{eq:not_sec_Tau_pl}).

\begin{lemma}[Dual Norm Estimates]
\label{lemma:dual_norm_ests}
Let $c,d \in \Z_{\geq 1}$ and assume that $V$ and $W$ are 
real Banach spaces of dimensions $d$ and $c$ respectively.
Let $\m \subset V$ be closed with $p \in \m$ 
and assume that the tensor 
powers of $V$ are all equipped with admissible norms 
(cf. Definition \ref{admissible_tensor_norm}).
Let $v_1 , \ldots , v_d \in V$ be a unit $V$-norm basis 
of $V$. 
For each integer $n \in \Z_{\geq 1}$ define 
(cf. \eqref{eq:V_ten_prod_basis})
\beq
    \label{eq:dual_norm_ests_lemma_V_ten_prod_basis}
        \cv_n := \left\{
        v_{l_1} \otimes \ldots \otimes v_{l_n} :
        (l_1 , \ldots , l_n) \in
        \{ 1 ,\ldots , d\}^n 
        \right\} \subset V^{\otimes n}
\eeq 
and (cf. \eqref{eq:symm_j_lin_form_determine_set})
\beq
    \label{eq:dual_norm_ests_lemma_V_n^ord}
        \cv_n^{\ord} := \left\{ 
        v_{l_1} \otimes \ldots \otimes v_{l_n} :
        (l_1 , \ldots , l_n) \in \{1 , \ldots , d\}^n
        \text{ such that }
        l_1 \leq \ldots \leq l_n \right\} 
        \subset \cv_n \subset V^{\otimes n}.
\eeq
Further, let $w_1, \ldots , w_c \in W$ be a unit $W$-norm 
basis of $W$, and let 
$w_1^{\ast} , \ldots , w_c^{\ast} \in W^{\ast}$
be the corresponding dual basis of $W^{\ast}$.

Let $\gamma > 0$ with $k \in \Z_{\geq 0}$ such that 
$\gamma \in (k,k+1]$.
For a given $j \in \{0, \ldots , k\}$, define 
$\m_{p,j} \subset \Lip(\gamma,\m,W)^{\ast}$ by
\beq
    \label{eq:dual_norm_ests_lemma_m_p_j}
        \m_{p,j} := \twopartdef{
        \left\{ \de_{p,0,s} : s \in \{1, \ldots , c\} \right\}
        }
        {j=0}
        {
        \left\{ \de_{p,j,v,s} : v \in \cv_j^{\ord} 
        \text{ and } s \in \{ 1 , \ldots , c \} \right\}
        }
        {j \geq 1.}
\eeq
For every $s \in \{1, \ldots , c\}$, 
the linear functional
$\de_{p,0,s} \in \Lip(\gamma,\m,W)^{\ast}$ is 
defined, for 
$F = \left( F^{(0)} , \ldots , F^{(k)} \right) 
\in \Lip(\gamma,\m,W)$, by 
$\de_{p,0,s}(F) := w_s^{\ast} \left( F^{(0)}(p) \right)$
(cf. \eqref{eq:de_p_0_s}).
For every $s \in \{1 , \ldots , c\}$, $j \in \{1, \ldots , k\}$,
and $v \in \cv_j^{\ord}$, the linear functional 
$\de_{p,j,v,s} \in \Lip(\gamma,\m,W)^{\ast}$ is 
defined, for 
$F = \left( F^{(0)} , \ldots , F^{(k)} \right) 
\in \Lip(\gamma,\m,W)$, by 
$\de_{p,j,v,s}(F) := w_s^{\ast} \left( F^{(j)}(p)[v] \right)$
(cf. \eqref{eq:de_p_j_v_s}).
For a given $l \in \{0, \ldots , k\}$ define 
$\Tau_{p,l} \subset \Lip(\gamma,\m,W)^{\ast}$ by 
(cf. \eqref{eq:not_sec_Tau_pl})
\beq
    \label{eq:dual_norm_ests_lemma_Tau_pl}
        \Tau_{p,l} := \bigcup_{j=0}^l \m_{p,j}.
\eeq
Then for every $\sigma \in \Tau_{p,l}$ and any 
$\phi \in \Lip(\gamma,\m,W)$ we have 
\beq
    \label{eq:dual_norm_ests_lemma_pointwise_conc}
        | \sigma (\phi) | \leq 
        \Lambda^{l}_{\phi}(p) \leq 
        ||\phi||_{\Lip(\gamma,\m,W)}.
\eeq 
Consequently, for every $\sigma \in \Tau_{p,l}$, we have
\beq
    \label{eq:dual_norm_ests_lemma_conc}
        || \sigma ||_{\Lip(\gamma,\m,W)^{\ast}} 
        \leq 1 
        \qquad \text{so that} \qquad 
        \Tau_{p,l} \subset  
        \ovB_{\Lip(\gamma,\m,W)^{\ast}} (0,1).
\eeq
\end{lemma}

\begin{proof}[Proof of Lemma \ref{lemma:dual_norm_ests}]
Let $c,d \in \Z_{\geq 1}$ and assume that $V$ and $W$ are 
real Banach spaces of dimensions $d$ and $c$ respectively.
Let $\m \subset V$ be closed with $p \in \m$ 
and assume that the tensor 
powers of $V$ are all equipped with admissible norms 
(cf. Definition \ref{admissible_tensor_norm}).
Let $v_1 , \ldots , v_d \in V$ be a unit $V$-norm basis of $V$. 
For each integer $n \in \Z_{\geq 1}$ define 
(cf. \eqref{eq:dual_norm_ests_lemma_V_ten_prod_basis})
\beq
    \label{eq:dual_norm_ests_lemma_pf_V_ten_prod_basis}
        \cv_n := \left\{
        v_{l_1} \otimes \ldots \otimes v_{l_n} :
        (l_1 , \ldots , l_n) \in
        \{ 1 ,\ldots , d\}^n 
        \right\} \subset V^{\otimes n}
\eeq 
and (cf. \eqref{eq:dual_norm_ests_lemma_V_n^ord})
\beq
    \label{eq:dual_norm_ests_lemma_pf_V_n^ord}
        \cv_n^{\ord} := \left\{ 
        v_{l_1} \otimes \ldots \otimes v_{l_n} :
        (l_1 , \ldots , l_n) \in \{1 , \ldots , d\}^n
        \text{ such that }
        l_1 \leq \ldots \leq l_n \right\} 
        \subset \cv_n \subset V^{\otimes n}.
\eeq
Further, let $w_1, \ldots , w_c \in W$ be a unit $W$-norm 
basis of $W$, and let 
$w_1^{\ast} , \ldots , w_c^{\ast} \in W^{\ast}$
be the corresponding dual basis of $W^{\ast}$.

Let $\gamma > 0$ with $k \in \Z_{\geq 0}$ such that 
$\gamma \in (k,k+1]$.
For a given $j \in \{0, \ldots , k\}$, define 
$\m_{p,j} \subset \Lip(\gamma,\m,W)^{\ast}$ by
(cf. \eqref{eq:dual_norm_ests_lemma_m_p_j})
\beq
    \label{eq:dual_norm_ests_lemma_pf_m_p_j}
        \m_{p,j} := \twopartdef{
        \left\{ \de_{p,0,s} : s \in \{1, \ldots , c\} \right\}
        }
        {j=0}
        {
        \left\{ \de_{p,j,v,s} : v \in \cv_j^{\ord} 
        \text{ and } s \in \{ 1 , \ldots , c \} \right\}
        }
        {j \geq 1.}
\eeq
For every $s \in \{1, \ldots , c\}$, 
the linear functional
$\de_{p,0,s} \in \Lip(\gamma,\m,W)^{\ast}$ is 
defined, for 
$F = \left( F^{(0)} , \ldots , F^{(k)} \right) 
\in \Lip(\gamma,\m,W)$, by 
$\de_{p,0,s}(F) := w_s^{\ast} \left( F^{(0)}(p) \right)$
(cf. \eqref{eq:de_p_0_s}).
For every $s \in \{1 , \ldots , c\}$, $j \in \{1, \ldots , k\}$,
and $v \in \cv_j^{\ord}$, the linear functional 
$\de_{p,j,v,s} \in \Lip(\gamma,\m,W)^{\ast}$ is 
defined, for 
$F = \left( F^{(0)} , \ldots , F^{(k)} \right) 
\in \Lip(\gamma,\m,W)$, by 
$\de_{p,j,v,s}(F) := w_s^{\ast} \left( F^{(j)}(p)[v] \right)$
(cf. \eqref{eq:de_p_j_v_s}).
For a given $l \in \{0, \ldots , k\}$ define 
$\Tau_{p,l} \subset \Lip(\gamma,\m,W)^{\ast}$ by 
(cf. \eqref{eq:dual_norm_ests_lemma_Tau_pl})
\beq
    \label{eq:dual_norm_ests_lemma_pf_Tau_pl}
        \Tau_{p,l} := \bigcup_{j=0}^l \m_{p,j}.
\eeq
Now fix $p \in \m$ and $l \in \{0, \ldots , k\}$.
Consider $\sigma \in \Tau_{p,l}$.
It follows from \eqref{eq:dual_norm_ests_lemma_pf_m_p_j} 
and \eqref{eq:dual_norm_ests_lemma_pf_Tau_pl} that either 
$\sigma = \de_{p,0,s}$ for some $s \in \{1, \ldots ,c\}$, 
or $\sigma = \de_{p,j,v,s}$ for some $j \in \{0, \ldots ,l\}$, 
some $v \in \cv_j^{\ord}$, and some $s \in \{1, \ldots , c\}$.

If $\sigma = \de_{p,0,s}$ for some $s \in \{1, \ldots ,c\}$
then, given any $\phi = \left( \phi^{(0)} , \ldots , \phi^{(k)} 
\right) \in \Lip(\gamma,\m,W)$, we may appeal to Lemma 
\ref{lemma:delta_funcs_estimates} to deduce that 
(cf. \eqref{eq:de_p_0_s_bounds})
\beq
    \label{eq:dual_norm_ests_lemma_pf_conc_A1}
        \left| \sigma (\phi) \right| =
        \left| \de_{p,0,s}(\phi) \right)
        \leq \left|\left|\phi^{(0)}(p) \right|\right|_W
        \leq \Lambda^l_{\phi}(p) 
        \leq
        ||\phi||_{\Lip(\gamma,\m,W)}. 
\eeq
If $\sigma = \de_{p,j,v,s}$ for some $j \in \{0, \ldots ,l\}$, 
some $v \in \cv_j^{\ord}$, and some $s \in \{1, \ldots , c\}$, 
then we first note that $||v||_{V^{\otimes j}} = 1$.
Thus, given any $\phi = \left( \phi^{(0)} , \ldots , \phi^{(k)} 
\right) \in \Lip(\gamma,\m,W)$, we may appeal to Lemma 
\ref{lemma:delta_funcs_estimates} to deduce that 
(cf. \eqref{eq:de_p,j,v,s(phi)_bounds})
\beq
    \label{eq:dual_norm_ests_lemma_pf_conc_B1}
        \left| \sigma (\phi) \right| =
        \left| \de_{p,j,v,s}(\phi) \right|
        \leq \left|\left|\phi^{(j)}(p)[v] \right|\right|_W
        \leq 
        \left|\left| \phi^{(j)}(p) 
        \right|\right|_{\cl(V^{\otimes j};W)}
        \leq 
        \Lambda^l_{\phi}(p)
        \leq
        ||\phi||_{\Lip(\gamma,\m,W)}.
\eeq
Together, \eqref{eq:dual_norm_ests_lemma_pf_conc_A1} 
and \eqref{eq:dual_norm_ests_lemma_pf_conc_B1} establish 
the claim made in \eqref{eq:dual_norm_ests_lemma_pointwise_conc}.
Since the $\Lip(\gamma,\m,W)^{\ast}$ norm of $\sigma$
is given by
$||\sigma||_{\Lip(\gamma,\m,W)^{\ast}}
:= \sup \left\{ | \sigma(\phi) | : 
\phi \in \Lip(\gamma,\m,W) \text{ with }
||\phi||_{\Lip(\gamma,\m,W)} = 1 \right\}$, 
the claims made in \eqref{eq:dual_norm_ests_lemma_conc} are
an immediate consequence of 
\eqref{eq:dual_norm_ests_lemma_pointwise_conc}.
This completes the proof of Lemma \ref{lemma:dual_norm_ests}.
\end{proof}
\vskip 4pt
\noindent
We now use the inequalities established in Lemma 
\ref{lemma:delta_funcs_estimates} to establish, for 
each $j \in \{0, \ldots , k\}$, that the 
symmetric $j$-linear form from $V$ to $W$ given by 
$\phi^{(j)}(p) \in \cl(V^{\otimes j};W)$ is determined, 
in a quantified sense, by the set real numbers
$\left\{ \sigma(\phi) : \sigma \in \m_{p,j} \right\} 
\subset \R$ where $\m_{p,j}$ is defined in 
\eqref{eq:m_p_j_lin_funcs_not_sec}.
The precise result is stated in the following lemma.

\begin{lemma}[Linear Functionals Determine Pointwise Value]
\label{lemma:lin_funcs_det_pointwise_value}
Let $c,d \in \Z_{\geq 1}$ and assume that $V$ and $W$ are 
real Banach spaces of dimensions $d$ and $c$ respectively.
Let $\m \subset V$ be closed with $p \in \m$ 
and assume that the tensor 
powers of $V$ are all equipped with admissible norms 
(cf. Definition \ref{admissible_tensor_norm}).
Let $v_1 , \ldots , v_d \in V$ be a unit $V$-norm basis 
of $V$. 
For each integer $n \in \Z_{\geq 1}$ define 
(cf. \eqref{eq:V_ten_prod_basis})
\beq
    \label{eq:V_ten_prod_basis_lin_funcs_lemma}
        \cv_n := \left\{
        v_{l_1} \otimes \ldots \otimes v_{l_n} :
        (l_1 , \ldots , l_n) \in
        \{ 1 ,\ldots , d\}^n 
        \right\} \subset V^{\otimes n}
\eeq 
and (cf. \eqref{eq:symm_j_lin_form_determine_set})
\beq
    \label{eq:V_n^ord_lin_funcs_lemma}
        \cv_n^{\ord} := \left\{ 
        v_{l_1} \otimes \ldots \otimes v_{l_n} :
        (l_1 , \ldots , l_n) \in \{1 , \ldots , d\}^n
        \text{ such that }
        l_1 \leq \ldots \leq l_n \right\} 
        \subset \cv_n \subset V^{\otimes n}.
\eeq
Further, let $w_1, \ldots , w_c \in W$ be a unit $W$-norm 
basis of $W$, and let 
$w_1^{\ast} , \ldots , w_c^{\ast} \in W^{\ast}$
be the corresponding dual basis of $W^{\ast}$.

Let $\gamma > 0$ with $k \in \Z_{\geq 0}$ such that 
$\gamma \in (k,k+1]$.
For a given $j \in \{0, \ldots , k\}$, define 
$\m_{p,j} \subset \Lip(\gamma,\m,W)^{\ast}$ by
\beq
    \label{eq:m_p_j_lin_funcs_lemma}
        \m_{p,j} := \twopartdef{
        \left\{ \de_{p,0,s} : s \in \{1, \ldots , c\} \right\}
        }
        {j=0}
        {
        \left\{ \de_{p,j,v,s} : v \in \cv_j^{\ord} 
        \text{ and } s \in \{ 1 , \ldots , c \} \right\}
        }
        {j \geq 1.}
\eeq
For every $s \in \{1, \ldots , c\}$, 
the linear functional
$\de_{p,0,s} \in \Lip(\gamma,\m,W)^{\ast}$ is 
defined, for 
$F = \left( F^{(0)} , \ldots , F^{(k)} \right) 
\in \Lip(\gamma,\m,W)$, by 
$\de_{p,0,s}(F) := w_s^{\ast} \left( F^{(0)}(p) \right)$
(cf. \eqref{eq:de_p_0_s}).
For every $s \in \{1 , \ldots , c\}$, $j \in \{1, \ldots , k\}$,
and $v \in \cv_j^{\ord}$, the linear functional 
$\de_{p,j,v,s} \in \Lip(\gamma,\m,W)^{\ast}$ is 
defined, for 
$F = \left( F^{(0)} , \ldots , F^{(k)} \right) 
\in \Lip(\gamma,\m,W)$, by 
$\de_{p,j,v,s}(F) := w_s^{\ast} \left( F^{(j)}(p)[v] \right)$
(cf. \eqref{eq:de_p_j_v_s}).

Suppose that $A \geq 0$ and 
$\phi = \left( \phi^{(0)} , \ldots , \phi^{(k)} 
\right) \in \Lip(\gamma,\m,W)$.
Then for any $j \in \{0, \ldots , k\}$ we have that the 
following implications are valid.
Firstly,
\beq
    \label{lin_funcs_lemma_I}
        \left|\left| \phi^{(j)}(p) 
        \right|\right|_{\cl(V^{\otimes j};W)}
        \leq A
        \implies 
        \max \left\{ |\sigma(\phi)| : 
        \sigma \in \m_{p,j} \right\}
        \leq A.
\eeq
Secondly,
\beq
    \label{lin_funcs_lemma_II}
        \max \left\{ |\sigma(\phi)| : 
        \sigma \in \m_{p,j} \right\}
        \leq A
        \implies 
        \left|\left| \phi^{(j)}(p) 
        \right|\right|_{\cl(V^{\otimes j};W)} 
        \leq c d^j A.
\eeq
Consequently, in the case that $A = 0$, the combination of 
\eqref{lin_funcs_lemma_I} and \eqref{lin_funcs_lemma_II} yields 
that
\beq
    \label{lin_funcs_lemma_A=0_conc}
        \phi^{(j)}(p) \equiv 0  
        \text{ in } 
        \cl(V^{\otimes j};W)
        \qquad \iff \qquad 
        \max \left\{ | \sigma (\phi) | : 
        \sigma \in \m_{p,j} \right\} = 0.
\eeq
\end{lemma}

\begin{proof}[Proof of Lemma 
\ref{lemma:lin_funcs_det_pointwise_value}]
Let $c,d \in \Z_{\geq 1}$ and assume that $V$ and $W$ are 
real Banach spaces of dimensions $d$ and $c$ respectively.
Let $\m \subset V$ be closed with $p \in \m$
and assume that the tensor 
powers of $V$ are all equipped with admissible norms 
(cf. Definition \ref{admissible_tensor_norm}).
Let $v_1 , \ldots , v_d \in V$ be a unit $V$-norm basis 
of $V$. 
For each integer $n \in \Z_{\geq 1}$ define 
(cf. \eqref{eq:V_ten_prod_basis_lin_funcs_lemma})
\beq
    \label{eq:V_ten_prod_basis_lin_funcs_lemma_pf}
        \cv_n := \left\{
        v_{l_1} \otimes \ldots \otimes v_{l_n} :
        (l_1 , \ldots , l_n) \in
        \{ 1 ,\ldots , d\}^n 
        \right\} \subset V^{\otimes n}
\eeq 
and (cf. \eqref{eq:V_n^ord_lin_funcs_lemma})
\beq
    \label{eq:V_n^ord_lin_funcs_lemma_pf}
        \cv_n^{\ord} := \left\{ 
        v_{l_1} \otimes \ldots \otimes v_{l_n} :
        (l_1 , \ldots , l_n) \in \{1 , \ldots , d\}^n
        \text{ such that }
        l_1 \leq \ldots \leq l_n \right\} 
        \subset \cv_n \subset V^{\otimes n}.
\eeq
A consequence of the tensor powers of $V$ being equipped
with admissible norms in the sense of Definition 
\ref{admissible_tensor_norm} is that 
$\cv_n \subset V^{\otimes}$ is a unit $V^{\otimes n}$-norm 
basis of $V^{\otimes n}$.
Further, let $w_1, \ldots , w_c \in W$ be a unit $W$-norm 
basis of $W$, and let 
$w_1^{\ast} , \ldots , w_c^{\ast} \in W^{\ast}$
be the corresponding dual basis of $W^{\ast}$.

Let $\gamma > 0$ with $k \in \Z_{\geq 0}$ such that 
$\gamma \in (k,k+1]$.
For a given $j \in \{0, \ldots , k\}$, define 
$\m_{p,j} \subset \Lip(\gamma,\m,W)^{\ast}$ by
(cf. \eqref{eq:m_p_j_lin_funcs_lemma})
\beq
    \label{eq:m_p_j_lin_funcs_lemma_pf}
        \m_{p,j} := \twopartdef{
        \left\{ \de_{p,0,s} : s \in \{1, \ldots , c\} \right\}
        }
        {j=0}
        {
        \left\{ \de_{p,j,v,s} : v \in \cv_j^{\ord} 
        \text{ and } s \in \{ 1 , \ldots , c \} \right\}
        }
        {j \geq 1.}
\eeq
For every $s \in \{1, \ldots , c\}$, 
the linear functional
$\de_{p,0,s} \in \Lip(\gamma,\m,W)^{\ast}$ is 
defined, for 
$F = \left( F^{(0)} , \ldots , F^{(k)} \right) 
\in \Lip(\gamma,\m,W)$, by 
$\de_{p,0,s}(F) := w_s^{\ast} \left( F^{(0)}(p) \right)$
(cf. \eqref{eq:de_p_0_s}).
For every $s \in \{1 , \ldots , c\}$, $j \in \{1, \ldots , k\}$,
and $v \in V^{\otimes j}$, the linear functional 
$\de_{p,j,v,s} \in \Lip(\gamma,\m,W)^{\ast}$ is 
defined, for 
$F = \left( F^{(0)} , \ldots , F^{(k)} \right) 
\in \Lip(\gamma,\m,W)$, by 
$\de_{p,j,v,s}(F) := w_s^{\ast} \left( F^{(j)}(p)[v] \right)$
(cf. \eqref{eq:de_p_j_v_s}).

Suppose that   
$\phi = \left( \phi^{(0)} , \ldots , \phi^{(k)} 
\right) \in \Lip(\gamma,\m,W)$.
Observe that if \eqref{lin_funcs_lemma_I}
and \eqref{lin_funcs_lemma_II} are both valid for any 
$A \geq 0$, then 
the equivalence claimed in \eqref{lin_funcs_lemma_A=0_conc}
is an immediate consequence of invoking both for the 
particular choice of $A := 0$.
Thus we need only verify that both the implications 
\eqref{lin_funcs_lemma_I} and \eqref{lin_funcs_lemma_II}
are valid for any $A \geq 0$.
Thus, for the remainder of the proof, we fix an arbitrary 
choice of $A \geq 0$.

We begin with the case that $j=0$ so that
(cf. \eqref{eq:m_p_j_lin_funcs_lemma_pf})
        $\m_{p,0} =
        \left\{ \de_{p,0,s} : 
        s \in \{1, \ldots , c\} \right\} 
        \subset \Lip(\gamma,\m,W)^{\ast}$.
By appealing to Lemma \ref{lemma:delta_funcs_estimates}
we may conclude that, for each $s \in \{1, \ldots , c\}$,
we have (cf. \eqref{eq:de_p_0_s_bounds})
\beq
    \label{eq:de_0_ests_B}
        \left| \de_{p,0,s}(\phi) \right|
        \leq
        \left|\left| \phi^{(0)}(p) \right|\right|_W
        \leq
        \sum_{a=1}^c \left| \de_{p,0,s}(\phi) \right|.
\eeq 
If $\left|\left| \phi^{(0)}(p) \right|\right|_W \leq A$
then, for any $s \in \{1, \ldots , c\}$, a consequence of
\eqref{eq:de_0_ests_B} is that 
$|\de_{p,0,s}(\phi)| \leq A$. Therefore, since 
$\m_{p,0} = \left\{ \de_{p,0,s} : s \in \{1, \ldots ,c\} \right\}$,
we deduce that 
$\max \left\{ |\sigma(\phi)| : \sigma \in \m_{p,0} \right\} \leq A$.
Thus, recalling that $\cl(V^{\otimes 0};W) = W$ and we choose
to equip this space with the norm $||\cdot||_W$ on $W$
(cf. Remark \ref{rmk:op_norm_convention}), 
the implication claimed in \eqref{lin_funcs_lemma_I}
is valid when $j=0$.

If 
$\max \left\{ |\sigma(\phi)| : \sigma \in \m_{p,0} \right\} \leq A$, 
then since 
$\m_{p,0} = \left\{ \de_{p,0,s} : s \in \{1, \ldots ,c\} \right\}$
we have, for every $s \in \{1, \ldots , c\}$, that 
$\left| \de_{p,0,s}(\phi) \right| \leq A$.
Therefore, via an application of \eqref{eq:de_0_ests_B},
it follows that $\left|\left| \phi^{(0)}(p) \right|\right|_W \leq cA$.
Thus, recalling that $\cl(V^{\otimes 0};W) = W$ and we choose
to equip this space with the norm $||\cdot||_W$ on $W$
(cf. Remark \ref{rmk:op_norm_convention}), 
the implication claimed in \eqref{lin_funcs_lemma_II}
is valid when $j=0$.

Hence we have established that the lemma is true when $j=0$.

We next establish that the lemma is true when $j=1$. 
In this case, since $\cv_1^{\ord} = \cv_1$, we have  
$\m_{p,1} = \left\{ \de_{p,1,v,s} :
v \in \cv_1 \text{ and } s \in \{1, \ldots , c\} 
\right\} \subset \Lip(\gamma,\m,W)^{\ast}$.
Given $v \in \cv_1$ and $s \in \{1, \ldots , c\}$, 
we can appeal to Lemma \ref{lemma:delta_funcs_estimates} 
to conclude that (cf. \eqref{eq:de_p,j,v,s(phi)_bounds})
\beq
    \label{eq:de_1_ests_B}
        \left| \de_{p,1,v,s}(\phi) \right|
        \leq
        \left|\left| \phi^{(1)}(p)[v] \right|\right|_W
        \leq
        \sum_{a=1}^c \left| \de_{p,1,v,a}(\phi) \right|.
\eeq
If $\left|\left|\phi^{(1)}(p) \right|\right|_{\cl(V;W)} \leq A$ 
then, since $v \in \cv_1$ means $||v||_V = 1$, it follows from
\eqref{eq:de_1_ests_B} that 
$\left| \de_{p,1,v,s}(\phi) \right| \leq A$.
The arbitrariness of $v \in \cv_1$ and $s \in \{1, \ldots , c\}$
means we can deduce that  
$\max \left\{ \left| \sigma(\phi) \right| : 
\sigma \in \m_{p,1} \right\} \leq A$.
Thus the implication \eqref{lin_funcs_lemma_I} is valid 
when $j=1$.

If $\max \left\{ \left| \sigma(\phi) \right| : 
\sigma \in \m_{p,1} \right\} \leq A$ then, for every 
$v \in \cv_1$ and $s \in \{1, \ldots , c\}$, we have that
$\left|\de_{p,1,v,s}(\phi)\right| \leq A$.
Now suppose that $v \in V$ with $||v||_V = 1$.
Then, since $\cv_1$ is a basis of $V$, there are coefficients
$b_1 , \ldots , b_d \in \R$ such that 
$v = \sum_{m=1}^d b_m v_m$.
Moreover, by appealing to Lemma 
\ref{lemma:admissible_ten_norm_coeff_bound}, 
we have (cf. \eqref{coeff_sum_bound} for $m=1$) that 
$\sum_{j=m}^d |b_m| \leq d$.
Therefore, via an application of \eqref{eq:de_1_ests_B}, 
we have
\beq
    \label{eq:norm_bound_by_funcs_m=1}
        \left|\left| \phi^{(1)}(p)[v] \right|\right|_W
        \leq
        \sum_{m=1}^d |b_m| \left|\left| \phi^{(1)}(p)[v_m]
        \right|\right|_W
        \stackrel{
        \eqref{eq:de_1_ests_B}
        }{\leq} 
        \sum_{m=1}^d |b_m| \sum_{a=1}^c 
        \left|\de_{p,1,v_m,a}(\phi)\right|
        \leq
        \sum_{m=1}^d |b_m| \sum_{a=1}^c A
        \leq cdA.
\eeq
By taking the supremum over $v \in V$ with $||v||_V = 1$ 
in \eqref{eq:norm_bound_by_funcs_m=1}, we obtain that
$\left|\left| \phi^{(1)}(p) \right|\right|_{\cl(V;W)} \leq cdA$.
Thus the implication \eqref{lin_funcs_lemma_II} is 
valid when $j=1$.

Hence we have established that the lemma is true when $j=1$.

We complete the proof by establishing that, when 
$k \geq 2$, the lemma is 
true for an arbitrary fixed $j \in \Z_{\geq 2}$.
In this case, we have 
$\m_{p,j} = \left\{ \de_{p,j,v,s} : v \in \cv^{\ord}_j 
\text{ and } s \in \{1, \ldots , c\} \right\} 
\subset \Lip(\gamma,\m,W)^{\ast}$. 
Given $v \in \cv^{\ord}_j$ and $s \in \{1, \ldots , c\}$, 
we can appeal to Lemma \ref{lemma:delta_funcs_estimates} 
to conclude that (cf. \eqref{eq:de_p,j,v,s(phi)_bounds})
\beq
    \label{eq:de_j_ests_B}
        \left| \de_{p,j,v,s}(\phi) \right|
        \leq
        \left|\left| \phi^{(j)}(p)[v] \right|\right|_W
        \leq
        \sum_{a=1}^c \left| \de_{p,j,v,a}(\phi) \right|.
\eeq
If $\left|\left| \phi^{(j)}(p) 
\right|\right|_{\cl(V^{\otimes j};W)} \leq A$, 
since $v \in \cv_j^{\ord}$ means $||v||_{V^{\otimes j}}=1$, 
it follows from \eqref{eq:de_j_ests_B} that 
$\left| \de_{p,j,v,s}(\phi) \right| \leq A$.
The arbitrariness of both $v \in \cv_j^{\ord}$ and 
$s \in \{1, \ldots , c\}$ means we can deduce that 
$\max \left\{ |\sigma(\phi)| : \sigma \in \m_{p,j} \right\} \leq A$.
Thus the implication \eqref{lin_funcs_lemma_I} is valid
for arbitrary $j \in \Z_{\geq 2}$.

Now suppose that 
$\max \left\{ |\sigma(\phi)| : \sigma \in \m_{p,j} \right\} \leq A$.
We first claim that this means, given any $v \in \cv_j$ and 
any $s \in \{1, \ldots , c\}$, that 
$\left| \de_{p,j,v,s}(\phi) \right| \leq A$.
Indeed, if $v \in \cv_j^{\ord} \subset \cv_j$ then this 
is an immediate consequence of the definition of $\m_{p,j}$.
And it follows for general $v \in \cv_j$ as follows.
Fix $v \in \cv_j$ and $s \in \{1, \ldots , c\}$.
Then there exists a permutation $\rho \in S_j$ such that
$\rho(v) \in \cv_j^{\ord}$; see Definition 
\ref{admissible_tensor_norm} for the details of how the 
permutation group $S_j$ acts on the tensor product $V^{\otimes j}$.
But $\phi^{(j)}(p) \in \cl(V^{\otimes j};W)$ is a 
$j$-symmetric linear form from $V$ to $W$; so, in particular, 
$\phi^{(j)}(p)[\rho(v)] = \phi^{(j)}(p)[v]$.
Hence $\de_{p,j,v,s}(\phi) = \de_{p,j,\rho(v),s}(\phi)$
from which it is immediate that 
$\left| \de_{p,j,v,s}(\phi) \right| \leq A$ as claimed.

Let $v \in V^{\otimes j}$ with $||v||_{V^{\otimes j}} = 1$.
Then, by appealing to Lemma
\ref{lemma:admissible_ten_norm_coeff_bound}, 
we have that there are coefficients 
$ \left\{ C_{l_1 \dots l_j} : (l_1 , \ldots , l_j) \in 
\{ 1 , \ldots , d \}^j \right\}$ for which 
(cf. \eqref{basis_expansion})
\beq 
    \label{basis_expansion_level_j}
        v = \sum_{l_1 = 1}^d \ldots \sum_{l_j = 1}^d 
        C_{l_1 \dots l_j} v_{l_1} \otimes \ldots \otimes v_{l_j},
\eeq
and such that (cf. \eqref{coeff_sum_bound} for $A$ there as $1$)
\beq
    \label{coeff_sum_bound_level_j}
        \sum_{l_1 = 1}^{d} \dots \sum_{l_j = 1}^d
        \left| C_{l_1 \dots l_j} \right|
        \leq d^j.
\eeq
Consequently, via an application of 
\eqref{eq:de_j_ests_B}, we have
\begin{multline}
    \label{eq:norm_bound_by_funcs_j}
        \left|\left| \phi^{(j)}(p)[v] \right|\right|_W
        \leq
        \sum_{l_1}^d \dots \sum_{l_j=1}^d 
        \left| C_{l_1 \dots l_j} \right|
        \left|\left| \phi^{(j)}(p) \left[ 
        v_{l_1} \otimes \ldots \otimes v_{l_j} \right]
        \right|\right|_W \\
        \stackrel{
        \eqref{eq:de_j_ests_B}
        }{\leq} 
        \sum_{l_1}^d \dots \sum_{l_j=1}^d 
        \left| C_{l_1 \dots l_j} \right|
        \sum_{a=1}^c 
        \left| \de_{p,j, v_{l_1} \otimes \ldots \otimes v_{l_j} ,a}
        (\phi) \right|
        \leq 
        cA \sum_{l_1}^d \dots \sum_{l_j=1}^d 
        \left| C_{l_1 \dots l_j} \right|
        \stackrel{
        \eqref{coeff_sum_bound_level_j}
        }{\leq} 
        c d^j A.
\end{multline}
By taking the supremum over $v \in V^{\otimes j}$ with 
$||v||_{V^{\otimes j}} = 1$ in \eqref{eq:norm_bound_by_funcs_j}, 
we obtain that 
$\left|\left| \phi^{(j)}(p) \right|\right|_{\cl(V^{\otimes j};W)} 
\leq cd^j A$.
Thus the implication \eqref{lin_funcs_lemma_II} is valid for 
arbitrary $j \in \Z_{\geq 2}$.
This completes the proof of Lemma 
\ref{lemma:lin_funcs_det_pointwise_value}.
\end{proof}
\vskip 4pt
\noindent
We conclude this section by using Lemma 
\ref{lemma:lin_funcs_det_pointwise_value} to establish that
the value of a $\Lip(\gamma)$ function 
$\phi \in \Lip(\gamma,\m,W)$ at a point $p \in \m$
is determined, in a quantified sense, by the set of real numbers
$\left\{ \sigma(\phi) : \sigma \in \Tau_{p,k} \right\} 
\subset \R$ where $\Tau_{p,k}$ is defined in 
\eqref{eq:not_sec_Tau_pl}.
The precise result is stated in the following lemma.

\begin{lemma}[Determining Lipschitz Functions at Points]
\label{lemma:number_of_coeffs_for_point_value}
Let $c,d \in \Z_{\geq 1}$ and $V$ and $W$ be real Banach spaces 
of dimensions $d$ and $c$ respectively.
Let $\m \subset V$ be a closed subset, and assume that the 
tensor powers of $V$ are all equipped with admissible norms 
(cf. Definition \ref{admissible_tensor_norm}).
Let $v_1 , \ldots , v_d \in V$ be a unit $V$-norm basis 
of $V$. 
For each integer $n \in \Z_{\geq 1}$ define 
(cf. \eqref{eq:V_ten_prod_basis})
\beq
    \label{eq:V_ten_prod_basis_num_pts_lemma}
        \cv_n := \left\{
        v_{l_1} \otimes \ldots \otimes v_{l_n} :
        (l_1 , \ldots , l_n) \in
        \{ 1 ,\ldots , d\}^n 
        \right\} \subset V^{\otimes n}
\eeq 
and (cf. \eqref{eq:symm_j_lin_form_determine_set})
\beq
    \label{eq:V_n^ord_num_pts_lemma}
        \cv_n^{\ord} := \left\{ 
        v_{l_1} \otimes \ldots \otimes v_{l_n} :
        (l_1 , \ldots , l_n) \in \{1 , \ldots , d\}^n
        \text{ such that }
        l_1 \leq \ldots \leq l_n \right\} 
        \subset \cv_n \subset V^{\otimes n}.
\eeq
Further, let $w_1, \ldots , w_c \in W$ be a unit $W$-norm 
basis of $W$, and let 
$w_1^{\ast} , \ldots , w_c^{\ast} \in W^{\ast}$
be the corresponding dual basis of $W^{\ast}$.

Let $\gamma > 0$ with $k \in \Z_{\geq 0}$ such that 
$\gamma \in (k,k+1]$, and let 
$\phi = \left( \phi^{(0)} , \ldots , \phi^{(k)} 
\right) \in \Lip(\gamma,\m,W)$. 
For each point $p \in \m$ and each $j \in \{0, \ldots , k\}$
define $\m_{p,j} \subset \Lip(\gamma,\m,W)^{\ast}$ by 
\beq
    \label{eq:m_p_j_num_pts_lemma}
        \m_{p,j} := \twopartdef{
        \left\{ \de_{p,0,s} : s \in \{1, \ldots , c\} \right\}
        }
        {j=0}
        {
        \left\{ \de_{p,j,v,s} : v \in \cv_j^{\ord} 
        \text{ and } s \in \{ 1 , \ldots , c \} \right\}
        }
        {j \geq 1.}
\eeq
For every $s \in \{1, \ldots , c\}$, 
the linear functional
$\de_{p,0,s} \in \Lip(\gamma,\m,W)^{\ast}$ is 
defined, for 
$F = \left( F^{(0)} , \ldots , F^{(k)} \right) 
\in \Lip(\gamma,\m,W)$, by 
$\de_{p,0,s}(F) := w_s^{\ast} \left( F^{(0)}(p) \right)$
(cf. \eqref{eq:de_p_0_s}).
For every $s \in \{1 , \ldots , c\}$, $j \in \{1, \ldots , k\}$,
and $v \in \cv_j^{\ord}$, the linear functional 
$\de_{p,j,v,s} \in \Lip(\gamma,\m,W)^{\ast}$ is 
defined, for 
$F = \left( F^{(0)} , \ldots , F^{(k)} \right) 
\in \Lip(\gamma,\m,W)$, by 
$\de_{p,j,v,s}(F) := w_s^{\ast} \left( F^{(j)}(p)[v] \right)$
(cf. \eqref{eq:de_p_j_v_s}).
Further, for a given $l \in \{0, \ldots , k\}$, 
define $\Tau_{p,l} \subset \Lip(\gamma,\m,W)^{\ast}$ by
\beq
    \label{eq:Tau_p,l_num_pts_lemma}
        \Tau_{p,l} := 
        \bigcup_{j=0}^l \m_{p,j}.
\eeq
Recall the notation introduced in Remark 
\ref{rmk:notational_easing} that, for each point 
$x \in \m$ and each $j \in \{0, \ldots ,k\}$, we set 
\beq
    \label{eq:notation_ease_num_pts_lemma}
        \Lambda_{\phi}^j(x) := 
        \max_{s \in \{0, \ldots , j\}} \left\{
        \left|\left| \phi^{(s)}(x) 
        \right|\right|_{\cl(V^{\otimes s};W)} \right\}.
\eeq
Then, given $A \geq 0$ and $l \in \{0, \ldots , k\}$, the 
following implications are valid.
Firstly, 
\beq
    \label{eq:imp_1_num_pts_lemma}
        \Lambda_{\phi}^l(p) \leq A 
        \implies 
        \max \left\{ |\sigma(\phi)| : \sigma \in \Tau_{p,l} 
        \right\} \leq A.
\eeq
Secondly, 
\beq
    \label{eq:imp_2_num_pts_lemma}
        \max \left\{ |\sigma(\phi)| : \sigma \in \Tau_{p,l} 
        \right\} \leq A
        \implies 
        \Lambda_{\phi}^l(p) \leq cd^l A.
\eeq
Consequently, in the case that $A=0$, the combination of 
\eqref{eq:imp_1_num_pts_lemma} and \eqref{eq:imp_2_num_pts_lemma}
yields that
\beq
    \label{eq:A=0_equiv_num_pts_lemma}
        \Lambda_{\phi}^l(p) = 0 
        \iff 
        \max \left\{ |\sigma(\phi)| : \sigma \in \Tau_{p,l} 
        \right\} = 0.
\eeq
Therefore the condition that  
$\phi(p) = \left( \phi^{(0)}(p) , \ldots , \phi^{(k)}(p) \right)$
vanishes at $p \in \m$ is equivalent to the system of 
equations generated by requiring, for every $\sigma \in \Tau_{p,k}$, 
that $\sigma(\phi) = 0$; that is, to a system of
\beq
    \label{eq:number_of_coeffs}
        \# \left( \Tau_{p,k} \right) 
        = c \sum_{j=0}^k \be ( d , j)
\eeq
real valued equations.
Here, for integers $a,b \in \Z_{\geq 0}$, $\be(a,b)$ is 
defined by (cf. \eqref{eq:beta_ab_def_not_sec})
\beq
    \label{eq:beta_ab_def}
        \be(a,b) := 
        \left( 
        \begin{array}{c}
            a+b-1 \\
            b
        \end{array}
        \right)
        = \frac{(a+b-1)!}{(a-1)!b!}
        =
        \twopartdef{1}
        {b=0}
        {\frac{1}{b!}\prod_{s=0}^{b-1}(a+s)}
        {b \geq 1.}
\eeq
\end{lemma}

\begin{proof}[Proof of Lemma 
\ref{lemma:number_of_coeffs_for_point_value}]
Let $c,d \in \Z_{\geq 1}$ and $V$ and $W$ be real Banach spaces 
of dimensions $d$ and $c$ respectively.
Let $\m \subset V$ be a closed subset, and assume that the 
tensor powers of $V$ are all equipped with admissible norms 
(cf. Definition \ref{admissible_tensor_norm}).
Let $v_1 , \ldots , v_d \in V$ be a unit $V$-norm basis 
of $V$. 
For each integer $n \in \Z_{\geq 1}$ define 
(cf. \eqref{eq:V_ten_prod_basis_num_pts_lemma})
\beq
    \label{eq:V_ten_prod_basis_num_pts_lemma_pf}
        \cv_n := \left\{
        v_{l_1} \otimes \ldots \otimes v_{l_n} :
        (l_1 , \ldots , l_n) \in
        \{ 1 ,\ldots , d\}^n 
        \right\} \subset V^{\otimes n}
\eeq 
and (cf. \eqref{eq:V_n^ord_num_pts_lemma})
\beq
    \label{eq:V_n^ord_num_pts_lemma_pf}
        \cv_n^{\ord} := \left\{ 
        v_{l_1} \otimes \ldots \otimes v_{l_n} :
        (l_1 , \ldots , l_n) \in \{1 , \ldots , d\}^n
        \text{ such that }
        l_1 \leq \ldots \leq l_n \right\} 
        \subset \cv_n \subset V^{\otimes n}.
\eeq
Further, let $w_1, \ldots , w_c \in W$ be a unit $W$-norm 
basis of $W$, and let 
$w_1^{\ast} , \ldots , w_c^{\ast} \in W^{\ast}$
be the corresponding dual basis of $W^{\ast}$.

Let $\gamma > 0$ with $k \in \Z_{\geq 0}$ such that 
$\gamma \in (k,k+1]$, and let 
$\phi = \left( \phi^{(0)} , \ldots , \phi^{(k)} 
\right) \in \Lip(\gamma,\m,W)$. 
For each point $p \in \m$ and each $j \in \{0, \ldots , k\}$
define $\m_{p,j} \subset \Lip(\gamma,\m,W)^{\ast}$ by 
(cf. \eqref{eq:m_p_j_num_pts_lemma})
\beq
    \label{eq:m_p_j_num_pts_lemma_pf}
        \m_{p,j} := \twopartdef{
        \left\{ \de_{p,0,s} : s \in \{1, \ldots , c\} \right\}
        }
        {j=0}
        {
        \left\{ \de_{p,j,v,s} : v \in \cv_j^{\ord} 
        \text{ and } s \in \{ 1 , \ldots , c \} \right\}
        }
        {j \geq 1.}
\eeq
For every $s \in \{1, \ldots , c\}$, 
the linear functional
$\de_{p,0,s} \in \Lip(\gamma,\m,W)^{\ast}$ is 
defined, for 
$F = \left( F^{(0)} , \ldots , F^{(k)} \right) 
\in \Lip(\gamma,\m,W)$, by 
$\de_{p,0,s}(F) := w_s^{\ast} \left( F^{(0)}(p) \right)$
(cf. \eqref{eq:de_p_0_s}).
For every $s \in \{1 , \ldots , c\}$, $j \in \{1, \ldots , k\}$,
and $v \in \cv_j^{\ord}$, the linear functional 
$\de_{p,j,v,s} \in \Lip(\gamma,\m,W)^{\ast}$ is 
defined, for 
$F = \left( F^{(0)} , \ldots , F^{(k)} \right) 
\in \Lip(\gamma,\m,W)$, by 
$\de_{p,j,v,s}(F) := w_s^{\ast} \left( F^{(j)}(p)[v] \right)$
(cf. \eqref{eq:de_p_j_v_s}).
Further, for a given $l \in \{0, \ldots , k\}$, 
define $\Tau_{p,l} \subset \Lip(\gamma,\m,W)^{\ast}$ by
(cf. \eqref{eq:Tau_p,l_num_pts_lemma})
\beq
    \label{eq:Tau_p,l_num_pts_lemma_pf}
        \Tau_{p,l} := 
        \bigcup_{j=0}^l \m_{p,j}.
\eeq
Recall the notation introduced in Remark 
\ref{rmk:notational_easing} that, for each point 
$x \in \m$ and each $j \in \{0, \ldots ,k\}$, we set 
(cf. \eqref{eq:notation_ease_num_pts_lemma})
\beq
    \label{eq:notation_ease_num_pts_lemma_pf}
        \Lambda_{\phi}^j(x) := 
        \max_{s \in \{0, \ldots , j\}} \left\{
        \left|\left| \phi^{(s)}(x) 
        \right|\right|_{\cl(V^{\otimes s};W)} \right\}.
\eeq
We now turn our attention to verifying that both the 
implications \eqref{eq:imp_1_num_pts_lemma} and 
\eqref{eq:imp_2_num_pts_lemma} are valid. 
For this purpose let $A \geq 0$ and fix 
$l \in \{0, \ldots ,k\}$.

First assume that $\Lambda_{\phi}^l(p) \leq A$.
A consequence of \eqref{eq:notation_ease_num_pts_lemma_pf} is 
that, for every $j \in \{0, \ldots , l\}$, we have the bound
$\left|\left| \phi^{(j)}(p) \right|\right|_{\cl(V^{\otimes j};W)} 
\leq A$.
In turn, by appealing to Lemma 
\ref{lemma:lin_funcs_det_pointwise_value}, 
we deduce, for every $j \in \{0 , \ldots , l\}$, that 
(cf. \eqref{lin_funcs_lemma_I}) 
$\max \left\{ |\sigma(\phi)| : \sigma \in \m_{p,j} \right\} \leq A$.
It now follows from \eqref{eq:Tau_p,l_num_pts_lemma_pf} that
$\max \left\{ |\sigma(\phi)| : \sigma \in \Tau_{p,l} \right\} 
\leq A$.
Thus the implication \eqref{eq:imp_1_num_pts_lemma} is 
valid as claimed.
We could alternatively have established 
\eqref{eq:imp_1_num_pts_lemma} via use of the 
inequality \eqref{eq:dual_norm_ests_lemma_pointwise_conc} 
established in Lemma \ref{lemma:dual_norm_ests}.

Now assume that 
$\max \left\{ |\sigma(\phi)| : \sigma \in \Tau_{p,l} \right\} 
\leq A$. 
A consequence of \eqref{eq:Tau_p,l_num_pts_lemma_pf} is that, 
for every $j \in \{0, \ldots , l\}$, we have 
$\max \left\{ |\sigma(\phi)| : \sigma \in \m_{p,j} \right\} \leq A$.
In turn, by appealing to Lemma 
\ref{lemma:lin_funcs_det_pointwise_value}, we deduce, for every
$j \in \{0, \ldots , l\}$, that 
(cf. \eqref{lin_funcs_lemma_II})
$\left|\left| \phi^{(j)}(p) \right|\right|_{\cl(V^{\otimes j};W)} 
\leq c d^j A$.
It now follows from \eqref{eq:notation_ease_num_pts_lemma_pf} 
that $\Lambda_{\phi}^l(p) \leq c d^l A$.
Thus the implication \eqref{eq:imp_2_num_pts_lemma} is 
valid as claimed.

The combination of the implication 
\eqref{eq:imp_1_num_pts_lemma} for $A:=0$ and the 
implication \eqref{eq:imp_2_num_pts_lemma} for $A:=0$
yield that 
$\Lambda_{\phi}^l(p)=0$ if and only if 
$\max \left\{ | \sigma(\phi)| : \sigma \in \Tau_{p,l} \right\} = 0$
as claimed in \eqref{eq:A=0_equiv_num_pts_lemma}.

To complete the proof, we consider imposing the condition 
that $\phi$ vanishes at the point $p$. 
In particular, this is equivalent to having, for every 
$j \in \{0, \ldots , k\}$, that 
$\left|\left| \phi^{(j)}(p) \right|\right|_{\cl(V^{\otimes j};W)} 
= 0$.
Consequently, by appealing to \eqref{eq:A=0_equiv_num_pts_lemma},
this is equivalent to having, for every $\sigma \in \Tau_{p,k}$,
that $\sigma(\phi) = 0$.
This is a system of $\#\left( \Tau_{p,k}\right)$ real-valued 
equations.
The proof will be complete if we can establish that 
$\# \left( \Tau_{p,k} \right) = c \sum_{j=0}^k \be(d,j)$
as claimed in \eqref{eq:number_of_coeffs} where, 
for integers $a,b \in \Z_{\geq 0}$, $\be(a,b)$ is given by
(cf. \eqref{eq:beta_ab_def})
\beq
    \label{eq:beta_ab_def_pf}
        \be (a,b) := 
        \left( 
        \begin{array}{c}
            a+b-1 \\
            b
        \end{array}
        \right)
        = \frac{(a+b-1)!}{(a-1)!b!}
        =
        \twopartdef{1}
        {b=0}
        {\frac{1}{b!}\prod_{s=0}^{b-1}(a+s)}
        {b \geq 1.}
\eeq
We first note via \eqref{eq:Tau_p,l_num_pts_lemma_pf} that
$\# \left( \Tau_{p,k} \right) = 
\sum_{j=0}^k \# \left( \m_{p,j} \right)$.
Moreover, for each $j \in \{0, \ldots , k\}$, it follows from 
\eqref{eq:m_p_j_num_pts_lemma_pf} that
\beq
    \label{eq:m_p_j_cardinality}
        \# \left( \m_{p,j} \right) := 
        \twopartdef{c}{j=0}
        {
        c \# \left( \cv_j^{\ord} \right)
        }
        {j \geq 1}
        \stackrel{
        \eqref{eq:beta_ab_def_pf}
        }{=}
        c \be(d,j).
\eeq
Hence $\# \left( \Tau_{p,k} \right) = c \sum_{j=0}^k \be(d,j)$
as claimed in \eqref{eq:number_of_coeffs}.
This completes the proof of Lemma
\ref{lemma:number_of_coeffs_for_point_value}.
\end{proof}

\section{The HOLGRIM Algorithm}
\label{sec:HOLGRIM_alg}
The following 
\textbf{HOLGRIM Extension Step} 
is used to dynamically grow the collection of 
points $P \subset \Sigma$ at which we require our 
next approximation of $\vph$ to agree with $\vph$.
\vskip 4pt
\noindent
\textbf{HOLGRIM Extension Step}

\noindent
Assume $P' \subset \Sigma$. Let $u \in \Span(\cf)$.
Let $q \in \{0, \ldots , k\}$ and $m \in \Z_{\geq 1}$ 
such that $\#(P') + m \leq \Lambda := \#(\Sigma)$.
\beq
    \label{eq:Sigma_star_def}
        \Sigma_q^{\ast} := \bigcup_{z \in \Sigma}
        \Tau_{z,q} \subset \Lip(\gamma,\Sigma,W)^{\ast}.
\eeq
First take 
\beq
    \label{HOLGRIM_ext_step_v2_sigma1}
        \sigma_1 := \argmax \left\{ 
        \left| \sigma (\vph - u) \right| : 
        \sigma \in \Sigma_q^{\ast} \right\},
\eeq
and then take $z_1 \in \Sigma$ to be the point for which
$\sigma_1 \in \Tau_{z_1,q}$.

Inductively for $j = 2,3,\ldots,m$ take
\beq
    \label{HOLGRIM_ext_step_v2_sigmaj}
        \sigma_j := \argmax \left\{
        \left| \sigma (\vph - u) \right| :
        \sigma \in \Sigma_q^{\ast} \setminus 
        \bigcup_{s=1}^{j-1} \Tau_{z_s, q} \right\}
\eeq
and then take $z_j \in \Sigma$ to be the point for which 
$\sigma_j \in \Tau_{z_j,q}$.

Once $z_1 , \ldots , z_m \in \Sigma$ have been
defined, we extend $P'$ to 
$P := P' \cup \left\{ z_1 , \ldots , z_m \right\}$.
\vskip 8pt
\noindent
The following \textbf{HOLGRIM Recombination Step} 
details how, for a given subset $P \subset \Sigma$ 
and a given $\ep_0 \geq 0$, 
we use recombination to find a function $u \in \Span(\cf)$
satisfying, for every point $z \in P$, that 
$\Lambda^k_{\vph-u}(z) \leq \ep_0$.
\vskip 4pt
\noindent
\textbf{HOLGRIM Recombination Step}

\noindent
Assume $P \subset \Sigma$ with $m \in \Z_{\geq 1}$ such 
that $m = \# \left( P \right)$.
For integers $a,b \in \Z_{\geq 1}$, let $\be(a,b)$ be
as defined in \eqref{eq:beta_ab_def_not_sec}. That is,
\beq
    \label{eq:beta_ab_def_recomb_step}
        \be (a,b) := 
        \left( 
        \begin{array}{c}
            a+b-1 \\
            b
        \end{array}
        \right)
        = \frac{(a+b-1)!}{(a-1)!b!}
        =
        \twopartdef{1}
        {b=0}
        {\frac{1}{b!}\prod_{s=0}^{b-1}(a+s)}
        {b \geq 1.}
\eeq
Let $Q = Q(m,c,d,k) \in \Z_{\geq 1}$ be as defined 
in \eqref{eq:D_ab_Q_ijab_def_not_sec}. That is,
\beq
    \label{lip_k_recomb_eqn_num_recomb_step}
        Q := 
        1 + m c \sum_{r=0}^k \be(d,r).
\eeq
Let $s \in \Z_{\geq 1}$, $q \in \{0 , \ldots , k\}$ 
and define $\Sigma_q^{\ast} := \cup_{z \in \Sigma} \Tau_{p,q}$.
For each $j \in \{1, \ldots , s\}$ we 
do the following.
\begin{enumerate}[label=(\Alph*)]
    \item\label{HOLGRIM_recomb_step_A}
    Define $L(k) \subset \Lip(\gamma,\Sigma,W)^{\ast}$ by
    \beq
        \label{eq:L_star_def_recomb_step}
            L(k) := \bigcup_{z \in P} 
            \Tau_{z,k} \subset \Lip(\gamma,\Sigma,W)^{\ast}.
    \eeq 
    By appealing to Lemma 
    \ref{lemma:number_of_coeffs_for_point_value} we see that, 
    the cardinality of $L(k)$ is
    (cf. \eqref{eq:number_of_coeffs})
    \beq
        \label{eq:L_star_cardinality_recomb_step}
            \#\left(L(k)\right) =
            \sum_{z \in P} \# \left( \Tau_{z,k} \right)
            =
            m c \sum_{j=0}^k \be(d,j)
            \stackrel{
            \eqref{lip_k_recomb_eqn_num_recomb_step}
            }{=} Q - 1.
    \eeq
    \item\label{HOLGRIM_recomb_step_B} 
    Let $L_j(k) \subset \Lip(\gamma,\Sigma,W)^{\ast}$ 
    be the subset resulting 
    from applying a random permutation to the ordering
    of the linear functionals in $L(k)$.
    \item\label{HOLGRIM_recomb_step_C}
    Apply recombination via Lemma 3.1 in \cite{LM22}, with
    the Banach space $X$, the integers $\n,m \in \Z_{\geq 1}$, 
    the subset $L \subset X^{\ast}$, the finite subset 
    $\cf \subset X$, 
    and the target $\vph \in \Span(\cf) \subset X$ in that result
    as $\Lip(\gamma,\Sigma,W)$, $\n, Q - 1 \in \Z_{\geq 1}$, 
    $L_j(k)$, and 
    $\vph \in \Span(\cf) \subset \Lip(\gamma,\Sigma,W)$ here
    respectively, to find $u_j \in \Span(\cf)$ satisfying 
    \beq
        \label{eq:recomb_outcome_recomb_step} 
                \max \left\{ \left| \sigma (\vph-u_j) \right| : 
                \sigma \in L_j(k) \right\} 
                \leq \frac{\ep_0}{c d^k}.
    \eeq
    Since $L_j(k)$ is a re-ordering of $L(k)$, 
    the combination of \eqref{eq:L_star_def_recomb_step} 
    and Lemma \ref{lemma:number_of_coeffs_for_point_value} 
    for the subset $\m$ there as $P$ here
    (specifically implication \eqref{eq:imp_2_num_pts_lemma}) 
    yields, for each point $z \in P$, that 
    $\Lambda^k_{\vph - u_j}(z) \leq \ep_0$.
    \item\label{HOLGRIM_recomb_step_D}
    Compute 
    \beq
        \label{eq:E[u_j]}
            E[u_j] := \max \left\{ 
            \left| \sigma ( \vph - u_j ) \right| : 
            \sigma \in \Sigma_q^{\ast}
            \right\}.
    \eeq
\end{enumerate}
After obtaining the functions $u_1 , \ldots , u_s \in 
\Span(\cf)$ we define $u \in \Span(\cf)$ by
\beq    
    \label{HOLGRIM_recomb_step_approx}
        u := \argmin \left\{ E[u_j] : 
        j \in \{ 1 , \ldots , s \} \right\}.
\eeq 
Then $u \in \Span(\cf)$ is returned as our approximation 
of $\vph$ that satisfies, for every $z \in P$, that 
$\Lambda^k_{\vph-u}(z) \leq \ep_0$.
\vskip 4pt
\noindent
Theoretically, there is no problem choosing $\ep_0 := 0$
in the \textbf{HOLGRIM Recombination Step}. The use of 
recombination in Lemma 3.1 in \cite{LM22} does theoretically 
find an approximation 
$u$ of $\vph$ such that, for every $\sigma \in L_j(k)$,
we have $\sigma(\vph-u) = 0$.
Hence, via Lemma \ref{lemma:number_of_coeffs_for_point_value} 
(cf. \eqref{eq:A=0_equiv_num_pts_lemma} for $l$ there are 
$k$ here), 
we may theoretically conclude that 
$\Lambda^k_{\vph-u}(z) = 0$ 
for all points $z \in P$ for the given subset 
$P \subset \Sigma$.
However, implementations of recombination inevitably result 
in numerical errors. 
That is, the returned coefficients will only solve the 
equations modulo some (ideally) small error term. 
Similarly to \cite{LM22}, we account for this in our 
analysis by only assuming in 
\eqref{eq:recomb_outcome_recomb_step} 
that the resulting 
approximation $u_j \in \Span(\cf)$ only satisfies, 
for every $\sigma \in L_j(k)$, that 
$|\sigma(\vph-u_j)| \leq \ep_0/ c d^k$ 
for some (small) constant $\ep_0 \geq 0$.
The inclusion of the $c d^k$ factor on the denominator is
to ensure that, via an application of implication 
\eqref{eq:imp_2_num_pts_lemma} in Lemma 
\ref{lemma:number_of_coeffs_for_point_value} for the subset 
$\m$ there as $P$ here, 
for every point $z \in P$ we have that
$\Lambda^k_{\vph-u}(z) \leq \ep_0$.

We exploit the potential advantage offered by shuffling 
the order of the inputs for recombination that is utilised
in \cite{LM22}.
That is, if recombination is applied to 
a linear system of equations corresponding to a matrix $A$, 
then a Singular Value Decomposition (SVD) of the matrix $A$
is used to find a basis of the kernel $\ker(A)$
(see Section 3 in \cite{LM22}, for example).
Consequently, re-ordering the rows of the matrix 
(i.e. changing the order in which the equations are 
considered) can potentially result in a different
basis for $\ker(A)$ being selected. 
Hence shuffling the order of the equations can affect the 
approximation returned by recombination via Lemma 3.1 in 
\cite{LM22}.
As done in \cite{LM22}, we exploit this by optimising 
the approximation returned by recombination over a 
chosen number of shuffles of the 
equations forming the linear system. 

We now detail our proposed \textbf{HOLGRIM} algorithm to find an 
approximation $u \in \Span(\cf)$ of $\vph \in \Span(\cf)$
that is close to $\vph$ throughout $\Sigma$ in the sense 
that, for a given $q \in \{0, \ldots , k\}$, we have  
$\Lambda^q_{\vph-u}(z) \leq \ep$ for every point 
$z \in \Sigma$.

\vskip 4pt
\noindent\textbf{HOLGRIM}
\begin{enumerate}[label=(\Alph*)]
    \item\label{HOLGRIM_A}
    Fix $\ep > 0$ as the \emph{target accuracy
    threshold}, $\ep_0 \in [0, \ep)$
    as the \emph{acceptable recombination error},
    $M \in \Z_{\geq 1}$ as the 
    \emph{maximum number of steps}, and 
    $q \in \{0, \ldots , k\}$ as the \emph{order level}.
    Choose integers $s_1 , \ldots , s_M \in \Z_{\geq 1}$
    as the \emph{shuffle numbers}, and integers 
    $k_1 , \ldots , k_M \in \Z_{\geq 1}$ with 
    \beq 
        \label{HOLGRIM_num_pts_bd}
            \kappa := k_1 + \ldots + k_M 
            \leq \min \left\{ 
            \frac{\n - 1}{cD(d,k)} ,
            \Lambda \right\}.
    \eeq
    Choose \emph{scaling factors}
    $A_1 , \ldots , A_{\n} \in \R$ such that, for every
    $i \in \{1, \ldots , \n\}$, we have 
    $||f_i||_{\Lip(\gamma,\Sigma,W)} \leq A_i$.
    \item\label{HOLGRIM_B} 
    For each $i \in \{1, \ldots , \n\}$, if 
    $a_i < 0$ then replace $a_i$ and $f_i$ by
    $-a_i$ and $-f_i$ respectively.
    This ensures that $a_1 , \ldots , a_{\n} > 0$
    whilst leaving the expansion 
    $\vph = \sum_{i=1}^{\n} a_i f_i$ unaltered.
    Additionally, for each 
    $i \in \{ 1 , \ldots , \n\}$ replace $f_i$ by 
    $h_i := f_i / A_i$ 
    so that $h_i \in \Lip(\gamma,\Sigma,W)$ with
    $||h_i||_{\Lip(\gamma,\Sigma,W)} \leq 1$,
    and that
    $\varphi = \sum_{i=1}^{\n} \al_i h_i$
    where 
    $\al_i := a_i A_i$. 
    \item\label{HOLGRIM_C}
    Apply the \textbf{HOLGRIM Extension Step}, 
    with $P' := \varnothing$, 
    $u \equiv 0$ and $m := k_1$, to obtain a 
    subset $\Sigma_1 = \left\{ z_{1,1} , \ldots ,
    z_{1,k_1} \right\} \subset \Sigma$.
    Apply the \textbf{HOLGRIM Recombination Step},
    with $P := \Sigma_1$ and $s:= s_1$, to find a 
    function $u_1 \in \Span(\cf)$ satisfying, for
    every point $z \in \Sigma_1$, that
    $\Lambda^{k}_{\vph-u_1}(z) \leq \ep_0$.

    \noindent
    If $M=1$ then the algorithm terminates here and 
    returns $u_1$ as the final approximation of $\vph$.
    
    \item\label{HOLGRIM_D} 
    If $M \geq 2$ then we proceed inductively for 
    $t \geq 2$ as follows.
    If $| \sigma ( \vph - u_{t-1}) | \leq \frac{\ep}{c d^q}$
    for every linear functional $\sigma \in \Sigma^{\ast}_q$
    then we stop and return $u_{t-1}$ as the final 
    approximation of $\vph$.
    Otherwise, we apply the
    \textbf{HOLGRIM Extension Step} with 
    $P' := \Sigma_{t-1}$, 
    $u := u_{t-1}$ and $m := k_t$, to obtain a 
    subset $\Sigma_t = \Sigma_{t-1} \cup
    \left\{ z_{t,1} , \ldots ,
    z_{t,k_t} \right\} \subset \Sigma$.
    Apply the \textbf{HOLGRIM Recombination Step}, 
    with $P := \Sigma_t$ and $s := s_t$, to find a
    function $u_t \in \Span(\cf)$ satisfying, for
    every point $z \in \Sigma_t$, that
    $\Lambda^k_{\varphi - u_t}(z) \leq \ep_0$.

    \noindent
    The algorithm ends either by returning $u_{t-1}$
    for $t \in \{2, \ldots , M\}$ for which the stopping 
    criterion was triggered as the final approximation of 
    $\vph$, or by returning $u_M$ as the final 
    approximation of $\vph$ if the stopping criterion is 
    never triggered.
\end{enumerate}
\vskip 4pt
\noindent
We claim that if $M \geq 2$ and the 
\textbf{HOLGRIM} algorithm terminates \emph{before}
the final $M^{\text{th}}$ step is completed, then the 
resulting approximation $u \in \Span(\cf)$ of $\vph$ 
satisfies, for every $z \in \Sigma$, that 
$\Lambda^q_{\vph-u}(z) \leq \ep$.
To see this, let $m \in \{2, \ldots , M\}$
and suppose that the \textbf{HOLGRIM} algorithm 
terminates during the $m^{\text{th}}$ step, 
i.e. that the $m^{\text{th}}$ step is not completed.
Then $u_{m-1} \in \Span(\cf)$ is the approximation of 
$\vph$ returned by the \textbf{HOLGRIM} algorithm.

It follows from Step \ref{HOLGRIM_D} in the 
\textbf{HOLGRIM} algorithm that this requires that 
$|\sigma(\vph - u_{m-1})| \leq \ep / c d^q$ for 
every $\sigma \in \Sigma^{\ast}_q$.
Recall that $\Sigma^{\ast}_q := \cup_{z \in \Sigma} \Tau_{z,q}$.
Thus, for any $z \in \Sigma$, we must have 
that $|\sigma(\vph - u_{m-1})| \leq \ep / c d^q$ for 
every $\sigma \in \Tau_{z,q}$.
Consequently, by appealing to Lemma 
\ref{lemma:number_of_coeffs_for_point_value} 
(specifically to the implication 
\eqref{eq:imp_2_num_pts_lemma}), 
we conclude that $\Lambda^q_{\vph - u_{m-1}}(z) \leq \ep$.
Since $z \in \Sigma$ was arbitrary we have established 
that $\max_{z \in \Sigma} \left\{ \Lambda^q_{\vph - u_{m-1}}(z)
\right\} \leq \ep$
as claimed.

\section{Complexity Cost}
\label{sec:HOLGRIM_complexity_cost}
In this section we establish an estimate for the complexity cost
of the \textbf{HOLGRIM} algorithm detailed in Section 
\ref{sec:HOLGRIM_alg}.
We first record the complexity cost of the 
\textbf{HOLGRIM Extension Step}.

\begin{lemma}[\textbf{HOLGRIM Extension Step} Complexity Cost]
\label{lemma:HOLGRIM_ext_step_cost}
Let $\n,\Lambda,m,t,c,d \in \Z_{\geq 1}$.
Let $V$ and $W$ be real Banach spaces of dimensions $d$ and 
$c$ respectively.
Assume the tensor powers of $V$ are all equipped with admissible
norms (cf. Definition \ref{admissible_tensor_norm}). 
Let $\Sigma \subset V$ be a finite subset with cardinality 
$\Lambda$. 
Let $\gamma > 0$ with $k \in \Z_{\geq 0}$ such that 
$\gamma \in (k,k+1]$, and let $q \in \{0, \ldots ,k\}$.
Define $\Sigma^{\ast}_q := \cup_{z \in \Sigma} \Tau_{z,q}$.
Let 
$\cf = \{ f_1 , \ldots , f_{\n} \} \subset \Lip(\gamma,\Sigma,W) 
\setminus \{0\}$.
Let $a_1 , \ldots , a_{\n} \in \R \setminus \{0\}$ and define
$\vph := \sum_{i=1}^{\n} a_i f_i$.
Assume that the subset 
$\left\{ \sigma(\vph) : \sigma \in \Sigma_q^{\ast} \right\} 
\subset \R$ and, for every $i \in \{1, \ldots , \n\}$, 
the subset
$\left\{ \sigma(f_i) : \sigma \in \Sigma^{\ast}_q \right\} 
\subset \R$ have already been computed.
Then for any subset $P' \subset \Sigma$ with 
$\#\left(P'\right) + m \leq \Lambda$ and any
$u \in \Span(\cf)$ with $\# \support(u) = t$, 
the complexity cost of applying the 
\textbf{HOLGRIM Extension Step}, with the $P'$, $u$, and $m$ 
there as the $P'$, $u$, and $m$ here respectively, is 
$\cO \left( c (m+t) \Lambda D(d,q) \right)$.
\end{lemma}

\begin{proof}[Proof of Lemma 
\ref{lemma:HOLGRIM_ext_step_cost}]
Let $\n,\Lambda,m,t,c,d \in \Z_{\geq 1}$.
Let $V$ and $W$ be real Banach spaces of dimensions $d$ and 
$c$ respectively.
Assume the tensor powers of $V$ are all equipped with admissible
norms (cf. Definition \ref{admissible_tensor_norm}). 
Let $\Sigma \subset V$ be a finite subset with cardinality 
$\Lambda$. 
Let $\gamma > 0$ with $k \in \Z_{\geq 0}$ such that 
$\gamma \in (k,k+1]$, and let $q \in \{0, \ldots ,k\}$.
Define $\Sigma^{\ast}_q := \cup_{z \in \Sigma} \Tau_{z,q}$.
Let 
$\cf = \{ f_1 , \ldots , f_{\n} \} \subset \Lip(\gamma,\Sigma,W) 
\setminus \{0\}$.
Let $a_1 , \ldots , a_{\n} \in \R \setminus \{0\}$ and define
$\vph := \sum_{i=1}^{\n} a_i f_i$.
Assume that the subset 
$\left\{ \sigma(\vph) : \sigma \in \Sigma_q^{\ast} \right\} 
\subset \R$ and, for every $i \in \{1, \ldots , \n\}$, 
the subset
$\left\{ \sigma(f_i) : \sigma \in \Sigma^{\ast}_q \right\} 
\subset \R$ have already been computed.
Suppose that $P' \subset \Sigma$ with 
$\#\left(P'\right) + m \leq \Lambda$ and
$u \in \Span(\cf)$ with $\# \support(u) = t$.
Recall our convention that $\support(u)$ is the set of the 
functions $f_i$ that correspond to a non-zero coefficient in 
the expansion of $u$ in terms of the $f_i$. That is, 
$u = \sum_{i=1}^{\n} u_i f_i$ for real numbers 
$u_1 , \ldots , u_{\n} \in \R$ and 
\beq
    \label{eq:ext_step_cost_lemma_pf_support_def}
        \support(u) := \left\{ f_j : j \in \{1 , \ldots , \n\} 
        \text{ and } u_j \neq 0 \right\}.
\eeq 
Consider carrying out the \textbf{HOLGRIM Extension Step} 
with the $P'$, $u$, and $m$ there are $P'$, $u$, and $m$ 
here respectively.
Since we have access to $\left\{ \sigma(\vph) : \sigma \in 
\Sigma_q^{\ast} \right\}$ and 
$\left\{ \sigma(f_i) : \sigma \in \Sigma_q^{\ast} \right\}$
for every $i \in \{1, \dots , \n\}$ without additional 
computation, and since $\# \support(u) = t$, the complexity 
cost of computing the set 
$\left\{ \left| \sigma(\vph-u) \right| : \sigma \in \Sigma_q^{\ast} 
\right\}$ 
is no worse than 
$\cO \left( t \# \left( \Sigma^{\ast}_q \right)\right) 
= \cO \left( c t \Lambda D(d,q) \right)$.
Each of the $m$ argmax values are found from a list of no 
greater than $\# \left( \Sigma_q^{\ast} \right) = c \Lambda D(d,q)$
real numbers, and so the complexity cost of extracting these
$m$ argmax values is no worse than 
$\cO \left( c m \Lambda D(d,q) \right)$.
Moreover, for $j \in \{2, \ldots ,m\}$, the list from which we 
extract the next the $j^{\text{th}}$ argmax value 
is obtained by removing $c D(d,q)$ values from the previous 
list, i.e. for every linear functional $\sigma \in \Tau_{z_{j-1},q}$ 
we remove the value $\sigma(\vph-u)$, 
which results in removing $c D(d,q)$ values in total 
since this is the cardinality of $\Tau_{z_{j-1},q}$ 
(cf. \eqref{eq:Tau_pl_card_not_sec}).
Thus the complexity cost of the discarding is, in total, no 
worse than $\cO\left(cm D(d,q) \right)$
The complexity cost of appending the resulting $m$ points 
$z_1 , \ldots , z_m \in \Sigma$ to the collection $P'$ is 
$\cO(m)$.
Therefore the entire \textbf{HOLGRIM Extension Step} has a 
complexity cost no worse than 
$\cO \left( c(m+t)\Lambda D(d,q) \right)$ 
as claimed.
This completes the proof of Lemma \ref{lemma:HOLGRIM_ext_step_cost}.
\end{proof}
\vskip 4pt 
\noindent
We next record the complexity cost of the 
\textbf{HOLGRIM Recombination Step}.

\begin{lemma}[\textbf{HOLGRIM Recombination Step} 
Complexity Cost]
\label{lemma:HOLGRIM_recomb_step_cost}
Let $\n,\Lambda,m,c,d \in \Z_{\geq 1}$.
Let $V$ and $W$ be real Banach spaces of dimensions $d$ and 
$c$ respectively.
Assume the tensor powers of $V$ are all equipped with admissible
norms (cf. Definition \ref{admissible_tensor_norm}). 
Let $\Sigma \subset V$ be a finite subset with cardinality 
$\Lambda$. 
Let $\gamma > 0$ with $k \in \Z_{\geq 0}$ such that 
$\gamma \in (k,k+1]$, and let $q \in \{0, \ldots ,k\}$.
Define $\Sigma^{\ast}_q := \cup_{z \in \Sigma} \Tau_{z,q}$ 
and $\Sigma^{\ast}_k := \cup_{z \in \Sigma} \Tau_{z,k}$.
Let 
$\cf = \{ f_1 , \ldots , f_{\n} \} \subset \Lip(\gamma,\Sigma,W) 
\setminus \{0\}$.
Let $a_1 , \ldots , a_{\n} \in \R \setminus \{0\}$ and define
$\vph := \sum_{i=1}^{\n} a_i f_i$.
Assume that the subset 
$\left\{ \sigma(\vph) : \sigma \in \Sigma_k^{\ast} \right\} 
\subset \R$ and, for every $i \in \{1, \ldots , \n\}$, 
the subset
$\left\{ \sigma(f_i) : \sigma \in \Sigma^{\ast}_k \right\} 
\subset \R$ have already been computed.
Then for any $P \subset \Sigma$ with cardinality $\#(P) = m$
the complexity cost of applying the 
\textbf{HOLGRIM Recombination Step}, with the subset $P$, the 
integer $s$, and the order level $q \in \{0, \ldots , k\}$ there 
as $P$, $s$, and $q$ here respectively, is 
\beq
    \label{eq:recomb_step_cost}
        \cO \left( smc D(d,k)( \n  + c D(d,q) \Lambda) + 
        sm^3 c^3 D(d,k)^3 
        \log \left( \frac{\n}{mcD(d,k)} \right) \right).
\eeq
\end{lemma}

\begin{proof}[Proof of Lemma \ref{lemma:HOLGRIM_recomb_step_cost}]
Let $\n,\Lambda,m,c,d \in \Z_{\geq 1}$.
Let $V$ and $W$ be real Banach spaces of dimensions $d$ and 
$c$ respectively.
Assume the tensor powers of $V$ are all equipped with admissible
norms (cf. Definition \ref{admissible_tensor_norm}). 
Let $\Sigma \subset V$ be a finite subset with cardinality 
$\Lambda$. 
Let $\gamma > 0$ with $k \in \Z_{\geq 0}$ such that 
$\gamma \in (k,k+1]$, and let $q \in \{0, \ldots ,k\}$.
Define $\Sigma^{\ast}_q := \cup_{z \in \Sigma} \Tau_{z,q}$ 
and $\Sigma^{\ast}_k := \cup_{z \in \Sigma} \Tau_{z,k}$.
Let 
$\cf = \{ f_1 , \ldots , f_{\n} \} \subset \Lip(\gamma,\Sigma,W) 
\setminus \{0\}$.
Let $a_1 , \ldots , a_{\n} \in \R \setminus \{0\}$ and define
$\vph := \sum_{i=1}^{\n} a_i f_i$.
Assume that the subset 
$\left\{ \sigma(\vph) : \sigma \in \Sigma_k^{\ast} \right\} 
\subset \R$ and, for every $i \in \{1, \ldots , \n\}$, 
the subset
$\left\{ \sigma(f_i) : \sigma \in \Sigma^{\ast}_k \right\} 
\subset \R$ have already been computed.
Let $P \subset \Sigma$ have cardinality $\#(P) = m$.

Consider applying the 
\textbf{HOLGRIM Recombination Step} with the subset $P$ the 
integer $s$, and the order level $q \in \{0, \ldots , k\}$ there 
as $P$, $s$, and $q$ here respectively.
Define (cf. \eqref{eq:L_star_def_recomb_step}) 
$L(k) := \cup_{z \in P} \Tau_{z,k} \subset 
\Lip(\gamma,\Sigma,W)^{\ast}$.
Since $\# (P) = m$, it follows that (cf. 
\eqref{eq:L_star_cardinality_recomb_step}
the cardinality of $L(k)$ is 
$\# \left( L(k) \right) = mcD(d,k)$.

Consider a fixed $j \in \{1 , \ldots , s\}$.
The complexity cost of shuffling of the elements in 
$L(k)$ to obtain $L_j(k)$ in 
\textbf{HOLGRIM Recombination Step} 
\ref{HOLGRIM_recomb_step_B} is 
$\cO(mcD(d,k))$.

In \textbf{HOLGRIM Recombination Step} 
\ref{HOLGRIM_recomb_step_C}, recombination is applied via
Lemma 3.1 in \cite{LM22}. 
In particular, Lemma 3.1 in \cite{LM22} is applied with
the Banach space $X$, the integers $\n,m \in \Z_{\geq 1}$, 
the subset $L \subset X^{\ast}$, the finite subset 
$\cf \subset X$, 
and the target $\vph \in \Span(\cf) \subset X$ in that result
as $\Lip(\gamma,\Sigma,W)$, $\n, mcD(d,k) + 1 \in \Z_{\geq 1}$, 
$L_j(k)$, and  
$\vph \in \Span(\cf) \subset \Lip(\gamma,\Sigma,W)$ here
respectively, to find $u_j \in \Span(\cf)$ satisfying 
\beq
    \label{eq:complex_cost_recomb_outcome} 
        \max \left\{ \left| \sigma (\vph-u_j) \right| : 
        \sigma \in L_j(k) \right\} 
        \leq \frac{\ep_0}{c d^k}.
\eeq
It is established in Lemma 3.1 \cite{LM22} that the 
complexity cost of this application is
\beq
    \label{eq:complex_cost_recombination_app}
        \cO \left( \n m c D(d,k) + m^3 c^3 D(d,k) 
        \log \left( \frac{\n}{m c D(d,k)} \right) \right).
\eeq
Further, since $\# \left( L_j(k) \right) = m c D(d,k)$, 
it follows from Lemma 3.1 in \cite{LM22} that the cardinality 
of $\support(u_j)$ satisfies
$\# \support(u_j) \leq m c D(d,k) + 1$. 
Thus, since we already have access to 
$\left\{ \sigma(\vph) : \sigma \in \Sigma_k^{\ast} \right\} 
\subset \R$ and, for every $i \in \{1, \ldots , \n\}$, 
the subset
$\left\{ \sigma(f_i) : \sigma \in \Sigma^{\ast}_k \right\} 
\subset \R$ without additional computation and 
$\Sigma^{\ast}_q \subset \Sigma^{\ast}_k$, 
the complexity cost of computing 
$E[u_j] := \max \left\{ 
|\sigma(\vph - u_j)| : \sigma \in \Sigma^{\ast}_q \right\}$
is 
\beq
    \label{eq:complex_cost_E[u_j]_compute}
        \cO \left( m c D(d,k) \# \left( \Sigma^{\ast}_q \right)
        \right)
        =
        \cO \left( m c^2 \Lambda D(d,k) D(d,q) \right).
\eeq
Therefore, the combination of 
\eqref{eq:complex_cost_recombination_app} 
and 
\eqref{eq:complex_cost_E[u_j]_compute}
yields that the complexity cost of carrying out
\textbf{HOLGRIM Recombination Step} \ref{HOLGRIM_recomb_step_A}, 
\ref{HOLGRIM_recomb_step_B}, \ref{HOLGRIM_recomb_step_C}, and 
\ref{HOLGRIM_recomb_step_D} for all $j \in \{1, \ldots , s\}$ 
is 
\begin{multline}
    \label{recomb_step_cost_v1}
        \cO \left( s m c^2 \Lambda D(d,k) D(d,q)
        +
        s \n m c D(d,k) + s m^3 c^3 D(d,k) 
        \log \left( \frac{\n}{m c D(d,k)} \right)
        \right)
        = \\
        \cO \left( 
        smc D(d,k) \left( \n + c \Lambda D(d,q) \right) 
        + s m^3 c^3 D(d,k) 
        \log \left( \frac{\n}{m c D(d,k)} \right)
        \right)
\end{multline}
The complexity cost of the final selection of 
$u := \argmin \left\{ E[w] : w \in \{u_1 , \ldots , u_s\} \right\}$
is $\cO(s)$.
Combined with \eqref{recomb_step_cost_v1}, this yields 
that the complexity cost of the entire 
\textbf{HOLGRIM Recombination Step} is 
\beq
    \label{recomb_step_cost_v2}
        \cO \left( 
        smc D(d,k) \left( \n + c \Lambda D(d,q) \right) 
        + s m^3 c^3 D(d,k) 
        \log \left( \frac{\n}{m c D(d,k)} \right)
        \right).
\eeq
as claimed in \eqref{eq:recomb_step_cost}.
This completes the proof of Lemma 
\ref{lemma:HOLGRIM_recomb_step_cost}.
\end{proof}
\vskip 4pt
\noindent
We now establish an upper bound for the complexity cost of 
the \textbf{HOLGRIM} algorithm via repeated use of 
Lemmas \ref{lemma:HOLGRIM_ext_step_cost} and 
\ref{lemma:HOLGRIM_recomb_step_cost}.
This is the content of the following result.

\begin{lemma}[\textbf{HOLGRIM Complexity Cost}]
\label{lemma:complex_cost_HOLGRIM_alg} 
Let $\n,M,\Lambda,c,d \in \Z_{\geq 1}$ and 
$\ep > \ep_0 \geq 0$.
Let $\gamma > 0$ with $k \in \Z_{\geq 0}$ such that 
$\gamma \in (k,k+1]$.
Let $q \in \{0, \ldots , k\}$.
Take $s_1 , \ldots , s_{M} \in \Z_{\geq 1}$ and 
$k_1 , \ldots , k_M \in \Z_{\geq 1}$ with 
\beq
    \label{eq:complex_cost_kappa_def}
        \kappa := k_1 + \ldots + k_M \leq 
        \min \left\{ \frac{\n - 1}{c D(d,k)} , \Lambda 
        \right\}.
\eeq
For $j \in \{1, \ldots , M\}$ let $\tau_j := \sum_{i=1}^j k_i$.
Let $V$ and $W$ be real finite-dimensional Banach spaces
of dimensions $d$ and $c$ respectively. 
Assume that the tensor powers of $V$ are all equipped with 
admissible norms (cf. Definition \ref{admissible_tensor_norm}),
and let $\Sigma \subset V$ be a finite subset of cardinality 
$\Lambda$. Let 
$\cf = \{ f_1 , \ldots , f_{\n} \} \subset \Lip(\gamma,\Sigma,W) 
\setminus \{0\}$.
Fix a choice of $A_1 , \ldots , A_{\n} \in \R$ such that, for 
every $i \in \{1, \ldots , \n\}$, we have 
$||f_i||_{\Lip(\gamma,\Sigma,W)} \leq A_i$.
Let $a_1 , \ldots , a_{\n} \in \R \setminus \{0\}$ and define
$\vph := \sum_{i=1}^{\n} a_i f_i$.
Then the complexity cost of applying the \textbf{HOLGRIM}
algorithm to approximate $\vph$, with $\ep$ as the target 
accuracy, $\ep_0$ as the acceptable recombination error, 
$M$ as the maximum number of steps, $q$ as the order level, 
$s_1 , \ldots , s_M$ as the shuffle numbers, the integers 
$k_1 , \ldots , k_M$ as the integers 
$k_1 , \ldots , k_M \in \Z_{\geq 1}$ chosen in 
\textbf{HOLGRIM} \ref{HOLGRIM_A}, and the real numbers 
$A_1 , \ldots , A_{\n}$ as the \emph{scaling factors} chosen 
in \textbf{HOLGRIM} \ref{HOLGRIM_A}, is 
\beq
    \label{eq:complex_cost_HOLGRIM_cost}
        \cO \left( c D(d,k) \n \Lambda + 
        \sum_{j=1}^M 
        c s_j \tau_j D(d,k) \left( \n +
        c D(d,q) \Lambda \right)
        + 
        s_j \tau_j^3 c^3 D(d,k)^3 \log \left(
        \frac{\n}{c \tau_j D(d,k)} \right)
        \right).
\eeq
\end{lemma}

\begin{proof}[Proof of Lemma \ref{lemma:complex_cost_HOLGRIM_alg}]
Let $\n,M,\Lambda,c,d \in \Z_{\geq 1}$ and 
$\ep > \ep_0 \geq 0$.
Let $\gamma > 0$ with $k \in \Z_{\geq 0}$ such that 
$\gamma \in (k,k+1]$.
Let $q \in \{0, \ldots , k\}$.
Take $s_1 , \ldots , s_{M} \in \Z_{\geq 1}$ and 
$k_1 , \ldots , k_M \in \Z_{\geq 1}$ with 
\beq
    \label{eq:complex_cost_kappa_def_pf}
        \kappa := k_1 + \ldots + k_M \leq 
        \min \left\{ \frac{\n - 1}{c D(d,k)} , \Lambda 
        \right\}.
\eeq
For $j \in \{1, \ldots , M\}$ let $\tau_j := \sum_{i=1}^j k_i$.
Let $V$ and $W$ be real finite-dimensional Banach spaces
of dimensions $d$ and $c$ respectively. 
Assume that the tensor powers of $V$ are all equipped with 
admissible norms (cf. Definition \ref{admissible_tensor_norm}),
and let $\Sigma \subset V$ be a finite subset of cardinality 
$\Lambda$. Let 
$\cf = \{ f_1 , \ldots , f_{\n} \} \subset \Lip(\gamma,\Sigma,W) 
\setminus \{0\}$.
Fix a choice of $A_1 , \ldots , A_{\n} \in \R$ such that, for 
every $i \in \{1, \ldots , \n\}$, we have 
$||f_i||_{\Lip(\gamma,\Sigma,W)} \leq A_i$.
Let $a_1 , \ldots , a_{\n} \in \R \setminus \{0\}$ and define
$\vph := \sum_{i=1}^{\n} a_i f_i$.

Consider applying the \textbf{HOLGRIM}
algorithm to approximate $\vph$ with $\ep$ as the target 
accuracy, $\ep_0$ as the acceptable recombination error, 
$M$ as the maximum number of steps, $q$ as the order level, 
$s_1 , \ldots , s_M$ as the shuffle numbers, the integers 
$k_1 , \ldots , k_M$ as the integers 
$k_1 , \ldots , k_M \in \Z_{\geq 1}$ chosen in 
\textbf{HOLGRIM} \ref{HOLGRIM_A}, and the real numbers 
$A_1 , \ldots , A_{\n}$ as the \emph{scaling factors} chosen 
in \textbf{HOLGRIM} \ref{HOLGRIM_A}.

Since the cardinality of $\cf$ is $\n$, the complexity cost of
the rescaling and sign alterations in Step \ref{HOLGRIM_B} 
of the \textbf{HOLGRIM} algorithm is $\cO(\n)$.
The complexity cost of computing the sets 
$\left\{ \sigma(f_i) : \sigma \in \Sigma^{\ast}_k \right\}$
for $i \in \{ 1 , \ldots , \n \}$ is 
$\cO\left( c\n \Lambda D(d,k)\right)$.
Subsequently, having access to the sets 
$\left\{ \sigma(f_i) : \sigma \in \Sigma^{\ast}_k \right\}$
for $i \in \{ 1 , \ldots , \n \}$ means that 
the complexity cost of computing the set
$\left\{ \sigma(\vph) : \sigma \in \Sigma^{\ast}_k \right\}$ is 
$\cO( c\n \Lambda D(d,k) )$.
Consequently, the total complexity cost of performing these
computations is $\cO (c \n \Lambda D(d,k))$.

We appeal to Lemma \ref{lemma:HOLGRIM_ext_step_cost}
to conclude that the complexity cost of performing the 
\textbf{HOLGRIM Extension Step} as in Step \ref{HOLGRIM_C}
of the \textbf{HOLGRIM} algorithm (i.e. with 
$P' := \emptyset$, $u := 0$, and $m := k_1$) is 
$\cO \left( c k_1 \Lambda D(d,k) \right)$.
An application of Lemma \ref{lemma:HOLGRIM_recomb_step_cost}, 
yields that the complexity cost of the use of
the \textbf{HOLGRIM Recombination Step} in Step 
\ref{HOLGRIM_C}
of the \textbf{HOLGRIM} algorithm 
(i.e. with the subset $P := \Sigma_1$ and the shuffle number 
$s:=s_1$) is 
$\cO \left( c s_1k_1 D(d,k) \left( \n + c \Lambda D(d,q) \right) + 
s_1 c^3 k_1^3 D(d,k)^3 \log \left( \n / c k_1 D(d,k) 
\right) \right)$.
Hence the complexity cost of Step \ref{HOLGRIM_C}
of the \textbf{HOLGRIM} algorithm is 
$\cO \left( c s_1 k_1 D(d,k) 
\left( \n + c \Lambda D(d,q) \right) + 
s_1 c^3 k_1^3 D(d,k)^3 \log 
\left( \frac{\n}{c k_1 D(d,k)} \right) \right)$.

In the case that $M=1$ we can already conclude that the 
complexity cost of performing the \textbf{HOLGRIM} algorithm 
is 
\beq
    \label{eq:complex_cost_HOLGRIM_cost_M=1}
        \cO \left( c D(d,k) \n \Lambda  + c s_1 k_1 D(d,k) 
        \left( \n + c \Lambda D(d,q) \right) + 
        s_1 c^3 k_1^3 D(d,k)^3 \log 
        \left( \frac{\n}{c k_1 D(d,k)} \right) \right)
\eeq
as claimed in \eqref{eq:complex_cost_HOLGRIM_cost}.
Now suppose that $M \geq 2$. We assume that all $M$ steps 
of the \textbf{HOLGRIM} algorithm are completed without
early termination since this is the case that will maximise
the complexity cost.
Under this assumption, for $j \in \{1, \ldots , M\}$ 
let $u_j \in \Span(\cf)$ denote the approximation of $\vph$ 
returned after step $j$ of the
\textbf{HOLGRIM} algorithm is completed.
Recall that $\tau_j := \sum_{i=1}^j k_i$.
The approximation $u_j$ is 
obtained by applying recombination via Lemma 3.1 in \cite{LM22} 
to find an approximation 
that is within $\ep_0$ of $\vph$ on a subset of 
$c\tau_jD(d,k)$ linear functionals from $\Sigma^{\ast}_k$
(cf. Step \ref{HOLGRIM_D} of the \textbf{HOLGRIM} algorithm).
Lemma 3.1 in \cite{LM22} therefore ensures that
$\# \support(u_j) \leq 1 + c \tau_j D(d,k)$.
In order to establish an upper bound for the maximal complexity
cost, we assume that we are in the most costly situation 
in which $\# \support(u_j) = 1 + c \tau_j D(d,k)$.

Let $t \in \{2, \ldots , M\}$ and consider 
Step \ref{HOLGRIM_D} of the \textbf{HOLGRIM} 
algorithm for the $s$ there as $t$ here.
Since $\# \support ( u_{t-1}) = 1 + c \tau_{t-1} D(d,k)$
and we have already computed the set
$\left\{ \sigma(\vph) : \sigma \in \Sigma_k^{\ast} \right\}$
and, for every $i \in \{1 , \ldots , \n\}$, the set
$\left\{ \sigma(f_i) : \sigma \in \Sigma^{\ast}_k \right\}$,
the complexity cost of computing the set 
$\left\{ \left| \sigma(\vph - u_{t-1}) \right| : 
\sigma \in \Sigma^{\ast}_q \right\}$ for the purpose of 
checking the termination criterion is 
$\cO \left( (1 + c \tau_{t-1} D(d,k) ) 
\# \left( \Sigma^{\ast}_q \right) \right) 
= \cO \left( c^2 \tau_{t-1} D(d,q) D(d,k) \Lambda \right)$.
Since $\# \support ( u_{t-1}) = 1 + c \tau_{t-1} D(d,k)$,
Lemma \ref{lemma:HOLGRIM_ext_step_cost} tells us that
the complexity cost of performing the 
\textbf{HOLGRIM Extension Step} as in Step \ref{HOLGRIM_D}
of the \textbf{HOLGRIM} algorithm (i.e. with 
$P' := \Sigma_{t-1}$, $u := u_{t-1}$, 
and $m := k_t$) is 
$\cO \left( c \left( k_t + c \tau_t D(d,k) \right) 
\Lambda D(d,q) \right) = 
\cO \left( c^2 \tau_t D(d,k) D(d,q) \Lambda \right)$.
Lemma \ref{lemma:HOLGRIM_recomb_step_cost} yield that 
the complexity cost of the use of the
\textbf{HOLGRIM Recombination Step} in Step \ref{HOLGRIM_D}
of the \textbf{HOLGRIM} algorithm 
(i.e. with $P := \Sigma_{t}$ and $s:=s_t$) is 
$\cO \left( c \tau_t s_t D(d,k) 
\left( \n + c \Lambda D(d,q) \right) + 
s_t c^3 \tau_t^3 D(d,k)^3
\log \left( \frac{\n}{c\tau_t D(d,k)} \right) \right)$.
Hence the entirety of Step \ref{HOLGRIM_D} 
(for $t$ here playing the role of $s$ there) 
has a complexity cost of 
\beq
    \label{eq:complex_cost_step_D_t_cost}
        \cO \left( c s_t \tau_t D(d,k) \n +
        c^2 s_t \tau_t D(d,q) D(d,k) \Lambda 
        + 
        s_t \tau_t^3 c^3 D(d,k)^3 \log \left(
        \frac{\n}{c \tau_t D(d,k)} \right)
        \right).
\eeq
Summing together the complexity costs arising in 
\eqref{eq:complex_cost_step_D_t_cost} for each 
$t \in \{2, \ldots , M\}$ yields that the complexity cost
of carrying out step \ref{HOLGRIM_D} of the 
\textbf{HOLGRIM} algorithm for every $t \in \{2, \ldots , M \}$
is
\beq
    \label{eq:complex_cost_step_D_cost}
        \cO \left(
        \sum_{j=2}^M 
        c s_j \tau_j D(d,k) \left( \n +
        c D(d,q) \Lambda \right)
        + 
        s_j \tau_j^3 c^3 D(d,k)^3 \log \left(
        \frac{\n}{c \tau_j D(d,k)} \right)
        \right).
\eeq
Having previously established the complexity cost of 
carrying out Steps 
\ref{HOLGRIM_A}, \ref{HOLGRIM_B}, and \ref{HOLGRIM_C} 
of the \textbf{HOLGRIM} algorithm in 
\eqref{eq:complex_cost_HOLGRIM_cost_M=1},
the combination of \eqref{eq:complex_cost_HOLGRIM_cost_M=1}
and \eqref{eq:complex_cost_step_D_cost} yields that
the complexity cost
of performing the entire \textbf{HOLGRIM} algorithm is
(recalling that $\tau_1 := k_1$)
\beq
    \label{eq:complex_cost_HOLGRIM_cost_pf}
        \cO \left( c D(d,k) \n \Lambda + 
        \sum_{j=1}^M 
        c s_j \tau_j D(d,k) \left( \n +
        c D(d,q) \Lambda \right)
        + 
        s_j \tau_j^3 c^3 D(d,k)^3 \log \left(
        \frac{\n}{c \tau_j D(d,k)} \right)
        \right)
\eeq
as claimed in \eqref{eq:complex_cost_HOLGRIM_cost}.
This completes the proof of Lemma 
\ref{lemma:complex_cost_HOLGRIM_alg}.
\end{proof}
\vskip 4pt
\noindent
We end this section by explicitly recording the complexity 
cost upper bound arising in Lemma 
\ref{lemma:complex_cost_HOLGRIM_alg} for some particular 
parameter choices. 
In both examples we assume that 
$1 + c D(d,k) \Lambda < \n$.

First consider the choices that $M := 1$, $k_1 := \Lambda$ and 
$s_1 := 1$.
This corresponds to making a single application of recombination 
to find an approximation $u$ of $\vph$ that is within $\ep_0$
of $\vph$ throughout $\Sigma$ in the pointwise sense that, 
for every $z \in \Sigma$, we have
$\Lambda^q_{\vph-u}(z) \leq \ep_0$.
Lemma \ref{lemma:complex_cost_HOLGRIM_alg} yields that the 
complexity cost of doing this is
$\cO \left( c D(d,k) \n \Lambda + c^2 D(d,q)D(d,k) \Lambda^2 
+ c^3 D(d,k)^3 \Lambda^3 \log 
\left( \frac{\n}{cD(d,k)\Lambda} \right) \right)$.

Secondly consider the choices that 
$M \in \Z_{\geq 1}$, $k_1 = \ldots = k_M =1$, and arbitrary fixed 
$s_1 , \ldots , s_M \in \Z_{\geq 1}$.
Let $P \subset \Sigma$ denote the collection of points that is 
inductively grown in the \textbf{HOLGRIM} algorithm.
These choices correspond to adding a single new points from 
$\Sigma$ to $P$ at each step of the \textbf{HOLGRIM} algorithm.
For each $j \in \{1, \ldots , M\}$ we have 
$\tau_j := \sum_{i=1}^j k_i = j$. 
Lemma \ref{lemma:complex_cost_HOLGRIM_alg} yields that the 
complexity cost of doing this is
$\cO \left( c D(d,k) \n \Lambda + \sum_{j=1}^{M} 
c s_j j D(d,k) \left( \n + c D(d,q) \Lambda \right)  
+ s_j c^3 j^3 D(d,k)^3  \log 
\left( \frac{\n}{cjD(d,k)} \right) \right)$.
If we further restrict to only allowing a single application of
recombination at each step (i.e. impose that 
$s_1 = \ldots = s_M = 1$) then the complexity cost is
$\cO \left( c D(d,k) \n \Lambda + 
c M^2 D(d,k) \left( \n + c D(d,q) \Lambda \right)
+ \sum_{j=1}^{M} c^3 j^3 D(d,k)^3 \log 
\left( \frac{\n}{c j D(d,k)} \right) \right)$.

In particular, if we take $M := \Lambda$ (which corresponds 
to allowing for the possibility that the subset $P$ eventually
becomes the entire set $\Sigma$) then the complexity cost is
\beq
    \label{eq:complex_cost_HOLGRIM_one_at_time_A}
        \cO \left( 
        c D(d,k) \n \Lambda^2 + 
        c^2 D(d,q)D(d,k) \Lambda^3 
        + \sum_{j=1}^{\Lambda} c^3 j^3 D(d,k)^3 \log 
        \left( \frac{\n}{c j D(d,k)} \right)\right).
\eeq 
If the integer $\n$ is large enough that 
$e^{1/3} c D(d,k) \Lambda < \n$
then for every $j \in \{1, \ldots , \Lambda \}$ we have
the estimate
$c^3 j^3 D(d,k)^3 \log ( \n / c j D(d,k) ) \leq 
c^3 D(d,k)^3 \Lambda^3 \log (\n / c D(d,k) \Lambda)$.
Thus, under these assumptions, the complexity cost 
in \eqref{eq:complex_cost_HOLGRIM_one_at_time_A} is no
worse than
\beq
    \label{eq:complex_cost_HOLGRIM_one_at_time_B}
        \cO \left( 
        c D(d,k) \n \Lambda^2 + 
        c^2 D(d,q)D(d,k) \Lambda^3 
        + c^3 D(d,k)^3 \Lambda^4 \log 
        \left( \frac{\n}{c D(d,k) \Lambda} \right)\right).
\eeq

\section{Theoretical Guarantees Inherited from GRIM}
\label{sec:HOLGRIM_conv_anal_via_GRIM_conv_thm}
In this section we discuss the convergence guarantees that 
can be established for the \textbf{HOLGRIM} algorithm by 
directly appealing to the convergence theory for the 
\textbf{Banach GRIM} algorithm developed in Section 6 
of \cite{LM22}.
We discuss the limitations of this approach and 
motivate the distinct approach using the 
\emph{Lipschitz Sandwich Theorems} established in \cite{LM24} 
to obtain convergence results for the \textbf{HOLGRIM} algorithm
that is adopted in Section \ref{sec:HOLGRIM_conv_anal}.

Let $c,d \in \Z_{\geq 1}$, $V$ be a $d$-dimensional real Banach
space, and $W$ be a $c$-dimensional real Banach space. 
Suppose that the tensor powers of $V$ are all equipped with 
admissible norms (cf. Definition \ref{admissible_tensor_norm}) 
and that $\Sigma \subset V$ is a finite subset of cardinality 
$\Lambda \in \Z_{\geq 1}$.
Let $\n \in \Z_{\geq 1}$ and $\gamma > 0$ with $k \in \Z_{\geq 0}$
such that $\gamma \in (k,k+1]$.
Let $D = D(d,k)$ be the constant defined in 
\eqref{eq:D_ab_Q_ijab_def_not_sec}; that is 
$D(d,k) := \sum_{l=0}^k \be(d,l)$ where $\be(d,0) := 1$ and, 
for each 
$l \in \{1, \ldots , k\}$ if $k \geq 1$, 
$\be(d,l) := \frac{(d+l-1)\dots d}{l!}$
(cf. \eqref{eq:beta_ab_def_not_sec}).

Suppose, for every $i \in \{1, \ldots , \n\}$, that 
$f_i \in \Lip(\gamma,\Sigma,W)$ is not identically zero 
and define a finite subset 
$\cf := \{ f_1 , \ldots , f_{\n} \} \subset 
\Lip(\gamma,\Sigma,W)$.
For non-zero coefficients 
$a_1 , \ldots , a_{\n} \in \R \setminus \{0\}$ 
consider $\vph := \sum_{i=1}^{\n} a_i f_i \in \Span(\cf) 
\subset \Lip(\gamma,\Sigma,W)$. 
For each $i \in \{1, \ldots , \n\}$ fix a choice of 
$A_i \in \R_{>0}$ such that 
$||f_i||_{\Lip(\gamma,\Sigma,W)} \leq A_i$.
Finally, let $\ep > 0$ and $q \in \{0 , \ldots , k\}$, 
and fix a choice of $\ep_0 \in [0,\ep/cd^q)$.

Consider applying the \textbf{HOLGRIM} algorithm to 
approximate $\vph$ throughout $\Sigma$ with 
$M := \min \left\{ \frac{\n-1}{cD(d,k)} , \Lambda \right\}$ 
in \textbf{HOLGRIM} \ref{HOLGRIM_A},
the target accuracy
in \textbf{HOLGRIM} \ref{HOLGRIM_A} as $\ep$, 
the acceptable recombination error in 
\textbf{HOLGRIM} \ref{HOLGRIM_A} as $\ep_0$, 
the order level in \textbf{HOLGRIM} \ref{HOLGRIM_A} as $q$,
$s_1 = \ldots = s_M = 1$ as the shuffle numbers in 
\textbf{HOLGRIM} \ref{HOLGRIM_A}, 
$k_1 = \ldots = k_M = 1$ as the integers 
$k_1 , \ldots , k_M$ in \textbf{HOLGRIM} \ref{HOLGRIM_A}, and 
$A_1 , \ldots , A_{\n} \in \R_{>0}$ as the scaling factors in 
\textbf{HOLGRIM} \ref{HOLGRIM_A}.

Following \textbf{HOLGRIM} \ref{HOLGRIM_B}, for each 
$i \in \{1, \ldots , \n\}$ replace $a_i$ and $f_i$ by
$-a_i$ and $-f_i$ if the original $a_i$ satisfies 
$a_i < 0$. 
Subsequently, for each $i \in \{1, \ldots , \n\}$, 
define $h_i := f_i / A_i$ so that 
$h_i \in \Lip(\gamma,\Sigma,W)$ with 
$||h_i||_{\Lip(\gamma,\Sigma,W)} \leq 1$. 
If, for each $i \in \{1, \ldots , \n\}$, we now define 
$\al_i := a_i A_i > 0$, then 
$\vph = \sum_{i=1}^{\n} \al_i h_i$.

We briefly outline how one could directly 
apply the convergence results for the \textbf{Banach GRIM}
algorithm established in Section 6 of \cite{LM22}.
For this purpose momentarily assume, for some integer 
$n \in \{1 , \ldots , M-1\}$, that 
$L = \{ z_1 , \ldots , z_n \} \subset \Sigma$ is the collection 
of points in $\Sigma$ that have been selected during the 
$\textbf{HOLGRIM}$ algorithm. 
Let $u_n \in \Span(\cf)$ denote the approximation found by 
recombination, via the \textbf{HOLGRIM Recombination Step}, 
at step $n$.
Then we have that for every $j \in \{1, \ldots , n\}$
and for every linear functional 
$\sigma \in \Tau_{z_j,k} \subset \Lip(\gamma,\Sigma,W)^{\ast}$
that $|\sigma(\vph-u_n)| \leq \ep_0 / c d^k$.
Moreover, since the cardinality of 
$T_n := \cup_{j=1}^n \Tau_{z_j,k}$ is
$n c D(d,k)$ (cf. \eqref{eq:Tau_pl_card_not_sec}), 
it follows from Lemma 3.1 in \cite{LM22} that, for 
$Q_n := 1 + n c D(d,k)$, there are
coefficients $b_1 , \ldots , b_{Q_n} \geq 0$ and 
indices $e(1) , \ldots , e(Q_n) \in \{1, \ldots , \n\}$
for which $u_n = \sum_{s=1}^{Q_n} b_s h_{e(s)}$, and 
such that $\sum_{s=1}^{Q_n} b_s = \sum_{i=1}^{\n} \al_i = C$.
Consequently $||u_n||_{\Lip(\gamma,\Sigma,W)} \leq C$.

Thus we have all the hypotheses required to apply 
Lemma 6.4 in \cite{LM22} for the choices of 
$L$, $C$, $\th$, and $\th_0$ there as 
$T_n := \cup_{j=1}^n \Tau_{z_j,k}$, $C$, $\ep/c d^q$, and 
$\ep_0 / cd^k$ here respectively. 
By doing so, we conclude that whenever 
$\sigma \in \Reach_{||\cdot||_{l^1\left(\R^{Q_n - 1}\right)}} 
\left( T_n , 2 C , \ep / c d^q , \ep_0 / c d^k \right)$ 
we have $|\sigma(\vph - u_n)| \leq \ep / c d^q$. 
Here we use the notation from Subsection 6.1 in \cite{LM22} that
\begin{multline}
    \label{eq:Reach_def_recall}
        \Reach_{||\cdot||_{l^1\left(\R^{Q_n-1}\right)}} 
        \left( T_n , 2 C , 
        \frac{\ep}{c d^q} , \frac{\ep_0}{c d^k} \right)
        := \\
        \twopartdef
        {\bigcup_{0 < r < \frac{\ep}{\ep_0} d^{k-q} }
        \ovSpan_{||\cdot||_{l^1\left(\R^{Q_n - 1}\right)} } 
        \left( T_n , r \right)_{\frac{1}{2C} 
        \left( \frac{\ep}{c d^q} - \frac{\ep_0}{c d^k}r \right)}
        }
        {\ep_0 \neq 0}
        {\Span \left( T_n \right)_{\frac{\ep}{2C c d^q}} }
        {\ep_0 = 0}
\end{multline}
with 
\beq
    \label{eq:Span(L,r)_recall}
        \Span_{||\cdot||_{l^1\left(\R^{Q_n-1}\right)} }
        ( T_n , r)
        := 
        \left\{ 
        \sum_{\sigma \in T_n}
        a(\sigma) \sigma ~:~
            \forall \sigma \in T_n
            \text{ we have } a(\sigma) \in \R 
            \text{ and }
            \sum_{\sigma \in T_n }
            |a(\sigma)| \leq r
        \right\}.
\eeq
Finally, the subscripts in \eqref{eq:Reach_def_recall} denote
the fattening of the subset in $\Lip(\gamma,\Sigma,W)^{\ast}$.
To elaborate, this means that
$\ovSpan_{||\cdot||_{l^1\left( \R^{Q_n - 1}\right)} }
( T_n , r)_{\frac{1}{2C} 
\left( \frac{\ep}{c d^q} - \frac{\ep_0}{c d^k}r \right)}$
denotes the collection of linear functionals 
$\sigma \in \Lip(\gamma,\Sigma,W)^{\ast}$ for which there
exists a linear functional 
$\rho \in \ovSpan_{||\cdot||_{l^1 \left( \R^{Q_n - 1}\right)} }
( T_n  , r)$
whose distance to $\sigma$ in $\Lip(\gamma,\Sigma,W)^{\ast}$ is no 
greater than $\frac{1}{2C} 
\left( \frac{\ep}{c d^q} - \frac{\ep_0}{c d^k}r \right)$.
Similarly, 
$\Span (T_n)_{\frac{\ep}{2C c d^q}}$ 
denotes the collection of linear functionals 
$\sigma \in \Lip(\gamma,\Sigma,W)^{\ast}$ for which there
exists $\rho \in \Span (T_n)$
whose distance to $\sigma$ in $\Lip(\gamma,\Sigma,W)^{\ast}$ is no 
greater than $\frac{\ep}{2C c d^q}$.

Consider the choice of the next point $z_{n+1} \in \Sigma$ 
made by the \textbf{HOLGRIM} algorithm.
An examination of Step \ref{HOLGRIM_D} of the \textbf{HOLGRIM}
algorithm reveals that if the algorithm does not terminate 
before the next point is selected, then there is a linear 
functional $\sigma_{n+1} \in \Sigma_q^{\ast} := 
\cup_{z \in \Sigma} \Tau_{z,q}$ for which 
$|\sigma_{n+1}(\vph-u_n)| > \ep / c d^q$, and the point 
$z_{n+1} \in \Sigma$ is chosen to be such that 
$\sigma_{n+1} \in \Tau_{z_{n+1},q}$.
As a consequence, we have that 
$\sigma_{n+1} \notin 
\Reach_{||\cdot||_{l^1\left(\R^{Q_n - 1}\right)}} 
\left( T_n , 2 C , \ep / c d^q , \ep_0 / c d^k \right)$

Therefore, by following the strategy of Section 6 in \cite{LM22} 
verbatim, we could establish that the integer
\beq 
    \label{eq:HOLGRIM_step_num_bd_direct_GRIM}
        N := \max \left\{ d \in \Z_{\geq 1} : 
        \begin{array}{c}
            \text{There exists } z_1 , \ldots , z_d \in \Sigma 
            \text{ such that for every } j \in \{1, \ldots , d-1 \} 
            \\
            \text{there exists } \sigma \in \Tau_{z_{j+1},q} 
            \text{ with } \sigma \notin 
            \Reach_{||\cdot||_{l^1 \left(\R^{Q_j - 1} \right)}} 
            \left( T_j , 2 C , 
            \frac{\ep}{c d^q} , \frac{\ep_0}{c d^k} \right)
        \end{array}
        \right\}
\eeq
is an upper bound for the maximum number of steps that the 
\textbf{HOLGRIM} algorithm can complete before terminating.
The upper bound given in 
\eqref{eq:HOLGRIM_step_num_bd_direct_GRIM} suffers the 
following issues.

Firstly, the integer $N$ defined in 
\eqref{eq:HOLGRIM_step_num_bd_direct_GRIM} is determined 
by geometrical properties of the subset 
$\Sigma^{\ast}_k := \cup_{z \in \Sigma} \Tau_{z,k}$
of the dual space $\Lip(\gamma,\Sigma,W)^{\ast}$.
Computing $N$, or computing any of the geometric 
quantities that are established to provide upper bounds for
$N$ in Subsection 6.2 of \cite{LM22}, requires computing 
the distance between the linear functionals in $\Sigma^{\ast}_k$
in the dual space $\Lip(\gamma,\Sigma,W)^{\ast}$.
An immediate issue is how to compute such distances. The 
distance in the dual space $\Lip(\gamma,\Sigma,W)^{\ast}$ between 
$\sigma_1 , \sigma_2 \in \Lip(\gamma,\Sigma,W)^{\ast}$ is given by 
\beq
    \label{eq:dual_space_dist}
        \sup \left\{ | \sigma_1(\phi) - \sigma_2(\phi)| : 
        \phi \in \Lip(\gamma,\Sigma,W) \text{ with } 
        ||\phi||_{\Lip(\gamma,\Sigma,W)} \leq 1 \right\}.
\eeq
It is not immediately clear how to compute the supremum 
in \eqref{eq:dual_space_dist} in a cost-effective way.

Secondly, it is not clear how the distance in 
\eqref{eq:dual_space_dist} for $\sigma_1 , \sigma_2 \in 
\Sigma^{\ast}_k$ relates to the distance between the points
$z_1 , z_2 \in \Sigma$ for which $\sigma_1 \in \Tau_{z_1,k}$
and $\sigma_2 \in \Tau_{z_2,k}$.
The data of interest in our problem is the points in $\Sigma$; 
the linear functionals in $\Sigma^{\ast}_k$ are introduced to 
enable our use of recombination 
(cf. \textbf{HOLGRIM Recombination Step}). 
Consequently we would prefer upper bounds for the number of 
steps that the \textbf{HOLGRIM} algorithm can complete before
terminating to depend on the geometry of the points 
$\Sigma \subset V$ rather than the geometry of the linear 
functionals $\Sigma^{\ast}_k \subset \Lip(\gamma,\Sigma,W)^{\ast}$.
But it is not clear how the $\Lip(\gamma,\Sigma,W)^{\ast}$ 
distances involved in the definition of $N$ in 
\eqref{eq:HOLGRIM_step_num_bd_direct_GRIM} can be converted
to distances between the points in $\Sigma$ themselves. 

In Section \ref{sec:HOLGRIM_conv_anal}, 
we circumvent both these issues by utilising the 
\emph{Lipschitz Sandwich Theorems} established in \cite{LM24} 
to obtain an upper bound on the maximum number of steps that
the \textbf{HOLGRIM} algorithm can complete without terminating
that both depends on the geometry of the points $\Sigma$ rather
than the linear functionals $\Sigma^{\ast}_k$, and can be 
estimated knowing only the distances between the points in 
$\Sigma$.

\section{Point Separation Lemma}
\label{sec:HOLGRIM_sup_lemmata}
In this section we establish that the points selected 
during the \textbf{HOLGRIM} algorithm, under the choices that 
$M := \min \left\{ \frac{\n -1}{cD(d,k)} , \Lambda \right\}$ 
and, for every integer $t \in \{1, \ldots , M\}$, that
$s_t := 1$ and $k_t := 1$, have to be 
a definite $V$-distance apart.
This is the content of the following result.

\begin{lemma}[\textbf{HOLGRIM Point Separation}]
\label{lemma:HOLGRIM_dist_between_interp_points}
Let $\n , \Lambda , c , d \in \Z_{\geq 1}$, 
$\gamma > 0$ with $k \in \Z_{\geq 0}$ such that 
$\gamma \in (k,k+1]$, and fix a choice of $q \in \{0 , \ldots , k\}$.
Let $\ep , \ep_0 \geq 0$ be real numbers such that 
$0 \leq \ep_0 < \frac{\ep}{c d^q}$. 
Assume $V$ and $W$ are finite dimensional real Banach 
spaces, of dimensions $d$ and $c$ respectively, 
with $\Sigma \subset V$ a finite subset of cardinality 
$\Lambda$.
Assume that the tensor products of $V$ are all 
equipped with admissible norms (cf. Definition
\ref{admissible_tensor_norm}).  
Assume that for every $i \in \{1, \ldots , \n\}$
the element $f_i = \left( f_i^{(0)} , \ldots , 
f_i^{(k)} \right) \in \Lip(\gamma,\Sigma,W)$ is non-zero 
and define $\cf := \left\{ f_i ~:~ i \in \{1,\ldots,\n\}
\right\} \subset \Lip(\gamma,\Sigma,W)$.
Choose scalars $A_1 , \ldots , A_{\n} \in \R_{>0}$ such 
that, for every $i \in \{1, \ldots , \n\}$, we have 
$||f_i||_{\Lip(\gamma,\Sigma,W)} \leq A_i$.
Given $a_1 , \ldots , a_{\n} \in \R \setminus\{0\}$, 
define $\vph = \left( \vph^{(0)} , \ldots , \vph^{(k)}
\right) \in \Span(\cf) \subset \Lip(\gamma,\Sigma,W)$ 
and $C > 0$ by 
\beq
    \label{eq:dist_lemma_vph_C}
        (\bI) \quad 
        \vph := \sum_{i=1}^{\n} a_i f_i
        \quad \left(
        \begin{array}{c}
            \vph^{(l)} := \sum_{i=1}^{\n} a_i
            f^{(l)}_i \\
            \text{for every }
            l \in \{0, \ldots , k\} 
        \end{array}
        \right)
        \quad \text{and} \quad
        (\bII) \quad
        C := \sum_{i=1}^{\n} \left|a_i\right| A_i > 0.
\eeq
Then there exists a positive constant
$r = r (C, \gamma, \ep,\ep_0,q) >0$, given by
\beq
    \label{eq:dist_lemma_r}
        r := \sup \left\{ \lambda > 0 ~:~
        2C \lambda^{\gamma - q} + 
        \ep_0 e^{\lambda}
        \leq \frac{\ep}{c d^q} \right\},
\eeq
for which the following is true.

Consider applying the \textbf{HOLGRIM} algorithm to 
approximate $\vph$ throughout $\Sigma$ with the choices 
of $M := \min \left\{ \frac{\n - 1}{c D(d,k)} , \Lambda \right\}$,
the target accuracy threshold in \textbf{HOLGRIM} \ref{HOLGRIM_A} 
as $\th$, the acceptable recombination error in  
\textbf{HOLGRIM} \ref{HOLGRIM_A} as $\th_0$, the order level in 
\textbf{HOLGRIM} \ref{HOLGRIM_A} as $b$, 
the shuffle numbers in \textbf{HOLGRIM}
\ref{HOLGRIM_A} as $s_1 = \ldots = s_M = 1$, the integers
$k_1 , \ldots , k_M \in \Z_{\geq 1}$ in 
\textbf{HOLGRIM} \ref{HOLGRIM_A} 
as $k_1 = \ldots = k_M = 1$, and 
the scaling factors in \textbf{HOLGRIM} \ref{HOLGRIM_A} 
as $A_1 , \ldots , A_{\n}$.
Given $m \in \Z_{\geq 2}$, if the algorithm
reaches and carries
out the $m^{th}$ step without terminating,
let $u_m = \left( u_m^{(0)} , \ldots , u_m^{(k)} \right)
\in \Span(\cf) \subset \Lip(\gamma,\Sigma,W)$ denote the 
approximation found at the $m^{th}$ step and 
$\Sigma_m := \{ z_1 , \ldots , z_m \} \subset \Sigma$ 
denote the points selected 
such that at every $z \in \Sigma_m$ we have
$\Lambda^k_{\vph - u_m}(z) \leq \ep_0$
(cf. \textbf{HOLGRIM} \ref{HOLGRIM_C} and \ref{HOLGRIM_D}).
Then for every $z \in \Sigma$, the quantity 
$\Lambda^q_{\vph - u_m}(z):=
\max_{s \in \{0, \ldots , q\}}
\left|\left| \vph^{(s)}(z) - u_m^{(s)}(z)
\right|\right|_{\cl(V^{\otimes s};W)}$
satisfies 
\beq
    \label{eq:HOLGRIM_point_sep_general_est}
        \Lambda^q_{\vph - u_m}(z)
        \leq
        \min \left\{ 2C ~,~
        2C \max_{ h \in \{0, \ldots , q\} }
        \left\{\dist_V( z , \Sigma_m )^{\gamma -h}
        \right\} + 
        \ep_0 e^{\dist_V( z , \Sigma_m)}
        \right\}.
\eeq
Moreover, whenever $s,t \in \{ 1 , \ldots , m\}$ with 
$s \neq t$ we have that
\beq
    \label{eq:HOLGRIM_point_sep_dist_est}
        ||z_s - z_t||_V > r.
\eeq
\end{lemma}

\begin{proof}[Proof of Lemma
\ref{lemma:HOLGRIM_dist_between_interp_points}]
Let $\n , \Lambda , c , d \in \Z_{\geq 1}$, 
$\gamma > 0$ with $k \in \Z_{\geq 0}$ such that 
$\gamma \in (k,k+1]$, and fix a choice of $q \in \{0, \ldots , k\}$.
Let $\ep , \ep_0 \geq 0$ be real numbers 
such that $0 \leq \ep_0 < \frac{\ep}{c d^q}$. 
Assume $V$ and $W$ are finite dimensional real Banach 
spaces, of dimensions $d$ and $c$ respectively, 
with $\Sigma \subset V$ a finite subset of cardinality 
$\Lambda$.
Assume that the tensor products of $V$ are all 
equipped with admissible norms (cf. Definition
\ref{admissible_tensor_norm}).  
Assume that for every $i \in \{1, \ldots , \n\}$
the element $f_i = \left( f_i^{(0)} , \ldots , 
f_i^{(k)} \right) \in \Lip(\gamma,\Sigma,W)$ is non-zero 
and define $\cf := \left\{ f_i ~:~ i \in \{1,\ldots,\n\}
\right\} \subset \Lip(\gamma,\Sigma,W)$.
Choose scalars $A_1 , \ldots , A_{\n} \in \R_{>0}$ such 
that, for every $i \in \{1, \ldots , \n\}$, we have 
$||f_i||_{\Lip(\gamma,\Sigma,W)} \leq A_i$.
Given $a_1 , \ldots , a_{\n} \in \R \setminus\{0\}$, 
define $\vph = \left( \vph^{(0)} , \ldots , \vph^{(k)}
\right) \in \Span(\cf) \subset \Lip(\gamma,\Sigma,W)$ 
and $C > 0$ by 
\beq
    \label{eq:dist_lemma_vph_C_pf}
        (\bI) \quad 
        \vph := \sum_{i=1}^{\n} a_i f_i
        \quad \left(
        \begin{array}{c}
            \vph^{(l)} := \sum_{i=1}^{\n} a_i
            f^{(l)}_i \\
            \text{for every }
            l \in \{0, \ldots , k\} 
        \end{array}
        \right)
        \quad \text{and} \quad
        (\bII) \quad
        C := \sum_{i=1}^{\n} \left|a_i\right| A_i > 0.
\eeq
With a view to applying the \textbf{HOLGRIM} algorithm
to approximate $\varphi$ on $\Sigma$, for each 
$i \in \{1 , \ldots , \n\}$ let $\tilde{a}_i := |a_i|$
and $\tilde{f}_i$ be given by $f_i$ if $a_i > 0$ and 
$-f_i$ if $a_i < 0$.
Observe that for every $i \in \{1, \ldots , \n\}$
we have that $\left|\left| \tilde{f}_i 
\right|\right|_{\Lip(\gamma,\Sigma,W)}
= || f_i ||_{\Lip(\gamma,\Sigma,W)}$. Moreover, 
we also have that 
$\tilde{a}_1 , \ldots , \tilde{a}_{\n} > 0$ and that
$\vph = \sum_{i=1}^{\n} \tilde{a}_i \tilde{f}_i$.
Further, we rescale 
$\tilde{f}_i$ for each $i \in \{1, \ldots , \n\}$
to have unit $\Lip(\gamma,\Sigma,W)$ norm.
That is (cf. \textbf{HOLGRIM} \ref{HOLGRIM_B}), for each 
$i \in \{1 , \ldots , \n\}$ set 
$h_i := \tilde{f}_i / A_i$ and 
$\al_i := \tilde{a}_i A_i$. 
A particular consequence is that for every 
$i \in \{1, \ldots , \n\}$ we have 
$||h_i||_{\Lip(\gamma,\Sigma,W)} \leq 1$.
Additional consequences are that $C$ satisfies  
\beq
    \label{eq:dist_lemma_new_C}
        C = \sum_{i=1}^{\n} \left|a_i\right| A_i
        = \sum_{i=1}^{\n} \tilde{a}_i A_i
        = \sum_{i=1}^{\n} \al_i,
\eeq
and, for every $i \in \{1 , \ldots , \n\}$, that 
$\al_i h_i = \tilde{a}_i \tilde{f}_i = a_i f_i$. 
Thus the expansion for $\vph$ in (\bI) of
\eqref{eq:dist_lemma_vph_C_pf}
is equivalent to
\beq    
    \label{lip_k_varphi_h_dist_lemma}
        \varphi = \sum_{i=1}^{\n} 
        \al_i h_i,
        \qquad \text{and hence} \qquad
        || \varphi ||_{\Lip(\gamma,\Sigma,W)}
        \leq
        \sum_{i=1}^{\n} \al_i || h_i ||_{\Lip(\gamma,\Sigma,W)}
        \leq
        \sum_{i=1}^{\n} \al_i
        \stackrel{\eqref{eq:dist_lemma_new_C}}{=} C.
\eeq
Let $D = D(d,k)$ be the constant defined in 
\eqref{eq:D_ab_Q_ijab_def_not_sec}; that is
\beq
    \label{eq:dist_lemma_D_pf}
        D \stackrel{
        \eqref{eq:D_ab_Q_ijab_def_not_sec}
        }{=}
        \sum_{s=0}^k \be(d,s) 
        \stackrel{
        \eqref{eq:beta_ab_def_not_sec}
        }{=} 
        \sum_{s=0}^k 
        \left( 
        \begin{array}{c}
            d + s - 1 \\
            s
        \end{array}
        \right).
\eeq
Finally, define 
$r = r(C, \gamma, \ep, \ep_0,q) >0$ by
(cf. \eqref{eq:dist_lemma_r})
\beq
    \label{eq:dist_lemma_r_pf}
        r := \sup \left\{
        \lambda > 0 :
        2C \lambda^{\gamma - q} + 
        \ep_0 e^{\lambda}
        \leq
        \frac{\ep}{c d^q}
        \right\}.
\eeq
Now consider applying the \textbf{HOLGRIM} algorithm to 
approximate $\vph$ throughout $\Sigma$ with the choices 
of $M := \min \left\{ \frac{\n - 1}{c D(d,k)} , \Lambda \right\}$,
the target accuracy threshold in \textbf{HOLGRIM} \ref{HOLGRIM_A} 
as $\ep$, the acceptable recombination error in \textbf{HOLGRIM}
\ref{HOLGRIM_A} as $\ep_0$, the order level in \textbf{HOLGRIM} 
\ref{HOLGRIM_A} as $q$, the shuffle numbers in \textbf{HOLGRIM}
\ref{HOLGRIM_A} as $s_1 = \ldots = s_M = 1$, the integers
$k_1 , \ldots , k_M \in \Z_{\geq 1}$ in 
\textbf{HOLGRIM} \ref{HOLGRIM_A} 
as $k_1 = \ldots = k_M = 1$, and the scalars 
$A_1 , \ldots , A_{\n}$ as the scaling factors in 
\textbf{HOLGRIM} \ref{HOLGRIM_A}

Suppose that $m \in \Z_{\geq 2}$ and that the
\textbf{HOLGRIM}
algorithm reaches and carries out the $m^{th}$ step
without terminating.
Let $\Sigma_m = \{ z_1 , \ldots , z_m \}
\subset \Sigma$ denote the points chosen after
the $m^{th}$ step is completed. 
Then for every $l \in \{1, \ldots , m\}$, if we let 
$\Sigma_l := \{ z_1 , \ldots , z_l \} \subset
\Sigma$, we have, recalling 
\textbf{HOLGRIM} \ref{HOLGRIM_C} and \ref{HOLGRIM_D},
that recombination, via Lemma 3.1 in \cite{LM22},
has found an approximation $u_l = \left( u_l^{(0)} ,
\ldots , u^{(n)}_l \right) 
\in \Span(\cf) \subset \Lip(\gamma,\Sigma,W)$ of $\vph$ 
satisfying, for 
each $s \in \{1 , \ldots , l\}$, that
$\Lambda^k_{\vph - u_l}(z_s) \leq \ep_0$.
That is, given any $j \in \{0, \ldots , k\}$
and any $s \in \{1, \ldots , l\}$, we have
$\left|\left|\vph^{(j)}(z_s) - u^{(j)}_l(z_s) 
\right|\right|_{\cl(V^{\otimes j};W)} \leq \ep_0$.

Let $Q_l := 1 + c l D(d,k)$.
Then in the \textbf{HOLGRIM Recombination Step} at 
step $l$ of the \textbf{HOLGRIM} algorithm, 
recombination is applied, via Lemma 3.1 in \cite{LM22}, 
to a system of $Q_l$ real-valued equations
(cf. \eqref{eq:L_star_cardinality_recomb_step}).
Therefore recombination returns non-negative coefficients 
$b_{l,1} , \ldots , b_{l,Q_l} \geq 0$
and indices
$e_l(1) , \ldots , e_l(Q_l) \in \{1, \ldots , \n\}$
for which 
\beq
    \label{eq:dist_lemma_ul_expansion}
        u_l = \sum_{s=1}^{Q_l}
        b_{l,s} h_{e_l(s)}
        \quad \left(
        \begin{array}{c}
            u_l^{(j)} := \sum_{s=1}^{Q_l} 
            b_{l,s} h_{e_l(s)}^{(j)} \\
            \text{for every }
            j \in \{0, \ldots , n\} 
        \end{array}
        \right)
        \quad \text{and} \quad
        \sum_{s=1}^{Q_l} b_{l,s}
        =
        \sum_{i=1}^{\n} \al_i.
\eeq
A consequence of \eqref{eq:dist_lemma_ul_expansion}
is that
\beq
    \label{eq:dist_lemma_ul_norm_bd_one}
        || u_l ||_{\Lip(\gamma,\Sigma,W)}
        \leq 
        \sum_{s=1}^{Q_l} 
        b_{l,s} \left|\left| h_{e_l(s)}
        \right|\right|_{\Lip(\gamma,\Sigma,W)}
        \leq
        \sum_{s=1}^{Q_l} 
        b_{l,s}
        \stackrel{
        \eqref{eq:dist_lemma_ul_expansion}}
        {=}
        \sum_{i=1}^{\n} \al_i
        \stackrel{\eqref{eq:dist_lemma_new_C}}
        {=}
        C.
\eeq
Consider a point $x \in \Sigma$ and let 
$j \in \{1, \ldots , l\}$ be such that 
$\dist_V (x , \Sigma_l) = || x - z_j ||_V
= \dist_V(x , z_j)$.
Consider $h \in \{0, \ldots , q\}$.
By applying Lemma \ref{lemma:explicit_pointwise_est} 
to the function $\vph - u_l$,
with the $\Sigma$, $K_0$, $\ep_0$, $p$, $\gamma$, $l$, and $k$
of that result as $\Sigma$, $2C$, $\ep_0$, $z_j$, $\gamma$, $h$, 
$k$ here respectively, we deduce that
(cf. \eqref{eq:explicit_pointwise_est_lemma_conc})
\beq
    \label{eq:dist_lemma_pointwise_bound_1}
        \left|\left| \vph^{(h)}(x) - u^{(h)}_l(x) 
        \right|\right|_{\cl(V^{\otimes h};W)}
        \leq
        \min \left\{ 2C ~,~
        2C||x - z_j ||_V^{\gamma - h}
        +
        \ep_0 \sum_{s=0}^{k - h} \frac{1}{s!}
        || x - z_j ||_V^s
        \right\}.
\eeq
Recalling that $||x-z_j||_V = \dist_V(x, \Sigma_l)$, it follows
from \eqref{eq:dist_lemma_pointwise_bound_1} that
\beq
    \label{eq:dist_lemma_pointwise_bound_2}
        \left|\left| \vph^{(h)}(x) - u^{(h)}_l(x) 
        \right|\right|_{\cl(V^{\otimes h};W)}
        \leq
        \min \left\{ 2C ~,~
        2C\dist_V(x,\Sigma_l)^{\gamma - h}
        +
        \ep_0 e^{\dist_V(x,\Sigma_l)}
        \right\}.
\eeq
The arbitrariness of $h \in \{0, \ldots , q\}$ and
$x \in \Sigma$ allows us
to conclude that 
\eqref{eq:dist_lemma_pointwise_bound_2} is
valid for every $h \in \{0, \ldots , q\}$ and 
every $x \in \Sigma$. As a result we have that 
\beq
    \label{eq:dist_lemma_pointwise_bound_3}
        \max_{x \in \Sigma} \left\{
        \Lambda^q_{\vph - u_l}(x) \right\}
        \leq
        \min \left\{ 2C ~,~
        2C
        \max_{h \in \{0, \ldots , q\} }
        \left\{ \dist_V(x,\Sigma_l)^{\gamma - h}
        \right\}
        +
        \ep_0 e^{\dist_V(x,\Sigma_l)}
        \right\}.
\eeq
The estimate claimed in 
\eqref{eq:HOLGRIM_point_sep_general_est} is obtained
by taking $l:=m$ in 
\eqref{eq:dist_lemma_pointwise_bound_3}.

It remains only to establish the separation
of the points $z_1 , \ldots , z_m$ as claimed in
\eqref{eq:HOLGRIM_point_sep_dist_est}.
For this purpose, consider distinct 
$s,j \in \{1, \ldots , m\}$
and without loss of generality assume $s < j$.
By assumption the \textbf{HOLGRIM}
algorithm does not terminate on 
any of the first $m$ steps. Therefore
$\sigma_j := \argmax_{\sigma \in \Sigma^{\ast}_q}
| \sigma (\vph - u_{j-1})|$ must satisfy that
\beq
    \label{eq:dist_lemma_sigma_j_lower_bound}
        \left| \sigma_j \left( \vph - u_{j-1} \right) 
        \right| > \frac{\ep}{c d^q}.
\eeq 
Recall from the \textbf{HOLGRIM Extension Step} that 
$z_j \in \Sigma$ is taken to be the point for which 
$\sigma_j \in \Tau_{z_j,q}$. 
It is then a consequence of Lemma \ref{lemma:dual_norm_ests} 
that (cf. \eqref{eq:dual_norm_ests_lemma_pointwise_conc})
\beq
    \label{eq:dist_lemma_pf_sigm_j_eval_ub}
        \left| \sigma_j \left( \vph - u_{j-1} \right)
        \right| 
        \leq \Lambda^q_{\vph - u_{j-1}}(z_j).
\eeq
Thus we have that
\beq
    \label{eq:dist_lemma_pf_pointwise_diff_lb}
        \frac{\ep}{c d^q} 
        \stackrel{
        \eqref{eq:dist_lemma_sigma_j_lower_bound}
        }{<}
        \left| \sigma_j (\vph - u_{j-1}) \right|
        \stackrel{
        \eqref{eq:dist_lemma_pf_sigm_j_eval_ub}
        }{\leq}
        \Lambda^q_{\vph - u_{j-1}}(z_j).
\eeq
Note that $\vph , u_{j-1} \in \Lip(\gamma,\Sigma,W)$
and, via
\eqref{eq:dist_lemma_vph_C_pf} and 
\eqref{eq:dist_lemma_ul_norm_bd_one}, we have 
$||\vph||_{\Lip(\gamma,\Sigma,W)} , 
||u_{j-1}||_{\Lip(\gamma,\Sigma,W)} \leq C$.
Further, for any $t \in \{1, \ldots , j-1 \}$, we have 
$\Lambda^k_{\vph - u_{j-1}}(z_t) \leq \ep_0$.
Observe that the definition of $r$ here in 
\eqref{eq:dist_lemma_r_pf} matches the specification of
the constant $\de_0$ in \eqref{eq:pointwise_lip_sand_thm_de0} of 
the \emph{Pointwise Lipschitz Sandwich Theorem}
\ref{thm:pointwise_lip_sand_thm} for the choices
of the constants $K_0$, $\gamma$, $\ep$, $\ep_0$ and $l$ in that
theorem as $C$, $\gamma$, $\ep / c d^q$, $\ep_0$ and $q$ 
here respectively.
Thus, for each $t \in \{1, \ldots ,j-1\}$, 
we can apply the 
\emph{Pointwise Lipschitz Sandwich Theorem} 
\ref{thm:pointwise_lip_sand_thm}
with the choices that $B := \{ z_t\}$, 
$\psi := \vph$, and $\phi := u_{j-1}$ to conclude 
that for every $x \in \ovB_V(z_t , r) ~\cap~ \Sigma$
we have
(cf. \eqref{eq:pointwise_lip_sand_thm_imp})
\beq
    \label{eq:dist_lemma_pf_pointwise_diff_est_zi}
        \Lambda^q_{\vph - u_{j-1}}(x) 
        \leq \frac{\ep}{c d^q}. 
\eeq
The arbitrariness of $t \in \{1, \ldots , j-1\}$
allows us to conclude that 
\eqref{eq:dist_lemma_pf_pointwise_diff_est_zi} is valid
for every $t \in \{1 , \ldots , j-1\}$.
Together, \eqref{eq:dist_lemma_pf_pointwise_diff_lb}
and \eqref{eq:dist_lemma_pf_pointwise_diff_est_zi} mean 
that for every $t \in \{ 1 , \ldots , j-1\}$
we must have that 
$z_j \notin \ovB_V (z_t , r) \cap \Sigma$.
Since $s \in \{1 , \ldots , j-1\}$, we conclude 
that $z_j \notin \ovB_V (z_s ,r) \cap \Sigma$,
i.e. that $|| z_j - z_s ||_V > r$  
as claimed in \eqref{eq:HOLGRIM_point_sep_dist_est}.
This completes the proof of Lemma 
\ref{lemma:HOLGRIM_dist_between_interp_points}.
\end{proof}

\section{The HOLGRIM Convergence Theorem}
\label{sec:HOLGRIM_conv_anal}
In this section we establish a convergence result for the 
\textbf{HOLGRIM} algorithm and discuss its consequences.
The following theorem is our main convergence result for the 
\textbf{HOLGRIM} algorithm under the choices that 
$M := \min \left\{ \frac{\n -1}{c D(d,k)} , \Lambda \right\}$ and, 
for every $t \in \{1, \ldots , M\}$, that $s_t := 1$ and 
$k_t := 1$. 

\begin{theorem}[\textbf{HOLGRIM Convergence}]
\label{thm:HOLGRIM_conv_thm}
Let $\n, \Lambda , c , d \in \Z_{\geq 1}$, $\gamma > 0$ with 
$k \in \Z_{\geq 0}$ such that $\gamma \in (k,k+1]$, and 
fix a choice of $q \in \{0 , \ldots , k\}$.
Let $\ep , \ep_0$ be real numbers such that
$0 \leq \ep_0 < \frac{\ep}{c d^q}$.
Let $V$ and $W$ be finite dimensional real Banach 
spaces, of dimensions $d$ and $c$ respectively, 
with $\Sigma \subset V$ finite.
Assume that the tensor products of $V$ are all 
equipped with admissible norms (cf. Definition
\ref{admissible_tensor_norm}). 
Assume, for every $i \in \{1, \ldots , \n\}$, that
$f_i = \left( f_i^{(0)} , \ldots , 
f_i^{(k)} \right) \in \Lip(\gamma,\Sigma,W)$ is 
non-zero and define 
$\cf := \left\{ f_i ~:~ i \in \{1 , \ldots , \n\}
\right\} \subset \Lip(\gamma,\Sigma,W)$. 
Fix a choice of $A_1 , \ldots , A_{\n} \in \R_{>0}$ such that, for 
every $i \in \{1, \ldots , \n\}$, we have 
$||f_i||_{\Lip(\gamma,\Sigma,W)} \leq A_i$.
Given $a_1 , \ldots , a_{\n} \in \R \setminus \{0\}$, 
define $\vph = \left( \vph^{(0)} , \ldots , \vph^{(k)}
\right) \in \Span(\cf) \subset \Lip(\gamma, \Sigma,W)$ 
and constants $C , D > 0$ by
\beq
    \label{eq:HOLGRIM_conv_thm_vph_C_D}
        (\bI) \quad 
        \vph := \sum_{i=1}^{\n} a_i f_i,
        \quad \left(
        \begin{array}{c}
            \vph^{(l)} := \sum_{i=1}^{\n} a_i
            f^{(l)}_i \\
            \text{for every }
            l \in \{0, \ldots , k\} 
        \end{array}
        \right)
        \quad \text{and} \quad
        (\bII) \quad
        C := \sum_{i=1}^{\n} 
        |a_i| A_i > 0 
\eeq
and (cf. \eqref{eq:D_ab_Q_ijab_def_not_sec})
\beq
    \label{eq:HOLGRIM_conv_thm_D}
        D = D(d,k) := 
        \sum_{s=0}^k \be(d,s) =
        \mathlarger{\mathlarger{\sum}}_{s=0}^k
        \left(
        \begin{array}{c}
            d + s - 1 \\
            s
        \end{array}
        \right).
\eeq
Then there is a positive constant
$r = r (C,\gamma,\ep ,\ep_0,q) >0$, 
\beq
    \label{eq:HOLGRIM_conv_thm_r}
        r := \sup \left\{
        \lambda > 0 ~:~
        2C \lambda^{\gamma - q}
        +
        \ep_0 e^{\lambda} \leq \frac{\ep}{c d^q} \right\}
\eeq
and a non-negative integer
$N = N(\Sigma, C, \gamma, \ep, \ep_0,q) 
\in \Z_{\geq 0}$, given by the $r$-packing 
number of $\Sigma$
\beq
    \label{eq:HOLGRIM_conv_thm_N}
            N 
            := 
            \max \left\{ s \in \Z ~:~ 
                \exists ~
                z_1 , \ldots , z_s \in \Sigma 
                \text{ for which } 
                || z_a - z_b ||_V
                > r 
                \text{ if } a \neq b
            \right\},
\eeq
for which the following is true.

Consider applying the \textbf{HOLGRIM}
algorithm to approximate $\vph$ on $\Sigma$, with 
$M := \min \left\{ \frac{\n - 1}{cD} , \Lambda \right\}$,
$\ep$ as the target accuracy in \textbf{HOLGRIM} \ref{HOLGRIM_A}, 
$\ep_0$ as the acceptable recombination error in 
\textbf{HOLGRIM} \ref{HOLGRIM_A}, $q$ as  
the order level in \textbf{HOLGRIM} \ref{HOLGRIM_A}, 
$s_1 = \ldots = s_M = 1$
the shuffle numbers in \textbf{HOLGRIM} \ref{HOLGRIM_A}, 
$k_1 = \ldots = k_M = 1$ as the integers
$k_1 , \ldots , k_M \in \Z_{\geq 1}$ in 
\textbf{HOLGRIM} \ref{HOLGRIM_A}, and the real numbers 
$A_1, \ldots , A_{\n}$ as the scaling factors 
chosen in \textbf{HOLGRIM} \ref{HOLGRIM_A}.
Suppose that the integer $N$ in 
\eqref{eq:HOLGRIM_conv_thm_N} satisfies 
$N \leq M$.
Then after at most $N$ steps, the algorithm terminates. 
That is, there is some integer $m \in \{1, \ldots , N\}$
for which the \textbf{HOLGRIM} algorithm terminates after 
completing $m$ steps.
Consequently, if we define 
$Q_m := 1+mcD(d,k)$, 
there are coefficients 
$c_{1} , \ldots , c_{Q_m} \in \R$ and 
indices $e(1) , \ldots , e(Q_m) \in 
\{ 1 , \ldots , \n \}$ with 
\beq
    \label{eq:HOLGRIM_conv_thm_coeff_sum}
        \sum_{s=1}^{Q_m} \left|c_{s}\right| A_{e(s)} = C,
\eeq
and such that $u \in \Span(\cf) \subset 
\Lip(\gamma,\Sigma,W)$ defined by
\beq
    \label{eq:HOLGRIM_conv_thm_approx}
        u := \sum_{s=1}^{Q_m}
        c_{s} f_{e(s)}
        ~\left(
        \begin{array}{c}
            u^{(l)} := \sum_{s=1}^{Q_m} c_{s}
            f^{(l)}_{e(s)} \\
            \text{for every }
            l \in \{0, \ldots , k\} 
        \end{array}
        \right)
        \quad \text{satisfies} \quad
        \max_{z \in \Sigma} \left\{
        \Lambda^q_{\vph - u}(z) \right\}
        \leq \ep.
\eeq
Finally, if the coefficients 
$a_1 , \ldots , a_{\n} \in \R$ 
corresponding to $\vph$ 
(cf. (\bI) of \eqref{eq:HOLGRIM_conv_thm_vph_C_D})
are all positive (i.e. $a_1 , \ldots , a_{\n} > 0$) then 
the coefficients $c_{1} , \ldots , c_{Q_m} \in \R$ 
corresponding to $u$ 
(cf. \eqref{eq:HOLGRIM_conv_thm_approx}) are 
all non-negative (i.e. $c_{1} , \ldots , c_{Q_m} \geq 0$).
\end{theorem}

\begin{remark}
\label{rmk:HOLGRIM_conv_thm_2}
By invoking Lemma 
\ref{lemma:HOLGRIM_dist_between_interp_points} 
(cf. \eqref{eq:HOLGRIM_point_sep_general_est}) we can 
establish pointwise estimates for the sequence of 
approximations returned at each step of the \textbf{HOLGRIM}
algorithm.
That is, suppose $m \in \{1, \ldots , N\}$ such that the 
\textbf{HOLGRIM} algorithm terminates after step $m$, 
and for each $l \in \{1, \ldots ,m\}$ let $u_l \in \Span(\cf)$
denote the approximation found at step $l$ via the 
\textbf{HOLGRIM Recombination Step}.
Then Lemma \ref{lemma:HOLGRIM_dist_between_interp_points} 
yields that, for each $l \in \{1, \ldots , m\}$, 
we have, for every point $z \in \Sigma$, that 
\beq
    \label{eq:HOLGRIM_conv_thm_ul_err_bd}
        \Lambda^q_{\vph - u_l}(z)
        \leq
        \min \left\{ 2C ~,~
        2C \max_{ h \in \{0, \ldots , q\} }
        \left\{\dist_V( z , \Sigma_l )^{\gamma -h}
        \right\}
        + 
        \ep_0 e^{\dist_V( z , \Sigma_l)}
        \right\}
\eeq
where $\Sigma_l$ denotes the subset formed of the $l$ points in 
$\Sigma$ that have been selected once step $l$ is completed.
\end{remark}

\begin{remark}
\label{rmk:HOLGRIM_conv_thm_1}
Recall that the \textbf{HOLGRIM} algorithm is guaranteed to 
terminate after $M$ steps. Consequently the  
case that $N \leq M$ contains all the situations in which 
terminating after, at most, $N$ steps induces non-trivial 
conclusions.
If the integer $N$ defined in \eqref{eq:HOLGRIM_conv_thm_N}
satisfies $N < M$ then Theorem \ref{thm:HOLGRIM_conv_thm}
guarantees that the \textbf{HOLGRIM} algorithm will return 
an approximation $u \in \Span(\cf) \subset 
\Lip(\gamma,\Sigma,W)$ of $\vph$ that is a linear 
combination of \emph{less} than $\n$ of the elements 
$f_1 , \ldots , f_{\n}$, and is within $\ep$ of 
$\vph$ throughout $\Sigma$ in the pointwise sense that
$\max_{z \in \Sigma} \left\{ 
\Lambda^q_{\vph-u}(z) \right\} \leq \ep$.
Consequently, in this case the \textbf{HOLGRIM} algorithm 
is guaranteed to return an approximation $u$ of $\vph$ satisfying
\textbf{Finite Domain Approximation Conditions} 
\ref{approx_prop_1}, \ref{approx_prop_2}, and 
\ref{approx_prop_3} from Section \ref{sec:prob_formulation}.

The integer $N$ defined in \eqref{eq:HOLGRIM_conv_thm_N} 
depends both on the features $\cf$ through the constant $C$ 
and the data $\Sigma$.
The constant $C$ itself depends only on the values of the 
coefficients $a_1 , \ldots , a_{\n} \in \R$ and the values 
of the scaling factors $A_1 , \ldots , A_{\n} > 0$. 
The scaling factors depend on the values of the norms 
$||f_1||_{\Lip(\gamma,\Sigma,W)} , \ldots , 
||f_{\n}||_{\Lip(\gamma,\Sigma,W)}$ through the requirement, 
for every $i \in \{1, \ldots , \n\}$, that 
$||f_i||_{\Lip(\gamma,\Sigma,W)} \leq A_i$.
No additional constraints are imposed on the collection of 
features $\cf$; in particular, we do not assume the existence of 
a linear combination of fewer than $\n$ of the features in 
$\cf$ giving a good approximation of $\vph$ throughout 
$\Sigma$.

The integer $N$ defined in \eqref{eq:HOLGRIM_conv_thm_N} 
is, for the positive number $r > 0$ determined in 
\eqref{eq:HOLGRIM_conv_thm_r}, 
the $r$-\emph{packing} number of the subset $\Sigma$ in
$V$ as introduced by Kolmogorov \cite{Kol56}. 
Consequently the $r/2$-\emph{covering} number 
$N_{\cov}(\Sigma,V,r/2)$ of $\Sigma$ in $V$ is an upper bound 
for the integer $N$.
Here we use the same notation as used in \cite{LM22} (see, 
specifically, Subsection 6.2 in \cite{LM22}).
That is, for any $\rho > 0$, we have 
\beq
    \label{eq:HOLGRIM_conv_thm_rmk_cov_num}
        N_{\cov} (\Sigma , V,\rho) :=
        \min \left\{ s \in \Z ~:~ \exists ~ 
        z_1 , \ldots , z_s \text{ for which we have the inclusion }
        \Sigma \subset \bigcup_{j=1}^s \ovB_V(z_j,\rho)
        \right\}.
\eeq 
Recall our convention that balls denoted by $\B$ are taken to be
open whilst those denoted by $\ovB$ are taken to be closed.
Thus the maximum number of steps that the \textbf{HOLGRIM} 
algorithm can complete before terminating 
(cf. \eqref{eq:HOLGRIM_conv_thm_N}) is bounded above by the
$r/2$-covering number of $\Sigma$ in $V$. 
Moreover, this geometric property of the subset $\Sigma$ provides
an upper bound on the number of functions from the collection 
$\cf$ appearing in the approximation returned by the 
\textbf{HOLGRIM} algorithm.
Explicit estimates of this covering number provide 
\emph{worst-case} bounds for the performance of the 
\textbf{HOLGRIM} algorithm.
We illustrate a particular example of this in Remark 
\ref{rmk:HOLGRIM_conv_thm_example}.
\end{remark}

\begin{remark}
\label{rmk:HOLGRIM_conv_thm_reverse_imp}
The \emph{HOLGRIM Convergence Theorem} \ref{thm:HOLGRIM_conv_thm} 
fixes $\ep > 0$ as the accuracy we desire for an approximation, 
and then provides an upper bound on the number of functions
from the collection $\cf$ that are used by the 
\textbf{HOLGRIM} algorithm to construct an approximation 
$u$ of $\vph$ satisfying, for every $z \in \Sigma$, that 
$\Lambda^q_{\vph - u}(z) \leq \ep$.
But the \emph{HOLGRIM Convergence Theorem} 
\ref{thm:HOLGRIM_conv_thm} can also be used, for a fixed 
$n_0 \in \{2, \ldots , \n \}$, to determine how well the 
\textbf{HOLGRIM} algorithm can approximate $\vph$ 
throughout $\Sigma$ using no more than $n_0$ of the functions 
from $\cf$.

Consider a fixed $n_0 \in \{2 , \ldots , \n \}$ and a 
fixed $\ep_0 > 0$.
For $\lambda > 0$ use the notation 
\beq
    \label{eq:packing_number_notation_A}
        N_{\lambda} := \max \left\{ s \in \Z ~:~
        \exists ~ z_1 , \ldots , z_s \in \Sigma 
        \text{ with } ||z_a - z_b||_V > \lambda 
        \text{ if } a \neq b \right\}
\eeq 
for the $\lambda$-packing number of $\Sigma$ in $V$.
Define $\lambda_0 = \lambda_0 (n_0,c,d,\Sigma,\gamma) > 0$ and 
$\be = \be (n_0,c,d,\Sigma,C,\gamma,\ep_0,q) > 0$ by 
\beq
    \label{eq:lambda0_&_be0_def}
        \lambda_0 := \inf \left\{ 
        \lambda > 0 ~:~ N_{\lambda} \leq \frac{n_0 - 1}{cD}
        \right\}
        \qquad \text{and} \qquad 
        \be_0 := c d^q \left( 2 C \lambda_0^{\gamma - q} + 
        \ep_0 e^{\lambda_0} \right) > 0.
\eeq 
A first consequence of \eqref{eq:lambda0_&_be0_def} is that
$0 \leq \ep_0 < \be_0 / cd^q$.
Hence we can consider applying the \textbf{HOLGRIM} algorithm
to approximate $\vph$ on $\Sigma$ with 
$M := \min \left\{ N_{\lambda_0} , \Lambda \right\}$, 
$\be_0$ as the target accuracy in \textbf{HOLGRIM} \ref{HOLGRIM_A}, 
$\ep_0$ as the acceptable recombination error in 
\textbf{HOLGRIM} \ref{HOLGRIM_A}, 
$s_1 = \ldots = s_M = 1$ as the shuffle numbers in 
\textbf{HOLGRIM} \ref{HOLGRIM_A}, 
$k_1 = \ldots = k_M = 1$ as the integers $k_1 , \ldots , k_M$
in \textbf{HOLGRIM} \ref{HOLGRIM_A}, and 
$A_1 , \ldots , A_{\n}$ as the scaling factors in 
\textbf{HOLGRIM} \ref{HOLGRIM_A}.

The definitions of $\lambda_0$ and $\be_0$ in 
\eqref{eq:lambda0_&_be0_def} ensure that the positive 
constant $r$ and the integer $N$ arising respectively in 
\eqref{eq:HOLGRIM_conv_thm_r} and 
\eqref{eq:HOLGRIM_conv_thm_N} of the 
\emph{HOLGRIM Convergence Theorem} \ref{thm:HOLGRIM_conv_thm}
for these choices are given by $\lambda_0$ and $N_{\lambda_0}$
respectively.
Consequently the 
\emph{HOLGRIM Convergence Theorem} \ref{thm:HOLGRIM_conv_thm}
tells us that this application of the \textbf{HOLGRIM}
algorithm returns an approximation $u$ of $\vph$ that is a 
linear combination of at most $n_0$ of the functions in $\cf$, 
and that is within $\be_0$ of $\vph$ on $\Sigma$ in the 
sense that for every $z \in \Sigma$ we have 
$\Lambda^q_{\vph - u}(z) \leq \be_0$. 

In this way, given any $n_0 \in \{2 , \ldots , \n\}$ with 
$n_0 \geq 1 + cD$,
the relation given in \eqref{eq:lambda0_&_be0_def}
provides a guaranteed accuracy 
$\be_0 = \be_0(n_0,c,d,\Sigma,C,\gamma,\ep_0,q) > 0$ 
for how well the \textbf{HOLGRIM} algorithm can 
approximate $\vph$ with the additional constraint that the 
approximation is a linear combination of no greater than 
$n_0$ of the features in $\cf$. This guarantee ensures both 
that there is a linear combination of at most $n_0$ of the 
features in $\cf$ that is within $\be_0$ of $\vph$ 
throughout $\Sigma$ and that the \textbf{HOLGRIM} 
algorithm will find such a linear combination.
This idea is detailed further in Remarks 
\ref{rmk:HOLGRIM_conv_thm_Omega_n0_pointwise} and 
\ref{rmk:HOLGRIM_conv_thm_Omega_n0_lip_eta}, whilst 
an explicit example of such guarantees is provided in 
Remark \ref{rmk:HOLGRIM_conv_thm_example}.
\end{remark}

\begin{remark}
\label{rmk:HOLGRIM_conv_thm_5}
One can think of the order level $q \in \{0 , \ldots , k\}$ 
as determining the order of ``derivatives'' of $\vph$ that we
seek to approximate throughout $\Sigma$. 
For example, if we choose $q := 0$ then the 
\textbf{HOLGRIM} algorithm seeks only to approximate 
the function $\vph^{(0)} : \Sigma \to W$ in a 
pointwise sense throughout $\Sigma$.
And, in this case, Theorem \ref{thm:HOLGRIM_conv_thm}
establishes conditions under which the 
\textbf{HOLGRIM} algorithm is guaranteed to return 
an approximation 
$u =\left( u^{(0)} , \ldots , u^{(k)} \right)
\in \Span(\cf) \subset \Lip(\gamma,\Sigma,W)$
for which the function $u^{(0)} : \Sigma \to W$
is close to $\vph^{(0)}$ throughout $\Sigma$ in the 
sense that 
$\left|\left| \vph^{(0)} - u^{(0)}
\right|\right|_{C^0(\Sigma;W)} \leq \ep$
(cf. \eqref{eq:HOLGRIM_conv_thm_approx}).
On the other hand, if we choose $q := k$ then the 
\textbf{HOLGRIM} algorithm seeks to approximate 
the function $\vph^{(l)} : \Sigma \to \cl(V^{\otimes l};W)$
in a pointwise sense throughout $\Sigma$
for every $l \in \{0, \ldots , k\}$.
And, in this case, Theorem \ref{thm:HOLGRIM_conv_thm}
establishes conditions under which the 
\textbf{HOLGRIM} algorithm is guaranteed to return 
an approximation 
$u =\left( u^{(0)} , \ldots , u^{(k)} \right)
\in \Span(\cf) \subset \Lip(\gamma,\Sigma,W)$
for which, for every $l \in \{0, \ldots , k\}$, 
the function $u^{(l)} : \Sigma \to \cl(V^{\otimes};W)$
is close to $\vph^{(l)}$ throughout $\Sigma$ in the 
sense that for every $z \in \Sigma$ we have
$\left|\left| \vph^{(l)}(z) - u^{(l)}(z) 
\right|\right|_{\cl(V^{\otimes l};W)} \leq \ep$
(cf. \eqref{eq:HOLGRIM_conv_thm_approx}).
\end{remark}

\begin{remark}
\label{rmk:HOLGRIM_conv_thm_6}
The order level $q \in \{0, \ldots , k\}$ in the \textbf{HOLGRIM} 
algorithm determines 
the order to which we seek to approximate the function 
$\vph = \left( \vph^{(0)} , \ldots , 
\vph^{(k)} \right) \in \Lip(\gamma,\Sigma,W)$.
A particular consequence of \eqref{eq:HOLGRIM_conv_thm_r}
is that a larger value of the order level $q \in \{0, \ldots , k\}$
leads to a larger upper bound for the maximum number of 
steps the \textbf{HOLGRIM} algorithm can complete before
terminating.
That is, for fixed values of the parameters $C$, $\gamma$, 
$\ep$, and $\ep_0$ the function 
$q \mapsto r(C,\gamma,\ep,\ep_0,q)$, 
for the $r$ defined in \eqref{eq:HOLGRIM_conv_thm_r}, 
is decreasing on $\{0 , \ldots , k\}$.
Hence, again for fixed values of the parameters $C$, $\gamma$, 
$\ep$, and $\ep_0$, the function
$q \mapsto N(\Sigma,C,\gamma,\ep,\ep_0,q)$, 
for the integer $N$ defined in 
\eqref{eq:HOLGRIM_conv_thm_N}, 
is increasing on $\{0, \ldots , k\}$. 
Consequently, seeking only to approximate $\vph$ to 
order $q:=0$ results in the smallest upper bound 
for the maximum number of steps before the 
\textbf{HOLGRIM} algorithm terminates, 
whilst seeking to approximate $\vph$ to order
$q:=k$ results in the largest upper bound.
Larger values of $q$ result in the algorithm
finding an approximation of
$\vph$ in a stronger sense throughout $\Sigma$, at the 
cost of a worse upper bound for the maximum 
number of steps the algorithm runs for before
terminating.
\end{remark}

\begin{remark}
\label{rmk:HOLGRIM_conv_thm_gamma_depend}
As the regularity parameter $\gamma > 0$ increases, the integer 
$N$ in \eqref{eq:HOLGRIM_conv_thm_N} giving an upper 
bound on the maximum number of steps completed by the 
\textbf{HOLGRIM} algorithm before termination decreases.
To be more precise we adopt the same notation as that of the 
\emph{HOLGRIM Convergence Theorem} \ref{thm:HOLGRIM_conv_thm} 
and fix the values of the parameters $C$, $\ep$, $\ep_0$, and 
fix a choice of $q \in \Z_{\geq 0}$. 
Then consider varying the regularity parameter $\gamma$ in the 
range $(q,\infty)$; requiring $\gamma > q$ ensures that we can 
apply the \textbf{HOLGRIM} algorithm as in the 
\emph{HOLGRIM Convergence Theorem} \ref{thm:HOLGRIM_conv_thm}. 
The mapping $\gamma \mapsto r(C,\gamma,\ep,\ep_0,q)$ 
for the constant $r$ defined in \eqref{eq:HOLGRIM_conv_thm_r} 
determines an increasing function on $(q,\infty)$ since the 
assumption that $\ep \in (0,C)$ means that $r \leq 1$.
Consequently the mapping 
$\gamma \mapsto N(\Sigma,C,\gamma,\ep,\ep_0,q)$ for the 
integer $N$ defined in \eqref{eq:HOLGRIM_conv_thm_N} 
determines a decreasing function on $(q,\infty)$.
Hence the data concentration required to guarantee that the 
\textbf{HOLGRIM} algorithm returns an approximation $u$ of 
$\vph$ satisfying that
$\max_{z \in \Sigma} \left\{ \Lambda^q_{\vph-u}(z) \right\} \leq \ep$
and that $u$ is a linear combination of \emph{strictly less} 
than $\n$ of the functions in $\cf$ is decreasing as the regularity 
parameter $\gamma$ increases.
That is, a larger value of $\gamma > 0$ results in a smaller 
upper bound $N$ on the maximum number of steps completed by the 
\textbf{HOLGRIM} algorithm before it returns an approximation of 
$\vph$ satisfying 
\textbf{Finite Domain Approximation Conditions} 
\ref{approx_prop_1}, \ref{approx_prop_2}, and 
\ref{approx_prop_3} from Section \ref{sec:prob_formulation}.
\end{remark}

\begin{remark}
\label{rmk:HOLGRIM_conv_thm_3}
The approximation $u$ returned by the \textbf{HOLGRIM} algorithm 
satisfying the accuracy properties guaranteed by Theorem 
\ref{thm:HOLGRIM_conv_thm} additionally satisfies the following
\emph{robustness} property.
To illustrate this robustness, assume the notation of Theorem 
\ref{thm:HOLGRIM_conv_thm} and fix a choice of a finite 
$\xi \in (\ep,\infty)$.
Suppose that $K \geq 1$ and that, for every 
$i \in \{1, \ldots , \n\}$, the function 
$f_i \in \Lip(\gamma,\Sigma,W)$ admits an extension 
$F_i \in \Lip(\gamma,V,W)$ satisfying the norm estimate
$||F_i||_{\Lip(\gamma,V,W)} \leq K 
|| f_i ||_{\Lip(\gamma,\Sigma,W)}$.

The existence of such an extension is guaranteed by the 
Stein-Whitney extension theorem (Theorem 4 in Chapter VI of 
\cite{Ste70}. Whilst that result is stated for the case that 
$V=\R^d$ and $W = \R$, the proof carries over to our 
setting verbatim.
In this result, a method of constructing an extension of
an element $\phi \in \Lip(\gamma,\Sigma,W)$
to an element in $\Lip(\gamma,V,W)$ is presented.
Moreover, the resulting operator
$\bE : \Lip(\gamma,\Sigma,W) \to \Lip(\gamma,V,W)$ 
mapping an element $\phi \in \Lip(\gamma,\Sigma,W)$
to its extension to an element in $\Lip(\gamma,V,W)$ is 
established to be a bounded linear operator.

This extension operator is \emph{not} unique; the construction 
of an extension of $\phi \in \Lip(\gamma,\Sigma,W)$ to 
an element in $\Lip(\gamma,V,W)$ in the proof of the 
Stein-Whitney extension theorem in \cite{Ste70} involves 
making several particular choices of parameters, and varying 
each choice results in a different extension.
However, after fixing a particular choice of each parameter, 
a verbatim repetition of the 
arguments of Stein in Chapter VI of \cite{Ste70} establishes 
the existence 
of a constant $K_{\bE} = K_{\bE}(\gamma,d) \geq 1$ such that for any 
$\phi \in \Lip(\gamma,\Sigma,W)$ we have the norm estimate that
$||\bE[\phi]||_{\Lip(\gamma,V,W)} \leq K_{\bE} 
|| \phi ||_{\Lip(\gamma,\Sigma,W)}$.
Consequently, we could take $K := K_{\bE}$ above.

Recalling that $\vph = \sum_{i=1}^{\n} a_i f_i$, it 
follows that $\ti{\vph} := \sum_{i=1}^{\n} a_i F_i$ determines
an extension of $\vph$ to an element 
$\ti{\vph} \in \Lip(\gamma,V,W)$.
It follows that 
$||\ti{\vph}||_{\Lip(\gamma,V,W)} \leq 
K || \vph ||_{\Lip(\gamma,\Sigma,W)} \leq CK$.
Similarly, since $u = \sum_{s=1}^{Q_m} c_s f_{e(s)}$, it follows
that $\ti{u} := \sum_{s=1}^{Q_m} c_s F_{e(s)}$ determines an 
extension of $u$ to an element $\ti{u} \in \Lip(\gamma,V,W)$.
It follows, via \eqref{eq:HOLGRIM_conv_thm_coeff_sum}, that
$||\ti{u}||_{\Lip(\gamma,V,W)} \leq 
K || u ||_{\Lip(\gamma,\Sigma,W)} \leq CK$.
Moreover, since $\ti{\vph}$ and $\ti{u}$ coincide with 
$\vph$ and $u$ respectively throughout $\Sigma$, 
the second part of 
\eqref{eq:HOLGRIM_conv_thm_approx} yields that
\beq
    \label{eq:HOLGRIM_conv_thm_rmk3_extensions_close_Sigma}
        \max_{z \in \Sigma} \left\{
        \Lambda^q_{\ti{\vph} - \ti{u} }(z) \right\}
        =
        \max_{z \in \Sigma} \left\{ 
        \Lambda^q_{\vph - u}(z) \right\} 
        \leq 
        \ep.
\eeq
Fix a choice of $l \in \{0, \ldots , q\}$ and
define a constant $r = r(C,K,q,\ep,\xi,l) > 0$ by
\beq
    \label{eq:HOLGRIM_conv_thm_rmk3_r}
        r := \twopartdef
        {\sup \left\{ \lambda > 0 : 
        2CK \lambda^{\gamma - l} + \ep e^{\lambda} 
        \leq \xi \right\} }
        {q = k}
        { \sup \left\{ \lambda > 0 : 
        2eCK \lambda^{q+1-l} + \ep e^{\lambda} \leq \xi 
        \right\}}
        {q < k.}
\eeq
Set $\Omega := \Sigma_{r} := \left\{ v \in V : 
\exists ~u \in \Sigma \text{ with } ||u-v||_V \leq r \right\}$
to be the $r$-fattening of $\Sigma$ in $V$.
We claim that the extensions $\ti{\vph}$ and $\ti{u}$ of 
$\vph$ and $u$ respectively remain close throughout 
$\Omega$ in the sense that
\beq
    \label{eq:HOLGRIM_conv_thm_rmk3_extensions_close_Omega}
        \sup_{z \in \Omega} \left\{
        \Lambda^l_{\ti{\vph} - \ti{u}} (z) \right\}
        \leq \xi.
\eeq
We first verify 
\eqref{eq:HOLGRIM_conv_thm_rmk3_extensions_close_Omega}
in the case that $q = k$. 
In this case, an examination of $r$ as defined in 
\eqref{eq:HOLGRIM_conv_thm_rmk3_r} reveals that it coincides 
with the specification of the constant $\de_0$ in 
\eqref{eq:pointwise_lip_sand_thm_de0} of the 
\emph{Pointwise Lipschitz Sandwich Theorem} 
\ref{thm:pointwise_lip_sand_thm} for the choices 
of $\ep$, $\ep_0$, $K_0$, $\gamma$, and $l$ there as
$\xi$, $\ep$, $CK$, $\gamma$, and $l$ here respectively.
Therefore \eqref{eq:HOLGRIM_conv_thm_rmk3_extensions_close_Sigma}
provides the required hypothesis to 
appeal to  the 
\emph{Pointwise Lipschitz Sandwich Theorem} 
\ref{thm:pointwise_lip_sand_thm}
with $B := \Sigma$, $\psi := \ti{\vph}$, and $\phi := \ti{u}$.
By applying this result, we conclude, for every $z \in \Omega$, that 
$\Lambda^l_{\ti{\vph} - \ti{u}}(z) \leq \xi$ as claimed in 
\eqref{eq:HOLGRIM_conv_thm_rmk3_extensions_close_Omega}.

We now consider the case that $q < k$ so that $k \geq 1$ and
$q \in \{0, \ldots , k-1\}$.
Recall the notation that for any 
$\phi = \left( \phi^{(0)} , \ldots , \phi^{(k)} \right) 
\in \Lip(\gamma,\Sigma,W)$ and any $n \in \{0, \ldots , k\}$ that
$\phi_{[n]} := \left( \phi^{(0)} , \ldots , \phi^{(n)} \right)$.
We use the analogous notation for elements in 
$\Lip(\gamma,V,W)$.

Since $q+1 \leq k < \gamma$ we can 
invoke Lemma 6.1 in \cite{LM24} to deduce that 
$\ti{\vph}_{[q]} , \ti{u}_{[q]} \in \Lip(q+1,\Sigma,W)$
with
$|| \ti{\vph}_{[q]} ||_{\Lip(q+1,\Sigma,W)} \leq 
e || \ti{\vph} ||_{\Lip(\gamma,\Sigma,W)} \leq eCK$ and 
$|| \ti{u}_{[q]} ||_{\Lip(q+1,\Sigma,W)} \leq 
e ||\ti{u} ||_{\Lip(\gamma,\Sigma,W)} \leq eCK$.
An examination of $r$ as defined in 
\eqref{eq:HOLGRIM_conv_thm_rmk3_r} reveals that it coincides 
with the specification of the constant $\de_0$ in 
\eqref{eq:pointwise_lip_sand_thm_de0} of the 
\emph{Pointwise Lipschitz Sandwich Theorem} 
\ref{thm:pointwise_lip_sand_thm} for the choices 
of $\ep$, $\ep_0$, $K_0$, $\gamma$, and $l$ there as
$\xi$, $\ep$, $eCK$, $q+1$, and $l$ here respectively.
Therefore \eqref{eq:HOLGRIM_conv_thm_rmk3_extensions_close_Sigma}
provides the required hypothesis to appeal
to the \emph{Pointwise Lipschitz Sandwich Theorem} 
\ref{thm:pointwise_lip_sand_thm}
with $B := \Sigma$, $\psi := \ti{\vph}_{[q]}$, 
and $\phi := \ti{u}_{[q]}$. 
By doing so, we conclude, for every $z \in \Omega$, that 
$\Lambda^l_{\ti{\vph} - \ti{u} }(z) \leq \xi$ 
as claimed in 
\eqref{eq:HOLGRIM_conv_thm_rmk3_extensions_close_Omega}.

Thus the extension $\ti{u}$ is a linear combination of 
$Q_m$ of the functions $F_1 , \ldots , F_{\n}$ that is 
within $\xi$ of the extension $\ti{\vph}$ throughout 
$\Omega$ in the pointwise sense that 
$\sup_{z \in \Omega} \left\{ 
\Lambda^l_{\ti{\vph}-\ti{u}}(z) \right\} \leq \xi$.
Loosely speaking, the approximation on $\Sigma$ 
found by the \textbf{HOLGRIM} algorithm remains a good 
approximation on points in $V$ that are not too far away from
$\Sigma$.
To be more precise, provided we are given, for each 
$i \in \{1, \ldots ,\n\}$, an extension $F_i \in \Lip(\gamma,V,W)$
of $f_i \in \Lip(\gamma,\Sigma,W)$, the weights
$c_1 , \ldots , c_{Q_m} \in \R$ and the indices 
$e(1) , \ldots , e(Q_m) \in \{1, \ldots , \n\}$ returned by 
the \textbf{HOLGRIM} algorithm, determining a good approximation 
of $\vph$ throughout $\Sigma$ via the linear combination 
$u = \sum_{s=1}^{Q_m} c_s f_{e(s)}$ of the functions 
$f_1 , \ldots , f_{\n}$, additionally determine a good 
approximation of the extension $\ti{\vph}$ of $\vph$ 
throughout $\Omega$ via the linear combination 
$\ti{u} = \sum_{s=1}^{Q_m} c_s F_{e(s)}$ of the extensions
$F_1 , \ldots , F_{\n}$ of the functions $f_{1} , \ldots , f_{\n}$.
\end{remark}

\begin{remark}
\label{rmk:HOLGRIM_conv_thm_Omega_ptwise}
Following the strategy outlined in Section 
\ref{sec:compact_domains}, we may combine the 
theoretical guarantees for the \textbf{HOLGRIM}
algorithm provided by Theorem \ref{thm:HOLGRIM_conv_thm} with 
the \emph{Pointwise Lipschitz Sandwich Theorem} 
\ref{thm:pointwise_lip_sand_thm} (cf. Theorem 3.11 in \cite{LM24})
to extend the \textbf{HOLGRIM} algorithm to the setting 
in which the $\Lip(\gamma)$ functions $f_1 , \ldots , f_{\n}$
are defined throughout an arbitrary non-empty compact subset 
$\Omega \subset V$ rather than a finite subset.
To make this precise, let $\n \in \Z_{\geq 1}$, 
$\gamma > 0$ with $k \in \Z_{\geq 0}$ such that 
$\gamma \in (k,k+1]$, and
suppose that $\Omega \subset V$ is a fixed 
non-empty compact subset with
$f_1 , \ldots , f_{\n} \in \Lip (\gamma,\Omega,W)$.
As in Theorem \ref{thm:HOLGRIM_conv_thm} we consider 
$\vph \in \Lip(\gamma,\Omega,W)$ defined by 
$\vph := \sum_{i=1}^{\n} a_i f_i$ for non-zero coefficients
$a_1 , \ldots , a_{\n} \in \R$.
Let
$C := \sum_{i=1}^{\n} |a_i| ||f_i||_{\Lip(\gamma,\Omega,W)} > 0$,
and take $D = D(d,k)$ to be as defined in 
\eqref{eq:HOLGRIM_conv_thm_D}.
We now illustrate how the \textbf{HOLGRIM} algorithm can 
be used to obtain an approximation of $\vph$ throughout 
the compact subset $\Omega \subset V$.

Fix choices of $\ep \in (0,C)$ and $0 \leq \ep_0 < \ep/2 c d^k$
and $q \in \{0, \ldots , k\}$.
Define a positive constant 
$\de = \de(C,\gamma,\ep,c,d,q) > 0$ by 
\beq
    \label{eq:HOLGRIM_conv_thm_rmk_r_Omega}
        \de := \sup \left\{ \lambda > 0 ~:~ 
        2C \lambda^{\gamma -q} + 
        \frac{\ep}{2}e^{\lambda} \leq \ep 
        \right\}.
\eeq 
Next we define an integer 
$\Lambda = \Lambda (\Omega,C,\gamma,\ep,c,d,q) \in \Z_{\geq 1}$
by 
\beq
    \label{eq:HOLGRIM_conv_thm_rmk_Lambda}
        \Lambda := \min \left\{ s \in \Z_{\geq 1} 
        ~:~ \exists ~ z_1 , \ldots , z_s \in \Omega 
        \text{ such that }
        \Omega \subset \bigcup_{j=1}^s \ovB_V 
        ( z_j , \de ) \right\}.
\eeq 
The compactness of the subset $\Omega \subset V$ ensures 
that the integer $\Lambda$ defined in 
\eqref{eq:HOLGRIM_conv_thm_rmk_Lambda} is finite.

Choose a finite subset $\Sigma = \left\{ z_1 , \ldots , z_{\Lambda} 
\right\} \subset \Omega$ of cardinality $\Lambda$ for which 
$\Omega \subset \cup_{s=1}^{\Lambda} \ovB_{V}(z_s,\de)$.
With the finite subset $\Sigma$ fixed, define a positive 
constant $r = r(C,\gamma,\ep,\ep_0,c,d) > 0$ by 
(cf. \eqref{eq:HOLGRIM_conv_thm_r}) 
\beq
    \label{eq:HOLGRIM_conv_thm_rmk_compact_r}
        r := \sup \left\{ \lambda > 0 ~:~ 
        2C \lambda^{\gamma-k} + \ep_0 e^{\lambda} 
        \leq \frac{\ep}{2 c d^k} \right\},
\eeq
and define an integer 
$N = N(\Sigma,C,\gamma,\ep,\ep_0,c,d) \in \Z_{\geq 1}$ by 
(cf. \eqref{eq:HOLGRIM_conv_thm_N}) 
\beq
    \label{eq:HOLGRIM_conv_thm_rmk_compact_N}
        N := \max \left\{ s \in \Z_{\geq 1} ~:~ 
        \exists ~ z_1 , \ldots , z_s \in \Sigma 
        \text{ for which } 
        || z_a - z_b ||_V > r \text{ if } a \neq b \right\}.
\eeq
Consider applying the \textbf{HOLGRIM}
algorithm to approximate $\vph$ on $\Sigma$, with 
$M := \min \left\{ \frac{\n - 1}{cD} , \Lambda \right\}$,
$\ep/2$ as the target accuracy in \textbf{HOLGRIM} \ref{HOLGRIM_A}, 
$\ep_0$ as the acceptable recombination error in 
\textbf{HOLGRIM} \ref{HOLGRIM_A}, $k$ as  
the order level in \textbf{HOLGRIM} \ref{HOLGRIM_A}, 
$s_1 = \ldots = s_M = 1$
the shuffle numbers in \textbf{HOLGRIM} \ref{HOLGRIM_A}, 
$k_1 = \ldots = k_M = 1$ as the integers
$k_1 , \ldots , k_M \in \Z_{\geq 1}$ in 
\textbf{HOLGRIM} \ref{HOLGRIM_A}, and 
$A_1 := ||f_1||_{\Lip(\gamma,\Omega,W)} , \ldots , 
A_{\n} := ||f_{\n}||_{\Lip(\gamma,\Omega,W)}$ as the 
scaling factors chosen in \textbf{HOLGRIM} \ref{HOLGRIM_A}.

If the integer $N$ defined in 
\eqref{eq:HOLGRIM_conv_thm_rmk_compact_N} satisfies that 
$N \leq M$, then
Theorem \ref{thm:HOLGRIM_conv_thm} guarantees that 
there is some integer $m \in \{1, \ldots , M\}$
for which the \textbf{HOLGRIM} algorithm terminates after 
completing $m$ steps. Thus, if $Q_m := 1+mcD(d,k)$, 
there are coefficients 
$c_{1} , \ldots , c_{Q_m} \in \R$ and 
indices $e(1) , \ldots , e(Q_m) \in 
\{ 1 , \ldots , \n \}$ with 
(cf. \eqref{eq:HOLGRIM_conv_thm_coeff_sum})
\beq
    \label{eq:HOLGRIM_conv_thm_rmk_compact_coeff_sum}
        \sum_{s=1}^{Q_m} \left|c_{s}\right|
        \left|\left| f_{e(s)}
        \right|\right|_{\Lip(\gamma,\Omega,W)}
        = C,
\eeq
and such that $u \in \Span(\cf) \subset 
\Lip(\gamma,\Omega,W)$ defined by
(cf. \eqref{eq:HOLGRIM_conv_thm_approx})
\beq
    \label{eq:HOLGRIM_conv_thm_rmk_compact_approx}
        u := \sum_{s=1}^{Q_m}
        c_{s} f_{e(s)}
        ~\left(
        \begin{array}{c}
            u^{(l)} := \sum_{s=1}^{Q_m} c_{s}
            f^{(l)}_{e(s)} \\
            \text{for every }
            l \in \{0, \ldots , k\} 
        \end{array}
        \right)
        \quad \text{satisfies} \quad
        \max_{z \in \Sigma} \left\{
        \Lambda^k_{\vph - u}(z) \right\}
        \leq \frac{\ep}{2}.
\eeq
Observe that the condition 
\eqref{eq:HOLGRIM_conv_thm_rmk_compact_coeff_sum} means that
$||u||_{\Lip(\gamma,\Omega,W)} \leq C$.
Recall that the finite subset $\Sigma \subset \Omega$ was 
chosen to be a $\de$-cover of $\Omega$ in the sense that 
$\Omega \subset \cup_{z \in \Sigma} \ovB_V(z,\de)$.
Therefore the definition of $\de$ in 
\eqref{eq:HOLGRIM_conv_thm_rmk_r_Omega} and
the estimates in 
\eqref{eq:HOLGRIM_conv_thm_rmk_compact_approx} provide
the hypotheses required to appeal to the 
\emph{Pointwise Lipschitz Sandwich Theorem} 
\ref{thm:pointwise_lip_sand_thm} with the 
$K_0$, $\gamma$, $\ep$, $\ep_0$, $l$, $\Sigma$, $B$, 
$\psi$, and $\phi$ of that result as 
$C$, $\gamma$, $\ep$, $\ep/2$, $q$, $\Omega$, $\Sigma$, 
$\vph$, and $u$ here respectively.
Indeed the definition of $\de$ here in 
\eqref{eq:HOLGRIM_conv_thm_rmk_r_Omega} coincides with 
the definition of $\de_0$ in \eqref{eq:pointwise_lip_sand_thm_de0}
for the choices specified above.
Thus we may invoke the implication 
\eqref{eq:pointwise_lip_sand_thm_imp} to conclude that 
\beq
    \label{eq:HOLGRIM_conv_thm_compact_approx_Omega}
        \sup_{z \in \Omega} \left\{
        \Lambda^q_{\vph-u}(z) \right\}
        \leq \ep.
\eeq
Consequently the \textbf{HOLGRIM} algorithm has returned an 
approximation $u \in \Span(\cf) \subset \Lip(\gamma,\Omega,W)$
of $\vph$ that has the sparsity properties as outlined in 
Remark \ref{rmk:HOLGRIM_conv_thm_1} and is
within $\ep$ of $\vph$ throughout $\Omega$
in the pointwise sense that, for every point $z \in \Omega$, 
we have $\Lambda^q_{\vph - u}(z) \leq \ep$.
That is, in this case the \textbf{HOLGRIM} algorithm 
is guaranteed to return an approximation $u$ of $\vph$ satisfying
\textbf{Compact Domain Approximation Conditions} 
\ref{compact_approx_prop_1}, \ref{compact_approx_prop_2}, and 
\ref{compact_approx_prop_3} from Section 
\ref{sec:compact_domains}.
\end{remark}

\begin{remark}
\label{rmk:HOLGRIM_conv_thm_Omega_n0_pointwise}
Following the approach outlined in Remark 
\ref{rmk:HOLGRIM_conv_thm_reverse_imp}, we illustrate the 
theoretical guarantees available for the \textbf{HOLGRIM} 
algorithm when used as in Remark
\ref{rmk:HOLGRIM_conv_thm_Omega_ptwise}
but with an additional restriction on the maximum 
number of functions from $\cf$ that may be used to construct 
the approximation $u$. 

To make this precise, let $\n \in \Z_{\geq 1}$, 
$\gamma > 0$ with $k \in \Z_{\geq 0}$ such that 
$\gamma \in (k,k+1]$, and
suppose that $\Omega \subset V$ is a fixed 
non-empty compact subset with
$f_1 , \ldots , f_{\n} \in \Lip (\gamma,\Omega,W)$.
As in Theorem \ref{thm:HOLGRIM_conv_thm} we consider 
$\vph \in \Lip(\gamma,\Omega,W)$ defined by 
$\vph := \sum_{i=1}^{\n} a_i f_i$ for non-zero coefficients
$a_1 , \ldots , a_{\n} \in \R$.
Let
$C := \sum_{i=1}^{\n} |a_i| ||f_i||_{\Lip(\gamma,\Omega,W)} > 0$,
and take $D = D(d,k)$ to be as defined in 
\eqref{eq:HOLGRIM_conv_thm_D}.
Fix a choice of $q \in \{0 , \ldots , k\}$ and, assuming that
$\n > 1 + c D$, fix 
a choice of integer $n_0 \in \{1, \ldots , \n\}$ such that
$n_0 \geq 1 + cD$.
We consider the guarantees available for the pointwise quantity 
$\sup_{z \in \Omega} \left\{ \Lambda^q_{\vph-u}(z) \right\}$ 
for the approximation returned by the \textbf{HOLGRIM} algorithm
with the restriction that $u$ is a linear combination of at most
$n_0$ of the functions in $\cf$.

For this purpose we fix $\th \geq 0$ and define 
$\lambda_0 = \lambda_0 (n_0,c,d,\gamma) > 0$ and 
$\be_0 = \be_0 (n_0,C,c,d,\gamma,\th) > 0$ by
\beq
    \label{eq:HOLGRIM_conv_thm_Omega_n0_rmk_lambda0_be0_ptwise}
        \lambda_0 := \inf \left\{ t > 0 ~:~ 
        N_{\pack}(\Omega,V,t) \leq \frac{n_0-1}{cD} \right\}
        > 0
        \qquad \text{and} \qquad
        \be_0 := 2cd^k \left( 2C \lambda_0^{\gamma - k} 
        + \th e^{\lambda_0} \right) > 0.
\eeq 
With the value of $\be_0 > 0$ fixed, define a positive
constant $\de = \de(n_0,C,c,d,\gamma,\th,q) > 0$ and an integer 
$\Lambda = \Lambda (\Omega,n_0,C,c,d,\gamma,\th,q) 
\in \Z_{\geq 1}$ by
\beq
    \label{eq:HOLGRIM_conv_thm_Omega_n0_rmk_de_Lambda_n0_ptwise}
        \de := \sup \left\{ t > 0 ~:~ 2Ct^{\gamma - q} + 
        \frac{\be_0}{2}e^t \leq \be_0 \right\} > 0
        \quad \text{and} \quad 
        \Lambda := N_{\cov}(\Omega,V,\de).
\eeq 
Since $\Omega \subset V$ is compact, we have that 
$\Lambda$ defined in 
\eqref{eq:HOLGRIM_conv_thm_Omega_n0_rmk_de_Lambda_n0_ptwise} 
is finite. Choose a subset 
$\Sigma = \{ p_1 , \ldots , p_{\Lambda} \} 
\subset \Omega \subset V$
for which $\Omega \subset \cup_{s=1}^{\Lambda} \ovB_V(p_s,\de)$.
Moreover, since 
$N_{\pack}(\Sigma,V,\lambda_0) \leq N_{\pack}(\Omega,V,\lambda_0)$,
we have via 
\eqref{eq:HOLGRIM_conv_thm_Omega_n0_rmk_lambda0_be0_ptwise} that
$N_{\pack}(\Sigma,V,\lambda_0) \leq \frac{n_0-1}{cD}$.
Furthermore, it additionally follows from 
\eqref{eq:HOLGRIM_conv_thm_Omega_n0_rmk_lambda0_be0_ptwise}
that $0 \leq \th < \be_0 / 2 cd^k$.

Therefore we may apply the \textbf{HOLGRIM} algorithm to 
approximate $\vph$ on $\Sigma$ with 
$M := \min \left\{ \frac{n_0 - 1}{cD} , \Lambda \right\}$,
$\be_0/2$ as the target accuracy in \textbf{HOLGRIM} \ref{HOLGRIM_A}, 
$\th$ as the acceptable recombination error in 
\textbf{HOLGRIM} \ref{HOLGRIM_A}, $k$ as  
the order level in \textbf{HOLGRIM} \ref{HOLGRIM_A}, 
$s_1 = \ldots = s_M = 1$
the shuffle numbers in \textbf{HOLGRIM} \ref{HOLGRIM_A}, 
$k_1 = \ldots = k_M = 1$ as the integers
$k_1 , \ldots , k_M \in \Z_{\geq 1}$ in 
\textbf{HOLGRIM} \ref{HOLGRIM_A}, and 
$A_1 := ||f_1||_{\Lip(\gamma,\Omega,W)} , \ldots , 
A_{\n} := ||f_{\n}||_{\Lip(\gamma,\Omega,W)}$ as the 
scaling factors chosen in \textbf{HOLGRIM} \ref{HOLGRIM_A}.

Observe that $\lambda_0$ defined in 
\eqref{eq:HOLGRIM_conv_thm_example_rmk_lambda0_be0_ptwise}
coincides with the positive constant $r > 0$ arising in 
\eqref{eq:HOLGRIM_conv_thm_r} in the 
\emph{HOLGRIM Convergence Theorem} \ref{thm:HOLGRIM_conv_thm}
for the constants $C$, $\gamma$, $\ep$, $\ep_0$, and $q$ there
as $C$, $\gamma$, $\be_0/2$, $\th$ and $k$ here. Thus, since
$N_{\pack}(\Sigma,V,\lambda_0) \leq \frac{n_0-1}{cD}$,
the \emph{HOLGRIM Convergence Theorem} \ref{thm:HOLGRIM_conv_thm} 
tells us that there is an integer $m \in \{1, \ldots , M\}$
for which the \textbf{HOLGRIM} algorithm terminates after 
completing $m$ steps. Thus, if $Q_m := 1+mcD(d,k)$, 
there are coefficients 
$c_{1} , \ldots , c_{Q_m} \in \R$ and 
indices $e(1) , \ldots , e(Q_m) \in 
\{ 1 , \ldots , \n \}$ with 
(cf. \eqref{eq:HOLGRIM_conv_thm_coeff_sum})
\beq
    \label{eq:HOLGRIM_conv_thm_Omega_n0_rmk_coeff_sum_n0_ptwise}
        \sum_{s=1}^{Q_m} \left|c_{s}\right|
        \left|\left| f_{e(s)}
        \right|\right|_{\Lip(\gamma,\Omega,W)}
        = C,
\eeq
and such that $u \in \Span(\cf) \subset 
\Lip(\gamma,\Omega,W)$ defined by
(cf. \eqref{eq:HOLGRIM_conv_thm_approx})
\beq
    \label{eq:HOLGRIM_conv_thm_Omega_n0_rmk_u_n0_ptwise}
        u := \sum_{s=1}^{Q_m}
        c_{s} f_{e(s)}
        ~\left(
        \begin{array}{c}
            u^{(l)} := \sum_{s=1}^{Q_m} c_{s}
            f^{(l)}_{e(s)} \\
            \text{for every }
            l \in \{0, \ldots , k\} 
        \end{array}
        \right)
        \quad \text{satisfies} \quad
        \max_{z \in \Sigma} \left\{
        \Lambda^k_{\vph - u}(z) \right\}
        \leq \frac{\be_0}{2}.
\eeq
The condition 
\eqref{eq:HOLGRIM_conv_thm_Omega_n0_rmk_coeff_sum_n0_ptwise} 
means that $||u||_{\Lip(\gamma,\Omega,W)} \leq C$.
Recall that the finite subset $\Sigma \subset \Omega$ was 
chosen to be a $\de$-cover of $\Omega$ in the sense that 
$\Omega \subset \cup_{z \in \Sigma} \ovB_V(z,\de)$.
Therefore the definition of $\de$ in 
\eqref{eq:HOLGRIM_conv_thm_example_rmk_de_Lambda_n0_ptwise} and
the estimates in 
\eqref{eq:HOLGRIM_conv_thm_example_rmk_u_n0_ptwise} provide
the hypotheses required to appeal to the 
\emph{Pointwise Lipschitz Sandwich Theorem} 
\ref{thm:pointwise_lip_sand_thm} with the 
$K_0$, $\gamma$, $\ep$, $\ep_0$, $l$, $\Sigma$, $B$, 
$\psi$, and $\phi$ of that result as 
$C$, $\gamma$, $\be_0$, $\be_0/2$, $q$, $\Omega$, $\Sigma$, 
$\vph$, and $u$ here respectively.
Indeed the definition of $\de$ here in 
\eqref{eq:HOLGRIM_conv_thm_example_rmk_de_Lambda_n0_ptwise} 
coincides with 
the definition of $\de_0$ in \eqref{eq:pointwise_lip_sand_thm_de0}
for the choices specified above.
Thus we may invoke the implication 
\eqref{eq:pointwise_lip_sand_thm_imp} to conclude that 
\beq
    \label{eq:HOLGRIM_conv_thm_Omega_n0_rmk_u_[0,1]_n0_ptwise}
        \sup_{z \in \Omega} \left\{
        \Lambda^q_{\vph-u}(z) \right\}
        \leq \be_0.
\eeq
Moreover, we have that $Q_m \leq 1 + cMD \leq n_0$ so that
$u$ is guaranteed to be a linear combination of no greater than 
$n_0$ of the functions in $\cf$.
Thus given an integer $n_0 \in \{1, \ldots , \n\}$ 
satisfying $n_0 \geq 1 + c D(d,k)$, we are 
guaranteed that the \textbf{HOLGRIM} algorithm can find 
an approximation $u \in \Span(\cf)$ of $\vph$ that is a linear
combination of no more than $n_0$ of the functions in $\cf$, 
and is within $\be_0$, for $\be_0$ defined in 
\eqref{eq:HOLGRIM_conv_thm_Omega_n0_rmk_lambda0_be0_ptwise},
of $\vph$ throughout $\Omega$ in the pointwise sense detailed 
in \eqref{eq:HOLGRIM_conv_thm_Omega_n0_rmk_u_[0,1]_n0_ptwise}.
\end{remark}

\begin{remark}
\label{rmk:HOLGRIM_conv_thm_Omega_lip_eta}
Following the strategy outlined in Section 
\ref{sec:compact_domains}, we may combine the 
theoretical guarantees for the \textbf{HOLGRIM}
algorithm provided by Theorem \ref{thm:HOLGRIM_conv_thm} with 
the \emph{Lipschitz Sandwich Theorem} 
\ref{thm:lip_sand_thm} (cf. Theorem 3.1 in \cite{LM24})
to extend the \textbf{HOLGRIM} algorithm to the setting in which
the $\Lip(\gamma)$ functions $f_1 , \ldots , f_{\n}$
are defined throughout an arbitrary non-empty compact subset 
$\Omega \subset V$ rather than a finite subset, and strengthen 
the sense in which the returned approximation $u$ approximates 
$\vph$ throughout $\Omega$ to a $\Lip(\eta)$ sense for a
given $\eta \in (0,\gamma)$.

To make this precise, let $\n \in \Z_{\geq 1}$, 
$\gamma > 0$ with $k \in \Z_{\geq 0}$ such that 
$\gamma \in (k,k+1]$, 
$\eta \in (0,\gamma)$ with $q \in \{0, \ldots , k\}$ such that 
$\eta \in (q,q+1]$, and
suppose that $\Omega \subset V$ is a fixed 
non-empty compact subset with
$f_1 , \ldots , f_{\n} \in \Lip (\gamma,\Omega,W)$.
As in Theorem \ref{thm:HOLGRIM_conv_thm} we consider 
$\vph \in \Lip(\gamma,\Omega,W)$ defined by 
$\vph := \sum_{i=1}^{\n} a_i f_i$ for non-zero coefficients
$a_1 , \ldots , a_{\n} \in \R$.
Making the choices 
$A_1 := ||f_1||_{\Lip(\gamma,\Omega,W)} , \ldots , 
A_{\n} := ||f_{\n}||_{\Lip(\gamma,\Omega,W)}$ 
in \eqref{eq:HOLGRIM_conv_thm_vph_C_D}, take 
$C := \sum_{i=1}^{\n} |a_i| ||f_i||_{\Lip(\gamma,\Omega,W)} > 0$,
and take $D = D(d,k)$ to be as defined in 
\eqref{eq:HOLGRIM_conv_thm_D}.
We now illustrate how the \textbf{HOLGRIM} algorithm can 
be used to obtain an approximation $u$ of $\vph$ throughout 
the compact subset $\Omega \subset V$ in the sense that 
the $\Lip(\eta,\Omega,W)$ norm of the difference $\vph - u$
is small.

Fix a choice of $\ep \in (0,C)$. 
Retrieve the constants 
$\de = \de(C,\gamma,\eta,\ep) > 0$ and 
$\ep_0 = \ep_0(C,\gamma,\eta,\ep) > 0$ 
arising as $\de_0$ and $\ep_0$ respectively in the
\emph{Lipschitz Sandwich Theorem} \ref{thm:lip_sand_thm} 
for the constants $K_0$, $\gamma$, $\eta$, and $\ep$ there
as $C$, $\gamma$, $\eta$, and $\ep$ here respectively. 
Note we are not actually applying Theorem \ref{thm:lip_sand_thm}, 
but simply retrieving constants in preparation for its future 
application.
Subsequently define an integer 
$\Lambda = \Lambda (\Omega,C,\gamma,\eta,\ep) \in \Z_{\geq 1}$ 
as the $\de$-covering number of $\Omega$ in $V$; that is
\beq
    \label{eq:HOLGRIM_conv_thm_rmk_lip_eta_reduct_Lambda}
        \Lambda := \min \left\{ a \in \Z_{\geq 1} 
        ~:~ \exists ~ z_1 , \ldots , z_a \in \Omega 
        \text{ for which } 
        \Omega \subset \bigcup_{s=1}^a \ovB_V(z_s,\de) 
        \right\}.
\eeq 
The subset $\Omega \subset V$ being compact ensures that 
$\Lambda$ defined in 
\eqref{eq:HOLGRIM_conv_thm_rmk_lip_eta_reduct_Lambda} 
is finite. Choose  
$\Sigma = \left\{ p_1 , \ldots , p_{\Lambda} \right\} 
\subset \Omega \subset V$ such that 
$\Omega \subset \cup_{s=1}^{\Lambda} \ovB_V(p_s,\de)$.

With the finite subset $\Sigma \subset \Omega \subset V$ fixed,
we now define $r = r(C,\gamma,\eta,\ep) > 0$ by 
\beq
    \label{eq:HOLGRIM_conv_thm_rmk_lip_eta_r}
        r := \sup \left\{ \lambda > 0 ~:~ 
        2C \lambda^{\gamma - k} + 
        \frac{\ep_0}{2cd^k} e^{\lambda} 
        \leq
        \frac{\ep_0}{cd^k} \right\}.
\eeq 
Subsequently define an integer
$N = N(\Sigma,C,\gamma,\eta,\ep) \in \Z_{\geq 1}$ by
\beq
    \label{eq:HOLGRIM_conv_thm_rmk_lip_eta_N}
        N := \max \left\{ s \in \Z ~:~
        \exists ~ z_1 , \ldots , z_s \in \Sigma 
        \text{ for which } ||z_a - z_b||_V > r 
        \text{ if } a \neq b \right\}.
\eeq 
Consider applying the \textbf{HOLGRIM}
algorithm to approximate $\vph$ on $\Sigma$, with 
$M := \min \left\{ \frac{\n - 1}{cD} , \Lambda \right\}$,
$\ep_0$ as the target accuracy in \textbf{HOLGRIM} \ref{HOLGRIM_A}, 
$\ep_0 / 2 c d^k$ as the acceptable recombination error in 
\textbf{HOLGRIM} \ref{HOLGRIM_A}, $k$ as  
the order level in \textbf{HOLGRIM} \ref{HOLGRIM_A}, 
$s_1 = \ldots = s_M = 1$
the shuffle numbers in \textbf{HOLGRIM} \ref{HOLGRIM_A}, 
$k_1 = \ldots = k_M = 1$ as the integers
$k_1 , \ldots , k_M \in \Z_{\geq 1}$ in 
\textbf{HOLGRIM} \ref{HOLGRIM_A}, and 
$A_1 := ||f_1||_{\Lip(\gamma,\Omega,W)} , \ldots , 
A_{\n} := ||f_{\n}||_{\Lip(\gamma,\Omega,W)}$ as the 
scaling factors chosen in \textbf{HOLGRIM} \ref{HOLGRIM_A}.

If the integer $N$ defined in 
\eqref{eq:HOLGRIM_conv_thm_rmk_compact_N} satisfies that 
$N \leq M$, then
Theorem \ref{thm:HOLGRIM_conv_thm} guarantees that 
there is some integer $m \in \{1, \ldots , N\}$
for which the \textbf{HOLGRIM} algorithm terminates after 
completing $m$ steps. Thus, if $Q_m := 1+mcD(d,k)$, 
there are coefficients 
$c_{1} , \ldots , c_{Q_m} \in \R$ and 
indices $e(1) , \ldots , e(Q_m) \in 
\{ 1 , \ldots , \n \}$ with 
(cf. \eqref{eq:HOLGRIM_conv_thm_coeff_sum})
\beq
    \label{eq:HOLGRIM_conv_thm_rmk_lip_eta_coeff_sum}
        \sum_{s=1}^{Q_m} \left|c_{s}\right|
        \left|\left| f_{e(s)}
        \right|\right|_{\Lip(\gamma,\Omega,W)}
        = C,
\eeq
and such that $u \in \Span(\cf) \subset 
\Lip(\gamma,\Omega,W)$ defined by
(cf. \eqref{eq:HOLGRIM_conv_thm_approx})
\beq
    \label{eq:HOLGRIM_conv_thm_rmk_lip_eta_approx}
        u := \sum_{s=1}^{Q_m}
        c_{s} f_{e(s)}
        ~\left(
        \begin{array}{c}
            u^{(l)} := \sum_{s=1}^{Q_m} c_{s}
            f^{(l)}_{e(s)} \\
            \text{for every }
            l \in \{0, \ldots , k\} 
        \end{array}
        \right)
        \quad \text{satisfies} \quad
        \max_{z \in \Sigma} \left\{
        \Lambda^k_{\vph - u}(z) \right\}
        \leq \ep_0.
\eeq
Observe that the condition 
\eqref{eq:HOLGRIM_conv_thm_rmk_lip_eta_coeff_sum} means that 
$||u||_{\Lip(\gamma,\Omega,W)} \leq C$.
Recall that the finite subset $\Sigma \subset \Omega$ was 
chosen to be a $\de$-cover of $\Omega$ in the sense that 
$\Omega \subset \cup_{z \in \Sigma} \ovB_V(z,\de)$.
Therefore, recalling how the constants $\de$ and $\ep_0$ here
were retrieved from the \emph{Lipschtz Sandwich Theorem} 
\ref{thm:lip_sand_thm}, we have the required hypotheses 
in order to appeal to the \emph{Lipschitz Sandwich Theorem} 
\ref{thm:lip_sand_thm} for the constants 
$K_0$, $\gamma$, $\eta$, and $\ep$ there as 
$C$, $\gamma$, $\eta$, and $\ep$ here respectively, 
the subsets $B$ and $\Sigma$ there are $\Sigma$ and $\Omega$
here respectively, and the functions $\psi$ and $\phi$ there
as $\vph$ and $u$ here respectively.
By doing so, we may conclude via the implication 
\eqref{eq:lip_sand_thm_imp} that 
\beq
    \label{eq:HOLGRIM_conv_thm_rmk_lip_eta_approx_B}
        \left|\left| \vph_{[q]} - u_{[q]} 
        \right|\right|_{\Lip(\eta,\Omega,W)} 
        \leq \ep
\eeq
where we use the notation that if 
$\phi = \left( \phi^{(0)} , \ldots , \phi^{(k)} \right) 
\in \Lip(\gamma,\Omega,W)$ then 
$\phi_{[q]} := \left( \phi^{(0)} , \ldots , \phi^{(q)} \right)$.
Consequently the \textbf{HOLGRIM} algorithm has returned an 
approximation $u \in \Span(\cf) \subset \Lip(\gamma,\Omega,W)$
of $\vph$ that has the sparsity properties as outlined in 
Remark \ref{rmk:HOLGRIM_conv_thm_1} and is
within $\ep$ of $\vph$ in the $\Lip(\eta,\Omega,W)$ norm 
sense detailed in 
\eqref{eq:HOLGRIM_conv_thm_rmk_lip_eta_approx_B}.
That is, in this case the \textbf{HOLGRIM} algorithm 
is guaranteed to return an approximation $u$ of $\vph$ satisfying
\textbf{Compact Domain Lipschitz Approximation Conditions} 
\ref{compact_lip_eta_approx_prop_1}, 
\ref{compact_lip_eta_approx_prop_2}, and 
\ref{compact_lip_eta_approx_prop_3} from Section 
\ref{sec:compact_domains}.
\end{remark}

\begin{remark}
\label{rmk:HOLGRIM_conv_thm_Omega_n0_lip_eta}
Following the approach outlined in Remark 
\ref{rmk:HOLGRIM_conv_thm_reverse_imp}, we illustrate the 
theoretical guarantees available for the \textbf{HOLGRIM} 
algorithm when used as in Remark
\ref{rmk:HOLGRIM_conv_thm_Omega_lip_eta}
but with an additional restriction on the maximum 
number of functions from $\cf$ that may be used to construct 
the approximation $u$. 

To make this precise, let $\n \in \Z_{\geq 1}$, 
$\gamma > 0$ with $k \in \Z_{\geq 0}$ such that 
$\gamma \in (k,k+1]$, $\eta \in (0,\gamma)$ with 
$q \in \{0 , \ldots , k\}$ such that 
$\eta \in (q,q+1]$, and
suppose that $\Omega \subset V$ is a fixed 
non-empty compact subset with
$f_1 , \ldots , f_{\n} \in \Lip (\gamma,\Omega,W)$.
As in Theorem \ref{thm:HOLGRIM_conv_thm} we consider 
$\vph \in \Lip(\gamma,\Omega,W)$ defined by 
$\vph := \sum_{i=1}^{\n} a_i f_i$ for non-zero coefficients
$a_1 , \ldots , a_{\n} \in \R$.
Let $C := \sum_{i=1}^{\n} |a_i| ||f_i||_{\Lip(\gamma,\Omega,W)} > 0$,
and take $D = D(d,k)$ to be as defined in 
\eqref{eq:HOLGRIM_conv_thm_D}.
Assuming that $\n > 1 + c D$, fix 
a choice of integer $n_0 \in \{1, \ldots , \n\}$ such that
$n_0 \geq 1 + cD$.
We consider the guarantees available for the quantity 
$||\vph_{[q]} - u_{[q]} ||_{\Lip(\eta,\Omega,W)}$ 
for the approximation returned by the \textbf{HOLGRIM} algorithm
with the restriction that $u$ is a linear combination of at most
$n_0$ of the functions in $\cf$. Here we use the notation that if 
$\phi = \left( \phi^{(0)} , \ldots , \phi^{(k)} \right) 
\in \Lip(\gamma,\Omega,W)$ then 
$\phi_{[q]} := \left( \phi^{(0)} , \ldots , \phi^{(q)} \right)$.
We will see that the stronger sense in which $u$ 
approximates $\vph$ throughout $\Omega$ comes at a cost of 
less explicit constants.

Fix $\th \geq 0$ and define 
$\lambda_0 = \lambda_0 (n_0,c,d,\gamma) > 0$ and 
$\be_0 = \be_0 (n_0,C,c,d,\gamma,\th) > 0$ by
\beq
    \label{eq:HOLGRIM_conv_thm_Omega_n0_rmk_lambda0_be0_lip_eta}
        \lambda_0 := \inf \left\{ t > 0 ~:~ 
        N_{\pack}(\Omega,V,t) \leq \frac{n_0 - 1}{cD} \right\}
        > 0
        \qquad \text{and} \qquad
        \be_0 := cd^k \left( 2C \lambda_0^{\gamma - k} 
        + \th e^{\lambda_0} \right) > 0.
\eeq 
It follows from 
\eqref{eq:HOLGRIM_conv_thm_Omega_n0_rmk_lambda0_be0_lip_eta}
that $0 \leq \th < \be_0 / cd^k$.
With the value of $\be_0 > 0$ fixed, determine a positive
constant $\xi = \xi(n_0,C,c,d,\gamma,\eta,\th) > 0$ by 
\beq
    \label{eq:HOLGRIM_conv_thm_Omega_n0_rmk_xi_lip_eta}
        \xi := \inf \left\{ \ep > 0 ~:~ \be_0 \leq 
        \ep_0(C,\gamma,\eta,\ep) \right\}
\eeq 
where, given $\ep > 0$,  
$\ep_0 = \ep_0(C,\gamma,\eta,\ep) > 0$ denotes the constant 
$\ep_0$ retrieved from the \emph{Lipschitz Sandwich Theorem} 
\ref{thm:lip_sand_thm} for the choices of the constants 
$K_0$, $\gamma$, $\eta$, and $\ep$ there as 
$C$, $\gamma$, $\eta$, and $\ep$ here respectively.
After the value of $\xi > 0$ is fixed, retrieve the constant 
$\de = \de(n_0,C,c,d,\gamma,\eta,\th) > 0$ arising as the constant
$\de_0$ in the \emph{Lipschitz Sandwich Theorem} 
\ref{thm:lip_sand_thm} for the choices of the constants 
$K_0$, $\gamma$, $\eta$, and $\ep$ there as 
$C$, $\gamma$, $\eta$, and $\xi$ here respectively.
Finally, define an integer 
$\Lambda = \Lambda (\Omega,n_0,C,c,d,\gamma,\th,q) 
\in \Z_{\geq 1}$ by
\beq
    \label{eq:HOLGRIM_conv_thm_Omega_n0_rmk_Lambda_n0_lip_eta}
        \Lambda := N_{\cov}(\Omega,V,\de).
\eeq 
Since the subset $\Omega \subset V$ is compact, the integer
$\Lambda$ defined in 
\eqref{eq:HOLGRIM_conv_thm_example_rmk_Lambda_n0_lip_eta} 
is finite. Choose a subset 
$\Sigma = \{ p_1 , \ldots , p_{\Lambda} \} \subset \Omega$
for which $\Omega \subset \cup_{s=1}^{\Lambda} \ovB_V(p_s,\de)$.
Moreover, since $N_{\pack}(\Sigma,V,\lambda_0) \leq 
N_{\pack}(\Omega,V,\lambda_0)$, 
\eqref{eq:HOLGRIM_conv_thm_Omega_n0_rmk_lambda0_be0_lip_eta}
ensures that 
$N_{\pack}(\Sigma,V,\lambda_0) \leq \frac{n_0-1}{cD}$.

Apply the \textbf{HOLGRIM} algorithm to 
approximate $\vph$ on $\Sigma$ with 
$M := \min \left\{ \frac{n_0 - 1}{cD} , \Lambda \right\}$,
$\be_0$ as the target accuracy in \textbf{HOLGRIM} \ref{HOLGRIM_A}, 
$\th$ as the acceptable recombination error in 
\textbf{HOLGRIM} \ref{HOLGRIM_A}, $k$ as  
the order level in \textbf{HOLGRIM} \ref{HOLGRIM_A}, 
$s_1 = \ldots = s_M = 1$
the shuffle numbers in \textbf{HOLGRIM} \ref{HOLGRIM_A}, 
$k_1 = \ldots = k_M = 1$ as the integers
$k_1 , \ldots , k_M \in \Z_{\geq 1}$ in 
\textbf{HOLGRIM} \ref{HOLGRIM_A}, and 
$A_1 := ||f_1||_{\Lip(\gamma,\Omega,W)} , \ldots , 
A_{\n} := ||f_{\n}||_{\Lip(\gamma,\Omega,W)}$ as the 
scaling factors chosen in \textbf{HOLGRIM} \ref{HOLGRIM_A}.

Observe that $\lambda_0$ defined in 
\eqref{eq:HOLGRIM_conv_thm_Omega_n0_rmk_lambda0_be0_lip_eta}
coincides with the positive constant $r > 0$ arising in 
\eqref{eq:HOLGRIM_conv_thm_r} in the 
\emph{HOLGRIM Convergence Theorem} \ref{thm:HOLGRIM_conv_thm}
for the constants $C$, $\gamma$, $\ep$, $\ep_0$, and $q$ there
as $C$, $\gamma$, $\be_0$, $\th$ and $k$ here. Thus, since
$N_{\pack}(\Sigma,V,\lambda_0) \leq \frac{n_0-1}{cD}$,
the \emph{HOLGRIM Convergence Theorem} \ref{thm:HOLGRIM_conv_thm} 
tells us that there is an integer $m \in \{1, \ldots , M\}$
for which the \textbf{HOLGRIM} algorithm terminates after 
completing $m$ steps. Thus, if $Q_m := 1+mcD(d,k)$, 
there are coefficients 
$c_{1} , \ldots , c_{Q_m} \in \R$ and 
indices $e(1) , \ldots , e(Q_m) \in 
\{ 1 , \ldots , \n \}$ with 
(cf. \eqref{eq:HOLGRIM_conv_thm_coeff_sum})
\beq
    \label{eq:HOLGRIM_conv_thm_Omega_n0_rmk_coeff_sum_n0_lip_eta}
        \sum_{s=1}^{Q_m} \left|c_{s}\right|
        \left|\left| f_{e(s)}
        \right|\right|_{\Lip(\gamma,\Omega,W)}
        = C,
\eeq
and such that $u \in \Span(\cf) \subset 
\Lip(\gamma,\Omega,W)$ defined by
(cf. \eqref{eq:HOLGRIM_conv_thm_approx})
\beq
    \label{eq:HOLGRIM_conv_thm_Omega_n0_rmk_u_n0_lip_eta}
        u := \sum_{s=1}^{Q_m}
        c_{s} f_{e(s)}
        ~\left(
        \begin{array}{c}
            u^{(l)} := \sum_{s=1}^{Q_m} c_{s}
            f^{(l)}_{e(s)} \\
            \text{for every }
            l \in \{0, \ldots , k\} 
        \end{array}
        \right)
        \quad \text{satisfies} \quad
        \max_{z \in \Sigma} \left\{
        \Lambda^k_{\vph - u}(z) \right\}
        \leq \be_0.
\eeq
The condition 
\eqref{eq:HOLGRIM_conv_thm_Omega_n0_rmk_coeff_sum_n0_lip_eta} 
means that $||u||_{\Lip(\gamma,\Omega,W)} \leq C$.
Recall that the finite subset $\Sigma \subset \Omega$ was 
chosen to be a $\de$-cover of $\Omega$ in the sense that 
$\Omega \subset \cup_{z \in \Sigma} \ovB_V(z,\de)$.
Further recall that $\de$ is the constant $\de_0$ arising 
in the \emph{Lipschtz Sandwich Theorem} 
\ref{thm:lip_sand_thm} for the choices of the constants
$K_0$, $\gamma$, $\eta$, and $\ep$ there as 
$C$, $\gamma$, $\eta$, and $\xi$ here.
Finally, recall that $\xi > 0$ was chosen in 
\eqref{eq:HOLGRIM_conv_thm_example_rmk_xi_lip_eta} 
so that $\be_0 \in [0,\ep_0]$ where $\ep_0$ is the constant
$\ep_0$ arising in the \emph{Lipschtz Sandwich Theorem} 
\ref{thm:lip_sand_thm} for the choices of the constants
$K_0$, $\gamma$, $\eta$, and $\ep$ there as 
$C$, $\gamma$, $\eta$, and $\xi$ here.

Therefore, we have the required hypotheses to appeal to the 
\emph{Lipschitz Sandwich Theorem} 
\ref{thm:lip_sand_thm} for the constants 
$K_0$, $\gamma$, $\eta$, and $\ep$ there as 
$C$, $\gamma$, $\eta$, and $\xi$ here respectively, 
the subsets $B$ and $\Sigma$ there are $\Sigma$ and $\Omega$
here respectively, and the functions $\psi$ and $\phi$ there
as $\vph$ and $u$ here respectively.
By doing so, we may conclude via the implication 
\eqref{eq:lip_sand_thm_imp} that 
\beq
    \label{eq:HOLGRIM_conv_thm_Omega_n0_rmk_u_[0,1]_n0_lip_eta}
        \left|\left| \vph_{[q]} - u_{[q]} 
        \right|\right|_{\Lip(\eta,\Omega,W)} 
        \leq \xi
\eeq
where we use the notation that if 
$\phi = \left( \phi^{(0)} , \ldots , \phi^{(k)} \right) 
\in \Lip(\gamma,\Omega,W)$ then 
$\phi_{[q]} := \left( \phi^{(0)} , \ldots , \phi^{(q)} \right)$.

Moreover, we have that $Q_m \leq 1 + cMD \leq n_0$ so that
$u$ is guaranteed to be a linear combination of no greater than 
$n_0$ of the functions in $\cf$.
Thus given an integer $n_0 \in \{1, \ldots , \n\}$ with 
$n_0 \geq 1 + c D(d,k)$, we are 
guaranteed that the \textbf{HOLGRIM} algorithm can find 
an approximation $u \in \Span(\cf)$ of $\vph$ that is a linear
combination of no more than $n_0$ of the functions in $\cf$, 
and is within $\xi$, for $\xi$ defined in 
\eqref{eq:HOLGRIM_conv_thm_Omega_n0_rmk_xi_lip_eta},
of $\vph$ throughout $\Omega$ in the $\Lip(\eta,\Omega,W)$ 
norm sense detailed 
in \eqref{eq:HOLGRIM_conv_thm_Omega_n0_rmk_u_[0,1]_n0_lip_eta}.
\end{remark}

\begin{remark}
\label{rmk:HOLGRIM_conv_thm_example}
We now provide an illustrative example of several of the 
previous remarks regarding the \emph{HOLGRIM Convergence Theorem}
\ref{thm:HOLGRIM_conv_thm}.
For this purpose consider a fixed $\gamma > 0$ with 
$k \in \Z_{\geq 0}$ such that $\gamma \in (k,k+1]$, 
a fixed $\eta \in (0,\gamma)$ with $q \in \{0, \ldots , k\}$
such that $\eta \in (q,q+1]$, a fixed $\ep > 0$, and a 
fixed $d \in \Z_{\geq 1}$.
Take $V := \R^d$ equipped with its usual Euclidean norm 
$||\cdot||_2$, take $\Omega := [0,1]^d \subset \R^d$ to be 
the unit cube in $\R^d$, and take $W := \R$.
The norm $||\cdot||_2$ is induced by the standard Euclidean 
dot product $\left< \cdot , \cdot \right>_{\R^d}$ on $\R^d$. 
Equip the tensor powers of $\R^d$ with admissible norms in the 
sense of Definition \ref{admissible_tensor_norm} by 
extending the inner product $\left< \cdot , \cdot \right>_{\R^d}$
to the tensor powers, and subsequently taking the norm induced 
by the resulting inner product on each tensor power 
(see, for example, Section 2 in \cite{LM24} for details of this 
construction).
Introduce the notation, for $x \in \R^d$ and $r > 0$, that 
$\B^d(x,r) := \left\{ y \in \R^d : ||x-y||_2 < r \right\}$.

Let $\rho \in (0,1)$. Then we have that 
\beq
    \label{eq:pack&cov_num_ub}
        N_{\cov} \left( [0,1]^d , \R^d , \rho \right)
        \leq
        N_{\pack} \left( [0,1]^d , \R^d , \rho \right)
        \leq
        \frac{2^d}{\omega_d} \left( 1 + \frac{1}{\rho} \right)^d
\eeq
where $\omega_d$ denotes the Euclidean volume of the unit-ball 
$\B^d(0,1) \subset \R^d$.
The first inequality in \eqref{eq:pack&cov_num_ub}
is a consequence of the fact that the 
$\rho$-covering number of a set is always no greater than 
its $\rho$-packing number.
The second inequality in \eqref{eq:pack&cov_num_ub} 
is a consequence of a volume comparison argument, 
the details of which may be found, for example, in 
Remark 4.4 in \cite{LM24}.

Fix $\n \in \Z_{\geq 1}$ and consider a collection 
$\cf = \{ f_1 , \ldots , f_{\n} \} \subset \Lip(\gamma,[0,1]^d,\R)$
of non-zero functions in $\Lip(\gamma,[0,1]^d,\R)$.
Fix a choice of non-zero coefficients 
$a_1 , \ldots , a_{\n} \in \R$ and define 
$\vph \in \Lip(\gamma,[0,1]^d,\R)$ by 
$\vph := \sum_{i=1}^{\n} a_i f_i$.
Set $C := \sum_{i=1}^{\n} |a_i| 
||f_i||_{\Lip(\gamma,[0,1]^d,\R)} > 0$
and $D = D(d,k)$ to be the constant defined in 
\eqref{eq:HOLGRIM_conv_thm_D}.

We first consider the use of the \textbf{HOLGRIM} algorithm to 
find an approximation $u \in \Span(\cf)$ of $\vph$ 
satisfying that $\sup_{x \in [0,1]^d } \left\{ 
\Lambda^q_{\vph - u}(x) \right\} \leq \ep$ by following 
the strategy detailed in Remark 
\ref{rmk:HOLGRIM_conv_thm_Omega_ptwise}.
Define a positive constant 
$\de = \de(C,\gamma,\ep,q) > 0$ and an integer
$\Lambda = \Lambda (C,d,\gamma,\ep,q) \in \Z_{\geq 1}$
by (cf. \eqref{eq:HOLGRIM_conv_thm_rmk_r_Omega})
\beq
    \label{eq:HOLGRIM_conv_thm_example_rmk_de_Lambda_ptwise}
        \de := \sup \left\{ \lambda > 0 ~:~ 
        2C \lambda^{\gamma - q} + \frac{\ep}{2} e^{\lambda}
        \leq \ep \right\} > 0
        \quad \text{and} \quad
        \Lambda := \max \left\{ m \in \Z ~:~ m \leq 
        \frac{2^d}{\omega_d} \left( 1 + \frac{1}{\de} 
        \right)^d \right\}.
\eeq 
A consequence of \eqref{eq:pack&cov_num_ub} and 
\eqref{eq:HOLGRIM_conv_thm_example_rmk_de_Lambda_ptwise} is that
$N_{\cov}([0,1]^d,\R^d,\de) \leq \Lambda$, and so 
we may choose a subset 
$\Sigma = \{ p_1 , \ldots , p_{\Lambda} \} 
\subset [0,1]^d$ for which 
$[0,1]^d \subset \cup_{s=1}^{\Lambda} \ovB^d(p_s,\de)$.

Now fix a choice of $\ep_0 \in [0 , \ep/2d^k)$ and define 
a positive constant $r = r(C,d,\gamma,\ep,\ep_0) > 0$ and an integer 
$N = N(C,d,\gamma,\ep,\ep_0) \in \Z_{\geq 1}$ by 
(cf. \eqref{eq:HOLGRIM_conv_thm_rmk_compact_r})
\beq
    \label{eq:HOLGRIM_conv_thm_example_rmk_r_N_ptwise}
        r := \sup \left\{ \lambda > 0 : 
        2C \lambda^{\gamma - k} + \ep_0 e^{\lambda} 
        \leq \frac{\ep}{2 d^k} \right\} > 0 
        \quad \text{and} \quad 
        N := \max \left\{ m \in \Z ~:~ 
        m \leq 
        \frac{2^d}{\omega_d} \left( 1 + \frac{1}{r} 
        \right)^d \right\}.
\eeq 
It is then a consequence of \eqref{eq:pack&cov_num_ub} 
and \eqref{eq:HOLGRIM_conv_thm_example_rmk_r_N_ptwise} that
$N_{\pack}(\Sigma,\R^d,r) \leq N$.

Consider applying the \textbf{HOLGRIM}
algorithm to approximate $\vph$ on $\Sigma$, with 
$M := \min \left\{ \frac{\n - 1}{D} , \Lambda \right\}$,
$\ep/2$ as the target accuracy in \textbf{HOLGRIM} \ref{HOLGRIM_A}, 
$\ep_0$ as the acceptable recombination error in 
\textbf{HOLGRIM} \ref{HOLGRIM_A}, $k$ as  
the order level in \textbf{HOLGRIM} \ref{HOLGRIM_A}, 
$s_1 = \ldots = s_M = 1$
the shuffle numbers in \textbf{HOLGRIM} \ref{HOLGRIM_A}, 
$k_1 = \ldots = k_M = 1$ as the integers
$k_1 , \ldots , k_M \in \Z_{\geq 1}$ in 
\textbf{HOLGRIM} \ref{HOLGRIM_A}, and 
$A_1 := ||f_1||_{\Lip(\gamma,\Omega,W)} , \ldots , 
A_{\n} := ||f_{\n}||_{\Lip(\gamma,\Omega,W)}$ as the 
scaling factors chosen in \textbf{HOLGRIM} \ref{HOLGRIM_A}.

If the integer $N$ defined in 
\eqref{eq:HOLGRIM_conv_thm_example_rmk_r_N_ptwise} satisfies that 
$N \leq M$, then
Theorem \ref{thm:HOLGRIM_conv_thm} guarantees that 
there is some integer $m \in \{1, \ldots , M\}$
for which the \textbf{HOLGRIM} algorithm terminates after 
completing $m$ steps. Thus, if $Q_m := 1+mD(d,k)$, 
there are coefficients 
$c_{1} , \ldots , c_{Q_m} \in \R$ and 
indices $e(1) , \ldots , e(Q_m) \in \{ 1 , \ldots , \n \}$ with 
(cf. \eqref{eq:HOLGRIM_conv_thm_coeff_sum})
\beq
    \label{eq:HOLGRIM_conv_thm_example_rmk_coeff_sum_ptwise}
        \sum_{s=1}^{Q_m} \left|c_{s}\right|
        \left|\left| f_{e(s)}
        \right|\right|_{\Lip(\gamma,[0,1]^d,\R)}
        = C,
\eeq
and such that $u \in \Span(\cf) \subset 
\Lip(\gamma,[0,1]^d,\R)$ defined by
(cf. \eqref{eq:HOLGRIM_conv_thm_approx})
\beq
    \label{eq:HOLGRIM_conv_thm_example_rmk_u_ptwise}
        u := \sum_{s=1}^{Q_m}
        c_{s} f_{e(s)}
        ~\left(
        \begin{array}{c}
            u^{(l)} := \sum_{s=1}^{Q_m} c_{s}
            f^{(l)}_{e(s)} \\
            \text{for every }
            l \in \{0, \ldots , k\} 
        \end{array}
        \right)
        \quad \text{satisfies} \quad
        \max_{z \in \Sigma} \left\{
        \Lambda^k_{\vph - u}(z) \right\}
        \leq \frac{\ep}{2}.
\eeq
Observe that the condition 
\eqref{eq:HOLGRIM_conv_thm_example_rmk_coeff_sum_ptwise} 
means that $||u||_{\Lip(\gamma,[0,1]^d,\R)} \leq C$.
Recall that the finite subset $\Sigma \subset [0,1]^d$ was 
chosen to be a $\de$-cover of $[0,1]^d$ in the sense that 
$[0,1]^d \subset \cup_{x \in \Sigma} \ovB^d(x,\de)$.
Therefore the definition of $\de$ in 
\eqref{eq:HOLGRIM_conv_thm_example_rmk_de_Lambda_ptwise} and
the estimates in 
\eqref{eq:HOLGRIM_conv_thm_example_rmk_u_ptwise} provide
the hypotheses required to appeal to the 
\emph{Pointwise Lipschitz Sandwich Theorem} 
\ref{thm:pointwise_lip_sand_thm} with the 
$K_0$, $\gamma$, $\ep$, $\ep_0$, $l$, $\Sigma$, $B$, 
$\psi$, and $\phi$ of that result as 
$C$, $\gamma$, $\ep$, $\ep/2$, $q$, $[0,1]^d$, $\Sigma$, 
$\vph$, and $u$ here respectively.
Indeed the definition of $\de$ here in 
\eqref{eq:HOLGRIM_conv_thm_example_rmk_de_Lambda_ptwise} 
coincides with 
the definition of $\de_0$ in \eqref{eq:pointwise_lip_sand_thm_de0}
for the choices specified above.
Thus we may invoke the implication 
\eqref{eq:pointwise_lip_sand_thm_imp} to conclude that 
\beq
    \label{eq:HOLGRIM_conv_thm_example_rmk_u_[0,1]_ptwise}
        \sup_{x \in [0,1]^d} \left\{
        \Lambda^q_{\vph-u}(x) \right\}
        \leq \ep.
\eeq
Consequently the \textbf{HOLGRIM} algorithm has returned an 
approximation $u \in \Span(\cf) \subset \Lip(\gamma,[0,1]^d,\R)$
of $\vph$ that is
within $\ep$ of $\vph$ throughout $[0,1]^d$
in the pointwise sense that, for every point $x \in [0,1]^d$, 
we have $\Lambda^q_{\vph - u}(x) \leq \ep$.
Moreover, $u$ is a linear combination of $Q_m = 1 + m D(d,k)$ 
of the $\n$ functions in $\cf$.
Thus, if we are in the setting that $D \Lambda \leq \n - 1$, 
then \eqref{eq:HOLGRIM_conv_thm_example_rmk_de_Lambda_ptwise} 
yields the worst-case upper bound of 
\beq
    \label{eq:HOLGRIM_conv_thm_example_rmk_Qm_ub}
        Q_m \leq 1 + \Lambda D 
        \leq 1 + D + 
        \frac{2^dD}{\omega_d} 
        \left( 1 + \frac{1}{\de} \right)^d
\eeq
for the number of functions from $\n$ that are used to 
construct the approximation $u$ returned by the 
\textbf{HOLGRIM} algorithm.

We next consider the use of the \textbf{HOLGRIM} algorithm to 
find an approximation $u \in \Span(\cf)$ of $\vph$ 
satisfying that 
$||\vph_{[q]} - u_{[q]}||_{\Lip(\eta,[0,1]^d,\R)} \leq \ep$ 
by following the strategy detailed in Remark 
\ref{rmk:HOLGRIM_conv_thm_Omega_lip_eta}.
Retrieve the constants 
$\de = \de(C,\gamma,\eta,\ep) > 0$ and 
$\ep_0 = \ep_0(C,\gamma,\eta,\ep) > 0$ 
arising as $\de_0$ and $\ep_0$ respectively in the
\emph{Lipschitz Sandwich Theorem} \ref{thm:lip_sand_thm} 
for the constants $K_0$, $\gamma$, $\eta$, and $\ep$ there
as $C$, $\gamma$, $\eta$, and $\ep$ here respectively. 
Note we are not actually applying Theorem \ref{thm:lip_sand_thm}, 
but simply retrieving constants in preparation for its future 
application.
Subsequently define an integer 
$\Lambda = \Lambda (C,d,\gamma,\eta,\ep) \in \Z_{\geq 1}$ by 
\beq
    \label{eq:HOLGRIM_conv_thm_example_rmk_Lambda_lip_eta}
        \Lambda := \max \left\{ m \in \Z_{\geq 1} 
        ~:~ m \leq
        \frac{2^d}{\omega_d} \left( 1 + \frac{1}{\de} \right)^d 
        \right\}.
\eeq 
A consequence of \eqref{eq:pack&cov_num_ub} and 
\eqref{eq:HOLGRIM_conv_thm_example_rmk_Lambda_lip_eta} is
that $N_{\cov}([0,1]^d,\R^d,\de) \leq \Lambda$, and so 
we may choose a subset 
$\Sigma = \left\{ p_1 , \ldots , p_{\Lambda} \right\} 
\subset [0,1]^d$ for which
$[0,1]^d \subset \cup_{s=1}^{\Lambda} \ovB^d(p_s,\de)$.

With the finite subset $\Sigma \subset [0,1]^d$ fixed,
we now first define $r = r(C,\gamma,\eta,\ep) > 0$ by
\beq
    \label{eq:HOLGRIM_conv_thm_example_rmk_r_lip_eta}
        r := \sup \left\{ \lambda > 0 ~:~ 
        2C \lambda^{\gamma - k} + 
        \frac{\ep_0}{2cd^k} e^{\lambda} 
        \leq
        \frac{\ep_0}{cd^k} \right\} > 0 
\eeq 
and then an integer
$N = N(C,d,\gamma,\eta,\ep) \in \Z_{\geq 1}$ by
\beq
    \label{eq:HOLGRIM_conv_thm_example_rmk_N_lip_eta}
        N := \max \left\{ m \in \Z ~:~
        m \leq \frac{2^d}{\omega_d} 
        \left( 1 + \frac{1}{r} \right)^d \right\}.
\eeq 
It is then a consequence of \eqref{eq:pack&cov_num_ub} and 
\eqref{eq:HOLGRIM_conv_thm_example_rmk_N_lip_eta} that 
$N_{\pack}(\Sigma,\R^d,r) \leq N$.

Consider applying the \textbf{HOLGRIM}
algorithm to approximate $\vph$ on $\Sigma$, with 
$M := \min \left\{ \frac{\n - 1}{D} , \Lambda \right\}$,
$\ep_0$ as the target accuracy in \textbf{HOLGRIM} \ref{HOLGRIM_A}, 
$\ep_0 / 2 c d^k$ as the acceptable recombination error in 
\textbf{HOLGRIM} \ref{HOLGRIM_A}, $k$ as  
the order level in \textbf{HOLGRIM} \ref{HOLGRIM_A}, 
$s_1 = \ldots = s_M = 1$
the shuffle numbers in \textbf{HOLGRIM} \ref{HOLGRIM_A}, 
$k_1 = \ldots = k_M = 1$ as the integers
$k_1 , \ldots , k_M \in \Z_{\geq 1}$ in 
\textbf{HOLGRIM} \ref{HOLGRIM_A}, and 
$A_1 := ||f_1||_{\Lip(\gamma,\Omega,W)} , \ldots , 
A_{\n} := ||f_{\n}||_{\Lip(\gamma,\Omega,W)}$ as the 
scaling factors chosen in \textbf{HOLGRIM} \ref{HOLGRIM_A}.

If the integer $N$ defined in 
\eqref{eq:HOLGRIM_conv_thm_example_rmk_N_lip_eta} satisfies that 
$N \leq M$, then
Theorem \ref{thm:HOLGRIM_conv_thm} guarantees that 
there is some integer $m \in \{1, \ldots , M\}$
for which the \textbf{HOLGRIM} algorithm terminates after 
completing $m$ steps. Thus, if $Q_m := 1+mD(d,k)$, 
there are coefficients 
$c_{1} , \ldots , c_{Q_m} \in \R$ and 
indices $e(1) , \ldots , e(Q_m) \in 
\{ 1 , \ldots , \n \}$ with 
(cf. \eqref{eq:HOLGRIM_conv_thm_coeff_sum})
\beq
    \label{eq:HOLGRIM_conv_thm_example_rmk_coeff_sum_lip_eta}
        \sum_{s=1}^{Q_m} \left|c_{s}\right|
        \left|\left| f_{e(s)}
        \right|\right|_{\Lip(\gamma,[0,1]^d,\R)}
        = C,
\eeq
and such that $u \in \Span(\cf) \subset 
\Lip(\gamma,[0,1]^d,\R)$ defined by
(cf. \eqref{eq:HOLGRIM_conv_thm_approx})
\beq
    \label{eq:HOLGRIM_conv_thm_example_rmk_u_lip_eta}
        u := \sum_{s=1}^{Q_m}
        c_{s} f_{e(s)}
        ~\left(
        \begin{array}{c}
            u^{(l)} := \sum_{s=1}^{Q_m} c_{s}
            f^{(l)}_{e(s)} \\
            \text{for every }
            l \in \{0, \ldots , k\} 
        \end{array}
        \right)
        \quad \text{satisfies} \quad
        \max_{x \in \Sigma} \left\{
        \Lambda^k_{\vph - u}(x) \right\}
        \leq \ep_0.
\eeq
Observe that the condition 
\eqref{eq:HOLGRIM_conv_thm_example_rmk_coeff_sum_lip_eta} 
means that 
$||u||_{\Lip(\gamma,[0,1]^d,\R)} \leq C$.
Recall that the finite subset $\Sigma \subset [0,1]^d$ was 
chosen to be a $\de$-cover of $[0,1]^d$ in the sense that 
$[0,1]^d \subset \cup_{x \in \Sigma} \ovB^d(x,\de)$.
Therefore, recalling how the constants $\de$ and $\ep_0$ here
were retrieved from the \emph{Lipschtz Sandwich Theorem} 
\ref{thm:lip_sand_thm}, we have the required hypotheses 
in order to appeal to the \emph{Lipschitz Sandwich Theorem} 
\ref{thm:lip_sand_thm} for the constants 
$K_0$, $\gamma$, $\eta$, and $\ep$ there as 
$C$, $\gamma$, $\eta$, and $\ep$ here respectively, 
the subsets $B$ and $\Sigma$ there are $\Sigma$ and $[0,1]^d$
here respectively, and the functions $\psi$ and $\phi$ there
as $\vph$ and $u$ here respectively.
By doing so, we may conclude via the implication 
\eqref{eq:lip_sand_thm_imp} that 
\beq
    \label{eq:HOLGRIM_conv_thm_example_rmk_u_[0,1]_lip_eta}
        \left|\left| \vph_{[q]} - u_{[q]} 
        \right|\right|_{\Lip(\eta,\Omega,W)} 
        \leq \ep
\eeq
where we use the notation that if 
$\phi = \left( \phi^{(0)} , \ldots , \phi^{(k)} \right) 
\in \Lip(\gamma,[0,1]^d,\R)$ then 
$\phi_{[q]} := \left( \phi^{(0)} , \ldots , \phi^{(q)} \right)$.
Consequently the \textbf{HOLGRIM} algorithm has returned an 
approximation $u \in \Span(\cf) \subset \Lip(\gamma,[0,1]^d,\R)$
of $\vph$ that is within $\ep$ of $\vph$ in the 
$\Lip(\eta,[0,1]^d,\R)$ norm sense detailed in 
\eqref{eq:HOLGRIM_conv_thm_example_rmk_u_[0,1]_lip_eta}.
Moreover, $u$ is a linear combination of $Q_m = 1 + m D(d,k)$ 
of the $\n$ functions in $\cf$.
Thus, if we are in the setting that $D \Lambda \leq \n - 1$, 
then \eqref{eq:HOLGRIM_conv_thm_example_rmk_Lambda_lip_eta} 
yields the worst-case upper bound of 
\beq
    \label{eq:HOLGRIM_conv_thm_example_rmk_Qm_ub_lip_eta}
        Q_m \leq 1 + \Lambda D 
        \leq 1 + D + 
        \frac{2^dD}{\omega_d} 
        \left( 1 + \frac{1}{\de} \right)^d
\eeq
for the number of functions from $\n$ that are used to 
construct the approximation $u$ returned by the 
\textbf{HOLGRIM} algorithm.

Assume now that $\n > 1 + 2^d D / \omega_d$.
We fix an integer $n_0 \in \{1, \ldots , \n\}$ with 
$n_0 \geq 1 + 2^d D / \omega_d$ and 
consider how well the \textbf{HOLGRIM} algorithm can 
approximate $\vph$ throughout $[0,1]^d$ using at most 
$n_0$ functions from $\cf$ (cf. Remark 
\ref{rmk:HOLGRIM_conv_thm_reverse_imp}).
We first consider the guarantees available for the pointwise
quantity $\sup_{x \in [0,1]^d} \left\{ \Lambda^q_{\vph-u}(x) 
\right\}$ for the approximation $u$ returned by the 
\textbf{HOLGRIM} algorithm with the restriction that 
$u$ is a linear combination of at most $n_0$ of the 
functions in $\cf$. We follow the approach proposed in 
Remark \ref{rmk:HOLGRIM_conv_thm_Omega_n0_pointwise}.

For this purpose we fix $\th \geq 0$ and define 
$\lambda_0 = \lambda_0 (n_0,d,\gamma) > 0$ and 
$\be_0 = \be_0 (n_0,C,d,\gamma,\th) > 0$ by
\beq
    \label{eq:HOLGRIM_conv_thm_example_rmk_lambda0_be0_ptwise}
        \lambda_0 := \left( \frac{1}{2}\left( 
        \frac{n_0 - 1}{D}\omega_d \right)^{\frac{1}{d}} - 1 
        \right)^{-1}
        > 0
        \qquad \text{and} \qquad
        \be_0 := 2 d^k \left( 2C \lambda_0^{\gamma - k} 
        + \th e^{\lambda_0} \right) > 0.
\eeq 
With the value of $\be_0 > 0$ fixed, define a positive
constant $\de = \de(n_0,C,d,\gamma,\th,q) > 0$ and an integer 
$\Lambda = \Lambda (n_0,C,d,\gamma,\th,q) \in \Z_{\geq 1}$ by
\beq
    \label{eq:HOLGRIM_conv_thm_example_rmk_de_Lambda_n0_ptwise}
        \de := \sup \left\{ t > 0 ~:~ 2Ct^{\gamma - q} + 
        \frac{\be_0}{2}e^t \leq \be_0 \right\} > 0
        \quad \text{and} \quad 
        \Lambda := \max \left\{ m \in \Z ~:~ 
        m \leq \frac{2^d}{\omega_d} 
        \left( 1 + \frac{1}{\de} \right)^d \right\}.
\eeq 
A consequence of \eqref{eq:pack&cov_num_ub} and 
\eqref{eq:HOLGRIM_conv_thm_example_rmk_de_Lambda_n0_ptwise} 
is that $N_{\cov}([0,1]^d,\R^d,\de) \leq \Lambda$, and so 
we may choose a subset 
$\Sigma = \{ p_1 , \ldots , p_{\Lambda} \} \subset [0,1]^d$
for which $[0,1]^d \subset \cup_{s=1}^{\Lambda} \ovB^d(p_s,\de)$.
Moreover, 
\eqref{eq:HOLGRIM_conv_thm_example_rmk_lambda0_be0_ptwise}
ensures that
\beq
    \label{eq:HOLGRIM_conv_thm_example_rmk_lambda0_bd_ptwise}
        \frac{2^d}{\omega_d} \left( 1 + \frac{1}{\lambda_0} 
        \right)^{d} 
        \leq
        \frac{n_0 - 1}{D}
\eeq
so that, via \eqref{eq:pack&cov_num_ub}, we have that
$N_{\pack}(\Sigma,\R^d,\lambda_0) \leq \frac{n_0-1}{D}$.
Furthermore, it follows from 
\eqref{eq:HOLGRIM_conv_thm_example_rmk_lambda0_be0_ptwise}
that $0 \leq \th < \be_0 / 2 d^k$.

Therefore we may apply the \textbf{HOLGRIM} algorithm to 
approximate $\vph$ on $\Sigma$ with 
$M := \min \left\{ \frac{n_0 - 1}{D} , \Lambda \right\}$,
$\be_0/2$ as the target accuracy in \textbf{HOLGRIM} \ref{HOLGRIM_A}, 
$\th$ as the acceptable recombination error in 
\textbf{HOLGRIM} \ref{HOLGRIM_A}, $k$ as  
the order level in \textbf{HOLGRIM} \ref{HOLGRIM_A}, 
$s_1 = \ldots = s_M = 1$
the shuffle numbers in \textbf{HOLGRIM} \ref{HOLGRIM_A}, 
$k_1 = \ldots = k_M = 1$ as the integers
$k_1 , \ldots , k_M \in \Z_{\geq 1}$ in 
\textbf{HOLGRIM} \ref{HOLGRIM_A}, and 
$A_1 := ||f_1||_{\Lip(\gamma,\Omega,W)} , \ldots , 
A_{\n} := ||f_{\n}||_{\Lip(\gamma,\Omega,W)}$ as the 
scaling factors chosen in \textbf{HOLGRIM} \ref{HOLGRIM_A}.

Observe that $\lambda_0$ defined in 
\eqref{eq:HOLGRIM_conv_thm_example_rmk_lambda0_be0_ptwise}
coincides with the positive constant $r > 0$ arising in 
\eqref{eq:HOLGRIM_conv_thm_r} in the 
\emph{HOLGRIM Convergence Theorem} \ref{thm:HOLGRIM_conv_thm}
for the constants $C$, $\gamma$, $\ep$, $\ep_0$, and $q$ there
as $C$, $\gamma$, $\be_0/2$, $\th$ and $k$ here. Thus, since
$N_{\pack}(\Sigma,\R^d,\lambda_0) \leq \frac{n_0-1}{D}$,
the \emph{HOLGRIM Convergence Theorem} \ref{thm:HOLGRIM_conv_thm} 
tells us that there is an integer $m \in \{1, \ldots , M\}$
for which the \textbf{HOLGRIM} algorithm terminates after 
completing $m$ steps. Thus, if $Q_m := 1+mD(d,k)$, 
there are coefficients 
$c_{1} , \ldots , c_{Q_m} \in \R$ and 
indices $e(1) , \ldots , e(Q_m) \in 
\{ 1 , \ldots , \n \}$ with 
(cf. \eqref{eq:HOLGRIM_conv_thm_coeff_sum})
\beq
    \label{eq:HOLGRIM_conv_thm_example_rmk_coeff_sum_n0_ptwise}
        \sum_{s=1}^{Q_m} \left|c_{s}\right|
        \left|\left| f_{e(s)}
        \right|\right|_{\Lip(\gamma,[0,1]^d,\R)}
        = C,
\eeq
and such that $u \in \Span(\cf) \subset 
\Lip(\gamma,[0,1]^d,\R)$ defined by
(cf. \eqref{eq:HOLGRIM_conv_thm_approx})
\beq
    \label{eq:HOLGRIM_conv_thm_example_rmk_u_n0_ptwise}
        u := \sum_{s=1}^{Q_m}
        c_{s} f_{e(s)}
        ~\left(
        \begin{array}{c}
            u^{(l)} := \sum_{s=1}^{Q_m} c_{s}
            f^{(l)}_{e(s)} \\
            \text{for every }
            l \in \{0, \ldots , k\} 
        \end{array}
        \right)
        \quad \text{satisfies} \quad
        \max_{x \in \Sigma} \left\{
        \Lambda^k_{\vph - u}(x) \right\}
        \leq \frac{\be_0}{2}.
\eeq
The condition 
\eqref{eq:HOLGRIM_conv_thm_example_rmk_coeff_sum_n0_ptwise} 
means that $||u||_{\Lip(\gamma,[0,1]^d,\R)} \leq C$.
Recall that the finite subset $\Sigma \subset [0,1]^d$ was 
chosen to be a $\de$-cover of $[0,1]^d$ in the sense that 
$[0,1]^d \subset \cup_{x \in \Sigma} \ovB^d(x,\de)$.
Therefore the definition of $\de$ in 
\eqref{eq:HOLGRIM_conv_thm_example_rmk_de_Lambda_n0_ptwise} and
the estimates in 
\eqref{eq:HOLGRIM_conv_thm_example_rmk_u_n0_ptwise} provide
the hypotheses required to appeal to the 
\emph{Pointwise Lipschitz Sandwich Theorem} 
\ref{thm:pointwise_lip_sand_thm} with the 
$K_0$, $\gamma$, $\ep$, $\ep_0$, $l$, $\Sigma$, $B$, 
$\psi$, and $\phi$ of that result as 
$C$, $\gamma$, $\be_0$, $\be_0/2$, $q$, $[0,1]^d$, $\Sigma$, 
$\vph$, and $u$ here respectively.
Indeed the definition of $\de$ here in 
\eqref{eq:HOLGRIM_conv_thm_example_rmk_de_Lambda_n0_ptwise} 
coincides with 
the definition of $\de_0$ in \eqref{eq:pointwise_lip_sand_thm_de0}
for the choices specified above.
Thus we may invoke the implication 
\eqref{eq:pointwise_lip_sand_thm_imp} to conclude that 
\beq
    \label{eq:HOLGRIM_conv_thm_example_rmk_u_[0,1]_n0_ptwise}
        \sup_{x \in [0,1]^d} \left\{
        \Lambda^q_{\vph-u}(x) \right\}
        \leq \be_0.
\eeq
Moreover, we have that $Q_m \leq 1 + MD \leq n_0$ so that
$u$ is guaranteed to be a linear combination of no greater than 
$n_0$ of the functions in $\cf$.
Thus given an integer $n_0 \in \{1, \ldots , \n\}$ 
for which $n_0 \geq 1 + 2^d D / \omega_d$, we are 
guaranteed that the \textbf{HOLGRIM} algorithm can find 
an approximation $u \in \Span(\cf)$ of $\vph$ that is a linear
combination of no more than $n_0$ of the functions in $\cf$, 
and is within $\be_0$, for $\be_0$ defined in 
\eqref{eq:HOLGRIM_conv_thm_example_rmk_lambda0_be0_ptwise},
of $\vph$ throughout $[0,1]^d$ in the pointwise sense detailed 
in \eqref{eq:HOLGRIM_conv_thm_example_rmk_u_[0,1]_n0_ptwise}.

We next consider the guarantees available for the 
quantity $||\vph_{[q]} - u_{[q]}||_{\Lip(\eta,[0,1]^d,\R)}$ 
for the approximation $u$ returned by the 
\textbf{HOLGRIM} algorithm with the restriction that 
$u$ is a linear combination of, at most, $n_0$ of the 
functions in $\cf$. 
We follow the approach proposed in Remark 
\ref{rmk:HOLGRIM_conv_thm_Omega_n0_lip_eta}.
As observed in Remark \ref{rmk:HOLGRIM_conv_thm_Omega_n0_lip_eta},
the stronger sense in which $u$ approximates $\vph$ 
throughout $[0,1]^d$ comes at a cost of less explicit constants.

Fix $\th \geq 0$ and define 
$\lambda_0 = \lambda_0 (n_0,d,\gamma) > 0$ and 
$\be_0 = \be_0 (n_0,C,d,\gamma,\th) > 0$ by
\beq
    \label{eq:HOLGRIM_conv_thm_example_rmk_lambda0_be0_lip_eta}
        \lambda_0 := \left( \frac{1}{2}\left( 
        \frac{n_0 - 1}{D}\omega_d \right)^{\frac{1}{d}} - 1 
        \right)^{-1}
        > 0
        \qquad \text{and} \qquad
        \be_0 := d^k \left( 2C \lambda_0^{\gamma - k} 
        + \th e^{\lambda_0} \right) > 0.
\eeq 
It follows from 
\eqref{eq:HOLGRIM_conv_thm_example_rmk_lambda0_be0_lip_eta}
that $0 \leq \th < \be_0 / d^k$.
With the value of $\be_0 > 0$ fixed, determine a positive
constant $\xi = \xi(n_0,C,d,\gamma,\eta,\th) > 0$ by 
\beq
    \label{eq:HOLGRIM_conv_thm_example_rmk_xi_lip_eta}
        \xi := \inf \left\{ \ep > 0 ~:~ \be_0 \leq 
        \ep_0(C,\gamma,\eta,\ep) \right\}
\eeq 
where, given $\ep > 0$,  
$\ep_0 = \ep_0(C,\gamma,\eta,\ep) > 0$ denotes the constant 
$\ep_0$ retrieved from the \emph{Lipschitz Sandwich Theorem} 
\ref{thm:lip_sand_thm} for the choices of the constants 
$K_0$, $\gamma$, $\eta$, and $\ep$ there as 
$C$, $\gamma$, $\eta$, and $\ep$ here respectively.
After the value of $\xi > 0$ is fixed, retrieve the constant 
$\de = \de(n_0,C,d,\gamma,\eta,\th) > 0$ arising as the constant
$\de_0$ in the \emph{Lipschitz Sandwich Theorem} 
\ref{thm:lip_sand_thm} for the choices of the constants 
$K_0$, $\gamma$, $\eta$, and $\ep$ there as 
$C$, $\gamma$, $\eta$, and $\xi$ here respectively.
Finally, define an integer 
$\Lambda = \Lambda (n_0,C,d,\gamma,\th,q) \in \Z_{\geq 1}$ by
\beq
    \label{eq:HOLGRIM_conv_thm_example_rmk_Lambda_n0_lip_eta}
        \Lambda := \max \left\{ m \in \Z ~:~ 
        m \leq \frac{2^d}{\omega_d} 
        \left( 1 + \frac{1}{\de} \right)^d \right\}.
\eeq 
A consequence of \eqref{eq:pack&cov_num_ub} and 
\eqref{eq:HOLGRIM_conv_thm_example_rmk_Lambda_n0_lip_eta} 
is that $N_{\cov}([0,1]^d,\R^d,\de) \leq \Lambda$, and so 
we may choose a subset 
$\Sigma = \{ p_1 , \ldots , p_{\Lambda} \} \subset [0,1]^d$
for which $[0,1]^d \subset \cup_{s=1}^{\Lambda} \ovB^d(p_s,\de)$.
Moreover, 
\eqref{eq:HOLGRIM_conv_thm_example_rmk_lambda0_be0_lip_eta}
ensures that
\beq
    \label{eq:HOLGRIM_conv_thm_example_rmk_lambda0_bd_lip_eta}
        \frac{2^d}{\omega_d} \left( 1 + \frac{1}{\lambda_0} 
        \right)^{d} 
        \leq
        \frac{n_0 - 1}{D}
\eeq
so that, via \eqref{eq:pack&cov_num_ub}, we have that
$N_{\pack}(\Sigma,\R^d,\lambda_0) \leq \frac{n_0-1}{D}$.

Apply the \textbf{HOLGRIM} algorithm to 
approximate $\vph$ on $\Sigma$ with 
$M := \min \left\{ \frac{n_0 - 1}{D} , \Lambda \right\}$,
$\be_0$ as the target accuracy in \textbf{HOLGRIM} \ref{HOLGRIM_A}, 
$\th$ as the acceptable recombination error in 
\textbf{HOLGRIM} \ref{HOLGRIM_A}, $k$ as  
the order level in \textbf{HOLGRIM} \ref{HOLGRIM_A}, 
$s_1 = \ldots = s_M = 1$
the shuffle numbers in \textbf{HOLGRIM} \ref{HOLGRIM_A}, 
$k_1 = \ldots = k_M = 1$ as the integers
$k_1 , \ldots , k_M \in \Z_{\geq 1}$ in 
\textbf{HOLGRIM} \ref{HOLGRIM_A}, and 
$A_1 := ||f_1||_{\Lip(\gamma,\Omega,W)} , \ldots , 
A_{\n} := ||f_{\n}||_{\Lip(\gamma,\Omega,W)}$ as the 
scaling factors chosen in \textbf{HOLGRIM} \ref{HOLGRIM_A}.

Observe that $\lambda_0$ defined in 
\eqref{eq:HOLGRIM_conv_thm_example_rmk_lambda0_be0_lip_eta}
coincides with the positive constant $r > 0$ arising in 
\eqref{eq:HOLGRIM_conv_thm_r} in the 
\emph{HOLGRIM Convergence Theorem} \ref{thm:HOLGRIM_conv_thm}
for the constants $C$, $\gamma$, $\ep$, $\ep_0$, and $q$ there
as $C$, $\gamma$, $\be_0$, $\th$ and $k$ here. Thus, since
$N_{\pack}(\Sigma,\R^d,\lambda_0) \leq \frac{n_0-1}{D}$,
the \emph{HOLGRIM Convergence Theorem} \ref{thm:HOLGRIM_conv_thm} 
tells us that there is an integer $m \in \{1, \ldots , M\}$
for which the \textbf{HOLGRIM} algorithm terminates after 
completing $m$ steps. Thus, if $Q_m := 1+mD(d,k)$, 
there are coefficients 
$c_{1} , \ldots , c_{Q_m} \in \R$ and 
indices $e(1) , \ldots , e(Q_m) \in 
\{ 1 , \ldots , \n \}$ with 
(cf. \eqref{eq:HOLGRIM_conv_thm_coeff_sum})
\beq
    \label{eq:HOLGRIM_conv_thm_example_rmk_coeff_sum_n0_lip_eta}
        \sum_{s=1}^{Q_m} \left|c_{s}\right|
        \left|\left| f_{e(s)}
        \right|\right|_{\Lip(\gamma,[0,1]^d,\R)}
        = C,
\eeq
and such that $u \in \Span(\cf) \subset 
\Lip(\gamma,[0,1]^d,\R)$ defined by
(cf. \eqref{eq:HOLGRIM_conv_thm_approx})
\beq
    \label{eq:HOLGRIM_conv_thm_example_rmk_u_n0_lip_eta}
        u := \sum_{s=1}^{Q_m}
        c_{s} f_{e(s)}
        ~\left(
        \begin{array}{c}
            u^{(l)} := \sum_{s=1}^{Q_m} c_{s}
            f^{(l)}_{e(s)} \\
            \text{for every }
            l \in \{0, \ldots , k\} 
        \end{array}
        \right)
        \quad \text{satisfies} \quad
        \max_{x \in \Sigma} \left\{
        \Lambda^k_{\vph - u}(x) \right\}
        \leq \be_0.
\eeq
The condition 
\eqref{eq:HOLGRIM_conv_thm_example_rmk_coeff_sum_n0_lip_eta} 
means that $||u||_{\Lip(\gamma,[0,1]^d,\R)} \leq C$.
Recall that the finite subset $\Sigma \subset [0,1]^d$ was 
chosen to be a $\de$-cover of $[0,1]^d$ in the sense that 
$[0,1]^d \subset \cup_{x \in \Sigma} \ovB^d(x,\de)$.
Further recall that $\de$ is the constant $\de_0$ arising 
in the \emph{Lipschtz Sandwich Theorem} 
\ref{thm:lip_sand_thm} for the choices of the constants
$K_0$, $\gamma$, $\eta$, and $\ep$ there as 
$C$, $\gamma$, $\eta$, and $\xi$ here.
Finally, recall that $\xi > 0$ was chosen in 
\eqref{eq:HOLGRIM_conv_thm_example_rmk_xi_lip_eta} 
so that $\be_0 \in [0,\ep_0]$ where $\ep_0$ is the constant
$\ep_0$ arising in the \emph{Lipschtz Sandwich Theorem} 
\ref{thm:lip_sand_thm} for the choices of the constants
$K_0$, $\gamma$, $\eta$, and $\ep$ there as 
$C$, $\gamma$, $\eta$, and $\xi$ here.

Therefore, we have the required hypotheses to appeal to the 
\emph{Lipschitz Sandwich Theorem} 
\ref{thm:lip_sand_thm} for the constants 
$K_0$, $\gamma$, $\eta$, and $\ep$ there as 
$C$, $\gamma$, $\eta$, and $\xi$ here respectively, 
the subsets $B$ and $\Sigma$ there are $\Sigma$ and $[0,1]^d$
here respectively, and the functions $\psi$ and $\phi$ there
as $\vph$ and $u$ here respectively.
By doing so, we may conclude via the implication 
\eqref{eq:lip_sand_thm_imp} that 
\beq
    \label{eq:HOLGRIM_conv_thm_example_rmk_u_[0,1]_n0_lip_eta}
        \left|\left| \vph_{[q]} - u_{[q]} 
        \right|\right|_{\Lip(\eta,[0,1]^d,\R)} 
        \leq \xi
\eeq
where we use the notation that if 
$\phi = \left( \phi^{(0)} , \ldots , \phi^{(k)} \right) 
\in \Lip(\gamma,[0,1]^d,\R)$ then 
$\phi_{[q]} := \left( \phi^{(0)} , \ldots , \phi^{(q)} \right)$.

Moreover, we have that $Q_m \leq 1 + MD \leq n_0$ so that
$u$ is guaranteed to be a linear combination of no greater than 
$n_0$ of the functions in $\cf$.
Thus given an integer $n_0 \in \{1, \ldots , \n\}$ for 
which $n_0 \geq 1 + 2^d D / \omega_d$, we are 
guaranteed that the \textbf{HOLGRIM} algorithm can find 
an approximation $u \in \Span(\cf)$ of $\vph$ that is a linear
combination of no more than $n_0$ of the functions in $\cf$, 
and is within $\xi$, for $\xi$ defined in 
\eqref{eq:HOLGRIM_conv_thm_example_rmk_xi_lip_eta},
of $\vph$ throughout $[0,1]^d$ in the $\Lip(\eta,[0,1]^d,\R)$ 
norm sense detailed 
in \eqref{eq:HOLGRIM_conv_thm_example_rmk_u_[0,1]_n0_lip_eta}.
\end{remark}
\vskip 4pt
\noindent 
We end this section by providing a proof of the 
\emph{HOLGRIM Convergence Theorem} \ref{thm:HOLGRIM_conv_thm}.
Our strategy is to use the separation of the points 
selected during the \textbf{HOLGRIM} algorithm, established in
the \emph{Point Separation Lemma}
\ref{lemma:HOLGRIM_dist_between_interp_points}, to obtain 
the desired upper bound on the number of steps which the 
algorithm can completed before terminating.

\begin{proof}[Proof of Theorem
\ref{thm:HOLGRIM_conv_thm}]
Let $\n, \Lambda , c , d \in \Z_{\geq 1}$, $\gamma > 0$ with 
$k \in \Z_{\geq 0}$ such that $\gamma \in (k,k+1]$, and
fix a choice of $q \in \{0 , \ldots , k\}$.
Let $\ep , \ep_0$ be real numbers such that
$0 \leq \ep_0 < \frac{\ep}{c d^q}$.
Let $V$ and $W$ be finite dimensional real Banach 
spaces, of dimensions $d$ and $c$ respectively, 
with $\Sigma \subset V$ finite.
Assume that the tensor products of $V$ are all 
equipped with admissible norms (cf. Definition
\ref{admissible_tensor_norm}). 
Assume, for every $i \in \{1, \ldots , \n\}$, that
$f_i = \left( f_i^{(0)} , \ldots , 
f_i^{(k)} \right) \in \Lip(\gamma,\Sigma,W)$ is 
non-zero and define 
$\cf := \left\{ f_i ~:~ i \in \{1 , \ldots , \n\}
\right\} \subset \Lip(\gamma,\Sigma,W)$. 
Choose scalars $A_1 , \ldots , A_{\n} \in \R_{>0}$ such 
that, for every $i \in \{1, \ldots , \n\}$, we have 
$||f_i||_{\Lip(\gamma,\Sigma,W)} \leq A_i$.
Given $a_1 , \ldots , a_{\n} \in \R \setminus \{0\}$, 
define $\vph = \left( \vph^{(0)} , \ldots , \vph^{(k)}
\right) \in \Span(\cf) \subset \Lip(\gamma, \Sigma,W)$ 
and constants $C , D > 0$ by
(cf. \eqref{eq:HOLGRIM_conv_thm_vph_C_D})
\beq
    \label{eq:HOLGRIM_conv_thm_pf_vph_C_D}
        (\bI) \quad 
        \vph := \sum_{i=1}^{\n} a_i f_i,
        \quad \left(
        \begin{array}{c}
            \vph^{(l)} := \sum_{i=1}^{\n} a_i
            f^{(l)}_i \\
            \text{for every }
            l \in \{0, \ldots , k\} 
        \end{array}
        \right)
        \quad \text{and} \quad
        (\bII) \quad 
        C := \sum_{i=1}^{\n} |a_i| A_i > 0. 
\eeq
and (cf. \eqref{eq:HOLGRIM_conv_thm_D})
\beq
    \label{eq:HOLGRIM_conv_thm_pf_D}
        D = D(d,k) := 
        \sum_{s=0}^k \be(d,s) =
        \mathlarger{\mathlarger{\sum}}_{s=0}^k
        \left(
        \begin{array}{c}
            d + s - 1 \\
            s
        \end{array}
        \right).
\eeq
With a view to later applying the \textbf{HOLGRIM} algorithm
to approximate $\vph$ on $\Sigma$, for each 
$i \in \{1 , \ldots , \n\}$ let $\tilde{a}_i := |a_i|$
and $\tilde{f}_i$ be given by $f_i$ if $a_i > 0$ and 
$-f_i$ if $a_i < 0$.
Observe that for every $i \in \{1, \ldots , \n\}$
we have that $\left|\left| \tilde{f}_i 
\right|\right|_{\Lip(\gamma,\Sigma,W)}
= || f_i ||_{\Lip(\gamma,\Sigma,W)}$. Moreover, 
we also have that 
$\tilde{a}_1 , \ldots , \tilde{a}_{\n} > 0$ and that
$\vph = \sum_{i=1}^{\n} \tilde{a}_i \tilde{f}_i$.
Further, for each $i \in \{1, \ldots , \n\}$ we set 
$h_i := \tilde{f}_i / A_i$ and 
$\al_i := \tilde{a}_i A_i$ (cf. \textbf{HOLGRIM} \ref{HOLGRIM_B}).
A first consequence is that for every $i \in \{1, \ldots , \n\}$
we have that $||h_i||_{\Lip(\gamma,\Sigma,W)} \leq 1$.
Further consequences are that $C$ satisfies  
\beq
    \label{eq:HOLGRIM_conv_thm_pf_new_C}
        C = \sum_{i=1}^{\n} \left|a_i\right| A_i
        = \sum_{i=1}^{\n} \tilde{a}_i A_i
        = \sum_{i=1}^{\n} \al_i,
\eeq
and, for every $i \in \{1 , \ldots , \n\}$, that 
$\al_i h_i = \tilde{a}_i \tilde{f}_i = a_i f_i$. 
Thus the expansion for $\vph$ in (\bI) of
\eqref{eq:HOLGRIM_conv_thm_pf_vph_C_D} is equivalent to
\beq    
    \label{eq:HOLGRIM_conv_thm_pf_vph_alpha_expansion}
        \varphi = \sum_{i=1}^{\n} 
        \al_i h_i,
        \qquad \text{and hence} \qquad
        || \varphi ||_{\Lip(\gamma,\Sigma,W)}
        \leq
        \sum_{i=1}^{\n} \al_i 
        || h_i ||_{\Lip(\gamma,\Sigma,W)}
        \leq
        \sum_{i=1}^{\n} \al_i
        \stackrel{
        \eqref{eq:HOLGRIM_conv_thm_pf_new_C}}{=}
        C.
\eeq
Define a positive constant
$r = r(C,\gamma,\ep ,\ep_0,q) >0$ 
as in \eqref{eq:HOLGRIM_conv_thm_r}. That is, 
\beq
    \label{eq:HOLGRIM_conv_thm_pf_r}
        r := \sup \left\{ \lambda > 0 ~:~
        2C \lambda^{\gamma - q}
        +
        \ep_0 e^{\lambda} \leq \frac{\ep}{c d^q} \right\}. 
\eeq
Define
$N = N(\Sigma, C, \gamma, \ep, \ep_0,q) 
\in \Z_{\geq 0}$ to be the $r$-packing 
number of $\Sigma$ as in 
\eqref{eq:HOLGRIM_conv_thm_N}.
That is,
\beq
    \label{eq:HOLGRIM_conv_thm_pf_N}
            N :=
            \max \left\{ s \in \Z ~:~ 
                \exists ~
                z_1 , \ldots , z_s \in \Sigma 
                \text{ for which } 
                || z_a - z_b ||_V
                > r
                \text{ if } a \neq b
            \right\},
\eeq
Consider applying the \textbf{HOLGRIM}
algorithm to approximate $\vph$ on $\Sigma$, with 
$M := \min \left\{ \frac{\n - 1}{cD} , \Lambda \right\}$,
the target accuracy in \textbf{HOLGRIM} \ref{HOLGRIM_A}
is $\ep$, 
the acceptable recombination error in 
\textbf{HOLGRIM} \ref{HOLGRIM_A} is $\ep_0$, 
the order level in \textbf{HOLGRIM} \ref{HOLGRIM_A} is $q$, 
the shuffle numbers in \textbf{HOLGRIM} \ref{HOLGRIM_A} 
are $s_1 = \ldots = s_M = 1$, the integers 
$k_1 , \ldots , k_M \in \Z_{\geq 1}$ in 
\textbf{HOLGRIM} \ref{HOLGRIM_A} are 
$k_1 = \ldots = k_M = 1$, and the scaling factors in 
\textbf{HOLGRIM} \ref{HOLGRIM_A} as 
$A_1 , \ldots , A_{\n} > 0$.
Suppose that the integer $N$ in 
\eqref{eq:HOLGRIM_conv_thm_N} satisfies 
$N \leq M$.

We first prove that the algorithm terminates after completing,
at most, $N$ steps.
If $N =M$ then this is immediate since the \textbf{HOLGRIM}
algorithm is guaranteed to terminate after step $M$. 
Hence we only verify the claim in the non-trivial setting that 
$N < M$, in which case terminating after (at most) $N$ steps 
requires the termination criterion in \textbf{HOLGRIM}
\ref{HOLGRIM_D} to be triggered.

Let $z_1 \in \Sigma$ be the point chosen in the first step 
(cf. \textbf{HOLGRIM} \ref{HOLGRIM_C}).
Let $u_1 \in \Span(\cf)$ be the approximation found via 
the \textbf{HOLGRIM Recombination Step} satisfying, in 
particular, that $\Lambda^k_{\vph - u_1}(z_1) \leq \ep_0$.
Define $\Sigma_1 := \{z_1\} \subset \Sigma$.
We conclude via \eqref{eq:HOLGRIM_point_sep_general_est} 
in Lemma 
\ref{lemma:HOLGRIM_dist_between_interp_points} that, for 
every $z \in \ovB_V(z_1,r) \cap \Sigma$, we have 
$\Lambda^q_{\vph - u_1}(z) \leq \ep / c d^q$.

If $N = 1$ then \eqref{eq:HOLGRIM_conv_thm_pf_N} means that 
$\ovB_V(z_1,r) \cap \Sigma = \Sigma$.
Thus we have established, for every $z \in \Sigma$, that
$\Lambda^q_{\vph-u_1}(z) \leq \ep / c d^q$.
For each $z \in \Sigma$ it follows from 
\eqref{eq:imp_1_num_pts_lemma} in Lemma 
\ref{lemma:number_of_coeffs_for_point_value} 
that, for every $\sigma \in \Tau_{z,q}$, we have 
$|\sigma(\vph-u_1)| \leq \ep / c d^q$.
Recalling \textbf{HOLGRIM} \ref{HOLGRIM_D}, this means 
that the algorithm terminates before step $2$ is completed.
Hence the algorithm terminates after completing 
$N=1$ steps.

If $N \geq 2$ then we note that if the stopping criterion in 
\textbf{HOLGRIM} \ref{HOLGRIM_D} is triggered at the start 
of step $l \in \{2, \ldots , N\}$ then we evidently have that 
the algorithm has terminated after carrying out no more than 
$N$ steps.
Consequently, we need only deal with the case in which the 
algorithm reaches and carries out step $N$ without terminating.
In this case, we claim that the algorithm terminates after
completing step $N$, i.e. that the termination criterion at 
the start of step $N+1$ is triggered. 

Before proving this we fix some notation.
Recalling \textbf{HOLGRIM} \ref{HOLGRIM_D}, 
for $l \in \{2, \ldots , N\}$ let $z_l \in \Sigma$ denote
the point selected at step $l$, and let $u_l \in \Span(\cf)$
denote the approximation found via the 
\textbf{HOLGRIM Recombination Step} satisfying, for 
every $s \in \{1 , \ldots , l\}$, that 
$\Lambda^k_{\vph - u_l}(z_s) \leq \ep_0$.

By appealing to Lemma \ref{lemma:HOLGRIM_dist_between_interp_points}
we deduce both that, whenever $s,t \in \{1, \ldots , l\}$
with $s \neq t$, we have 
(cf. \eqref{eq:HOLGRIM_point_sep_dist_est})
\beq
    \label{eq:HOLGRIM_conv_thm_pf_interp_pts_sep}
        || z_t - z_s ||_V > r, 
\eeq
and that whenever $t \in \{1, \ldots ,l\}$ and 
$z \in \ovB_V(z_t,r) \cap \Sigma$ we have 
(cf. \eqref{eq:HOLGRIM_point_sep_general_est})
\beq
    \label{eq:HOLGRIM_conv_thm_pf_Lambda_b_vph-ul_bd_A}
        \Lambda^b_{\vph-u_l}(z) \leq \frac{\ep}{c d^q}.
\eeq
Consider step $N+1$ of the algorithm at which we 
examine the quantity 
$K := \max \left\{ |\sigma(\vph-u_N)| : 
\sigma \in \Sigma^{\ast}_q \right\}$.
If $K \leq \ep / c d^q$ then the algorithm terminates without
carrying out step $N+1$, and thus has terminated after carrying
out $N$ steps.

Assume that $K > \ep / cd^q$ so that 
$\sigma_{N+1} := \argmax \left\{|\sigma(\vph-u_N)| : 
\sigma \in \Sigma^{\ast}_q \right\}$ satisfies 
\beq
    \label{eq:HOLGRIM_conv_thm_pf_N+1_lin_func_lb}
        |\sigma_{N+1}(\vph-u_N)| > \frac{\ep}{c d^q}.
\eeq
Recall from the \textbf{HOLGRIM Extension Step} that 
$z_{N+1} \in \Sigma$ is taken to be the point for which 
$\sigma_{N+1} \in \Tau_{z_{N+1},q}$. 
It is then a consequence of Lemma \ref{lemma:dual_norm_ests} 
that (cf. \eqref{eq:dual_norm_ests_lemma_pointwise_conc})
\beq
    \label{eq:HOLGRIM_conv_thm_pf_sigm_j_eval_ub}
        \left| \sigma_{N+1} \left( \vph - u_N \right)
        \right| 
        \leq \Lambda^q_{\vph - u_N}(z_{N+1}).
\eeq
Thus we have that
\beq
    \label{eq:HOLGRIM_conv_thm_pf_pointwise_diff_lb}
        \frac{\ep}{c d^q} 
        \stackrel{
        \eqref{eq:HOLGRIM_conv_thm_pf_N+1_lin_func_lb}
        }{<}
        \left| \sigma_{N+1} (\vph - u_N) \right|
        \stackrel{
        \eqref{eq:HOLGRIM_conv_thm_pf_sigm_j_eval_ub}
        }{\leq}
        \Lambda^q_{\vph - u_N}(z_{N+1}).
\eeq
Together, \eqref{eq:HOLGRIM_conv_thm_pf_Lambda_b_vph-ul_bd_A} 
for $l := N$ and 
\eqref{eq:HOLGRIM_conv_thm_pf_pointwise_diff_lb}
yield that, for every $t \in \{1, \ldots , N\}$, we
have $z_{N+1} \notin \ovB_V(z_t,r)$; that is, 
$||z_{N+1} - z_t||_V > r$.

Therefore the points $z_1 , \ldots , z_{N+1} \in \Sigma$
satisfy that, whenever $s,t \in \{1, \ldots , N+1\}$
with $s \neq t$, we have $||z_s - z_t||_V > r$.
Hence
\beq
    \label{eq:HOLGRIM_conv_thm_pf_contradiction}
        N \stackrel{
        \eqref{eq:HOLGRIM_conv_thm_pf_N}
        }{=}
        \max \left\{ t \in \Z : \exists ~ 
        x_1 , \ldots , x_t \in \Sigma \text{ such that }
        ||x_s - x_t||_V > r \text{ whenever } s \neq t 
        \right\} 
        \geq N+1
\eeq
which is evidently a contradiction.
Thus we must in fact have that $K \leq \ep / c d^q$, 
and hence the algorithm must terminate before carrying out 
step $N+1$.

Having established the claimed upper bound for the maximum 
number of steps that the \textbf{HOLGRIM} algorithm 
can complete without terminating, we turn our attention 
to establishing the properties claimed for the approximation 
returned after the algorithm terminates.
Let $m \in \{1, \ldots , N\}$ be the integer for which the 
\textbf{HOLGRIM} algorithm terminates after step $m$. 
Let $u := u_m \in \Span(\cf)$. 
Then $u$ is the final approximation returned by the 
\textbf{HOLGRIM} algorithm.

If $m < \{1 , \ldots , M -1\}$ then the termination 
criterion of \textbf{HOLGRIM} \ref{HOLGRIM_D} is triggered
at the start of step $m+1$.
Consequently, we have that for every linear functional 
$\sigma \in \Sigma_q^{\ast}$ that 
$|\sigma(\vph-u)| \leq \ep / c d^q$.
Recall that $\Sigma^{\ast}_q := \cup_{z \in \Sigma} \Tau_{z,q}$.
Thus, for any $z \in \Sigma$, we must have 
that $|\sigma(\vph - u)| \leq \ep / c d^q$ for 
every $\sigma \in \Tau_{z,q}$.
Consequently, by appealing to Lemma 
\ref{lemma:number_of_coeffs_for_point_value} 
(specifically to the implication 
\eqref{eq:imp_2_num_pts_lemma}), 
we conclude that $\Lambda^q_{\vph - u}(z) \leq \ep$.
Since $z \in \Sigma$ was arbitrary we have established 
that $\max_{z \in \Sigma} \left\{ \Lambda^q_{\vph - u}(z) 
\right\} \leq \ep$.

If $m = M$ then, since $N \leq M$ and 
we have already established that the 
\textbf{HOLGRIM} terminates before completing step $k$ for 
any $k > N$, we must have that $N = M$. 
Consequently, in this case we have that $u = u_N$, and thus 
have previously established that for every linear functional 
$\sigma \in \Sigma_q^{\ast}$ we have 
$|\sigma(\vph-u)| \leq \ep / c d^q$.
And so, by repeating the argument of the paragraph above, 
we may once again conclude that 
$\max_{z \in \Sigma} \left\{ \Lambda^q_{\vph - u}(z) 
\right\}\leq \ep$.

The combination of the previous two paragraphs enables us 
to conclude that 
$\max_{z \in \Sigma} \left\{ \Lambda^q_{\vph - u}(z) 
\right\} \leq \ep$
as claimed in the second part of 
\eqref{eq:HOLGRIM_conv_thm_approx}.

Recalling \textbf{HOLGRIM} \ref{HOLGRIM_C} and \ref{HOLGRIM_D}, 
the \textbf{HOLGRIM Recombination Step} has found 
$u_m \in \Span(\cf)$ satisfying, for every 
$s \in \{1, \ldots ,m\}$,
that $\Lambda^k_{\vph - u_m}(z_s) \leq \ep_0$.
In particular, this is achieved via an application of Lemma 3.1 
in \cite{LM22} to a system of $Q_m := 1 + mcD(d,k)$ 
real-valued equations 
(cf. \eqref{lip_k_recomb_eqn_num_recomb_step} and 
\textbf{HOLGRIM Recombination Step} 
\ref{HOLGRIM_recomb_step_C}).
Therefore recombination returns non-negative coefficients 
$b_1 , \ldots , b_{Q_m} \geq 0$ with 
\beq
    \label{eq:HOLGRIM_conv_thm_pf_approx_coeff_sum}
        \sum_{s=1}^{Q_m} b_{s} 
        =
        \sum_{i=1}^{\n} \al_i
        \stackrel{
        \eqref{eq:HOLGRIM_conv_thm_pf_new_C} 
        }
        {=}
        C,
\eeq
and indices
$e(1) , \ldots , e(Q_m) \in \{1, \ldots , \n\}$ for which 
\beq
    \label{eq:HOLGRIM_conv_thm_pf_approx_u_expansion}
        u = \sum_{s=1}^{Q_m}
        b_{s} h_{e(s)}
        \quad \text{so that for every }
        j \in \{0, \ldots , n\} \quad 
        u^{(j)} = \sum_{s=1}^{Q_m} b_{s}
        h^{(j)}_{e(s)}.
\eeq
A consequence of \eqref{eq:HOLGRIM_conv_thm_pf_approx_u_expansion}
is that 
\beq
    \label{eq:HOLGRIM_conv_thm_pf_lin_comb_fi's}
        u = \sum_{s=1}^{Q_m} 
        b_{s} h_{e(s)}
        = \sum_{s=1}^{Q_m}
        \frac{b_{s}}{A_{e(s)}}
        \tilde{f}_{e(s)}.
\eeq
For each $s \in \{1 , \ldots , Q_m\}$,
define $c_{s} := \frac{b_{s}}{A_{e(s)}}$ if
$\tilde{f}_{e(s)} = f_{e(s)}$ (which, we recall, is
the case if $a_{e(s)} > 0$) and 
$c_{s} := -\frac{b_{s}}{A_{e(s)}}$ if
$\tilde{f}_{e(s)} = -f_{e(s)}$ (which, we recall, is
the case if $a_{e(s)} < 0$).
Then \eqref{eq:HOLGRIM_conv_thm_pf_lin_comb_fi's}
gives the expansion for 
$u \in \Lip(\gamma,\Sigma,W)$ in terms of
the functions $f_1 , \ldots , f_{\n}$
claimed in the first part of
\eqref{eq:HOLGRIM_conv_thm_approx}.
Moreover, from \eqref{eq:HOLGRIM_conv_thm_pf_approx_coeff_sum} 
we have that
\beq
    \label{eq:HOLGRIM_conv_thm_pf_approx_coeff_sum_got}
        \sum_{s=1}^{Q_m} \left| c_{s} \right|
        A_{e(s)}
        =
        \sum_{s=1}^{Q_m} b_{s}
        \stackrel{
        \eqref{eq:HOLGRIM_conv_thm_pf_approx_coeff_sum}
        }{=}
        C
\eeq
as claimed in \eqref{eq:HOLGRIM_conv_thm_coeff_sum}.

It remains only to prove that if the coefficients 
$a_1 , \ldots , a_{\n} \in \R \setminus \{0\}$ 
corresponding to $\vph$ 
(cf. (\bI) of \eqref{eq:HOLGRIM_conv_thm_pf_vph_C_D})
are all positive (i.e. $a_1 , \ldots , a_{\n} > 0$) then 
the coefficients $c_{1} , \ldots , c_{Q_m} \in \R$ 
corresponding to $u$ 
(cf. \eqref{eq:HOLGRIM_conv_thm_pf_lin_comb_fi's}) are 
all non-negative (i.e. $c_{1} , \ldots , c_{Q_m} \geq 0$).
First note that $a_1 , \ldots , a_{\n} > 0$ means, for 
every $i \in \{1 , \ldots , \n\}$, that 
$\tilde{f}_i = f_i$. Consequently, for every 
$s \in \{1 , \ldots , Q_m\}$, we have that
$\tilde{f}_{e(s)} = f_{e(s)}$, and so by definition $c_{s} = b_{s} / A_{e(s)}$.
Since $A_{e(s)} > 0$ and $b_{s} \geq 0$, it follows that $c_{s} \geq 0$.
This completes the proof of Theorem \ref{thm:HOLGRIM_conv_thm}.
\end{proof}

\vskip 4pt 
\noindent
University of Oxford, Radcliffe Observatory,
Andrew Wiles Building, Woodstock Rd, Oxford, 
OX2 6GG, UK.
\vskip 4pt
\noindent
TL: tlyons@maths.ox.ac.uk \\
\url{https://www.maths.ox.ac.uk/people/terry.lyons}
\vskip 4pt
\noindent
AM: andrew.mcleod@maths.ox.ac.uk \\
\url{https://www.maths.ox.ac.uk/people/andrew.mcleod}
\end{document}